\documentclass[12pt,a4paper,twoside]{article}
\usepackage[hmarginratio=1:1,bmargin=1.3in]{geometry}
\usepackage[frenchb]{babel}
\usepackage[utf8]{inputenc}
\usepackage[OT2,T1]{fontenc}
\usepackage{amsmath, amsthm}
\usepackage{amsfonts}
\usepackage{amssymb}
\usepackage[all]{xy}
\usepackage{graphicx}
\usepackage[nottoc, notlof, notlot]{tocbibind}
\usepackage{array}
\usepackage{url}
\usepackage{rotating}
\usepackage{fancyhdr}
\usepackage{fancyhdr}
\usepackage{titlesec}
\titleformat{\section}[hang]{\center\Large\bf}{\thesection.}{0.5cm}{}

\DeclareSymbolFont{cyrletters}{OT2}{wncyr}{m}{n}
\DeclareMathSymbol{\Sha}{\mathalpha}{cyrletters}{"58}
\DeclareMathSymbol{\Brusse}{\mathalpha}{cyrletters}{"42}
\addto\captionsfrenchb{}

\theoremstyle{plain}
\newtheorem{theorem}{Th\'eor\`eme}[section]
\newtheorem{lemma}[theorem]{Lemme}
\newtheorem{proposition}[theorem]{Proposition}
\newtheorem{corollary}[theorem]{Corollaire}
\newtheorem{definition}[theorem]{D\'efinition}

\theoremstyle{definition}
\newtheorem{remarque}[theorem]{Remarque}

\newtheorem{example}[theorem]{Exemple}
\newtheorem{notation}[theorem]{Notation}

\newtheoremstyle{hypo}  
  {\topsep}   
  {\topsep}   
  {\itshape}  
  {1.5ex}       
  {\bfseries} 
  {)}         
  {8pt plus 1pt minus 1pt}  
  {}          
\theoremstyle{hypo}
\newtheorem{hypo}[theorem]{(H}

\newtheoremstyle{hypol}  
  {\topsep}   
  {\topsep}   
  {\itshape}  
  {1.5ex}       
  {\bfseries} 
  {)$_{\ell}$}         
  {8pt plus 1pt minus 1pt}  
  {}          
\theoremstyle{hypol}
\newtheorem{hypol}[theorem]{(H}

\newtheoremstyle{DP}  
  {\topsep}   
  {\topsep}   
  {\itshape}  
  {}       
  {\bfseries} 
  {.}         
  {8pt plus 1pt minus 1pt}  
  {}          
\theoremstyle{DP}

\begin{document}

\title{\textbf{\scshape Variétés abéliennes sur les corps de fonctions de courbes sur des corps locaux supérieurs}}

\author{Diego Izquierdo\\
\small
École Normale Supérieure\\
\small
45, Rue d'Ulm - 75005 Paris - France\\
\small
\texttt{diego.izquierdo@ens.fr}
}
\date{}
\normalsize
\maketitle
\setcounter{section}{-1}

\section{\textsc{Introduction}}

\subsection{Contexte et motivations}

\hspace{4ex} Depuis les travaux de John Tate dans les années 1960, les théorèmes de dualité arithmétique pour la cohomologie galoisienne des groupes algébriques commutatifs sur des corps locaux ou globaux ont joué un rôle central en arithmétique. Les groupes algébriques concernés sont divers: les groupes finis, les tores, les variétés abéliennes. Rappelons brièvement les résultats portant sur ces dernières.\\

\hspace{4ex} Dans le cadre local, Tate montre en 1958, dans l'exposé \cite{Tat58}, qu'étant donnée une variété abélienne $A$ sur un corps $p$-adique $k$ de variété abélienne duale $A^t$, il existe un accouplement canonique $H^0(k,A) \times H^1(k,A^t) \rightarrow \text{Br}(k) \cong \mathbb{Q}/\mathbb{Z}$ qui met en dualité parfaite le groupe profini $H^0(k,A)$ et le groupe de torsion $H^1(k,A^t)$. Dans le cadre global, en généralisant des travaux de Cassels pour les courbes elliptiques, il construit pour chaque variété abélienne $A$ sur un corps de nombres $K$ de variété duale $A^t$ un accouplement $\Sha^1(K,A) \times \Sha^1(K,A^t) \rightarrow \mathbb{Q}/\mathbb{Z}$ puis annonce au Congrès International des Mathématiciens de 1962 (\cite{Tat63}) la non-dégénérescence de ce dernier modulo divisibles. Ici, $\Sha^1(K,A)$ désigne le groupe de Tate-Shafarevich constitué  des classes d'isomorphismes de torseurs sous $A$ triviales dans tous les complétés de $K$. Des résultats analogues ont aussi été établis pour les variétés abéliennes sur $\mathbb{F}_p((u))$ et sur $\mathbb{F}_p(u)$ (Remarque I.3.6 et Théorème I.6.13 de \cite{MilADT}).\\

\hspace{4ex} Par ailleurs, ces dernières années, nous avons été témoins d'un regain d'intérêt pour les théorèmes de dualité sur d'autres corps de caractéristique 0 que les corps $p$-adiques et les corps de nombres. Citons par exemple les travaux de Scheiderer et van Hamel (\cite{SVH}) et Harari et Szamuely (\cite{HS1}) pour $\mathbb{Q}_p((u))$ et $\mathbb{Q}_p(u)$, ceux de Colliot-Thélène et Harari (\cite{CTH}) pour $\mathbb{C}((t))((u))$ et $\mathbb{C}((t))(u)$, et ceux de l'auteur (\cite{Izq1}) pour les corps de la forme $k((u))$ et $k(u)$ avec $k=\mathbb{Q}_p((t_1))...((t_d))$ ou $k=\mathbb{C}((t_1))...((t_d))$. Cependant, aucun de ces travaux ne porte sur les variétés abéliennes. En fait, à la connaissance de l'auteur, on ne dispose jusqu'à présent que de résultats pour les variétés abéliennes sur $\mathbb{C}((u))$ et $\mathbb{C}(u)$: cela remonte à des travaux de Ogg dans les années 1960 (\cite{Ogg}). Le but du présent article est donc d'établir des théorèmes de dualité, analogues à ceux de Tate rappelés ci-dessus, pour les variétés abéliennes sur les corps de la forme $k((u))$ et $k(u)$ avec $k=\mathbb{Q}_p$ ou $k=\mathbb{C}((t))$, voire avec $k=\mathbb{Q}_p((t_1))...((t_d))$ ou $k=\mathbb{C}((t_1))...((t_d))$. 

\begin{remarque}
Il est vrai qu'on peut déjà trouver des résultats similaires pour les variétés abéliennes sur $\mathbb{Q}_p((t_1))...((t_d))$ dans \cite{Koy}, mais l'article en question contient un grand nombre d'erreurs et soit le théorème principal soit la principale proposition permettant de le prouver semble erroné (voir la remarque \ref{faux}).
\end{remarque}

\subsection{Organisation de l'article}

\hspace{4ex} Cet article est constitué de 6 sections.\\

\hspace{4ex} La première partie permet de faire quelques rappels et d'établir quelques résultats préliminaires. On y étudie notamment la cohomologie des tores sur des corps (dits $d$-locaux) de la forme $\mathbb{C}((t_0))...((t_d))$ ou $\mathbb{Q}_p((t_2))...((t_d))$.\\

\hspace{4ex} La deuxième partie porte sur les variétés abéliennes sur $\mathbb{C}((t_0))((t_1))$ et sur $\mathbb{C}((t_0))(t)$. Voici les principaux résultats obtenus:

\begin{theorem} (théorèmes \ref{1-local}, \ref{noyau} et \ref{C((t))(X)} et corollaire \ref{corglob}) \label{thintro}
\begin{itemize}
\item[(i)] Soient $k = \mathbb{C}((t_0))((t_1))$ et $A$ une variété abélienne sur $k$. Soit $A^t$ sa variété abélienne duale. Les groupes $H^1(k,A)$ et $(H^0(k,A^t)^{\wedge})^D$ sont isomorphes modulo divisibles. Plus précisément, on a une suite exacte:
$$0 \rightarrow (\mathbb{Q}/\mathbb{Z})^{m(A)} \rightarrow H^1(k,A) \rightarrow (H^0(k,A^t)^{\wedge})^D \rightarrow 0,$$
où $m(A)$ est un entier naturel compris entre 0 et $4\dim A$ dépendant de la géométrie de $A$. L'entier $m(A)$ est nul si, et seulement si, la variété abélienne sur $\mathbb{C}((t_0))$ apparaissant dans la réduction de $A$ modulo $t_1$ a très mauvaise réduction. Lorsque la fibre spéciale du modèle de Néron est connexe, le noyau de $H^1(k,A) \rightarrow (H^0(k,A^t)^{\wedge})^D$ est un groupe contenant $H^1_{nr}(k,A)$ qui peut être décrit explicitement: on le note $H^1_{nrs}(k,A)$.
\item[(ii)] Soient $k=\mathbb{C}((t_0))$ et $K$ le corps des fonctions d'une courbe projective lisse géométriquement intègre $X$ sur $k$. On note $X^{(1)}$ l'ensemble des points fermés de $X$. Soit $A$ une variété abélienne sur $K$, de variété abélienne duale $A^t$. Soit $Z$ l'ensemble des $v \in X^{(1)}$ tels que $m(A \times_K K_v)=0$. On suppose que, pour toute place $v \in X^{(1)}\setminus Z$, la fibre spéciale du modèle de Néron de $A \times_K K_v$ est connexe. Alors il existe une dualité parfaite:
$$\overline{\Sha^1(K,A)} \times \overline{\Sha^1_{nrs}(A^t)} \rightarrow \mathbb{Q}/\mathbb{Z}$$
ainsi qu'un accouplement $\Sha^1(K,A) \times \Sha^1(K,A^t) \rightarrow \mathbb{Q}/\mathbb{Z}$ dont le noyau à gauche (resp. à droite) est constitué des éléments de $\Sha^1(K,A)$ (resp. $\Sha^1(K,A^t)$) qui sont divisibles dans $\Sha^1_{nrs}(A)$ (resp. $\Sha^1_{nrs}(A^t)$).
Ici:
\begin{gather*}
\Sha^1(K,A) := \text{Ker}\left(H^1(K,A) \rightarrow \prod_{v \in X^{(1)}} H^1(K_v,A)\right),\\
\Sha^1_{nrs}(A^t) := \text{Ker}\left( H^1(K,A^t) \rightarrow \prod_{v \in X^{(1)}\setminus Z} H^1(K_v,A^t)/H^1_{nrs}(K_v,A^t) \times \prod_{v \in Z} H^1(K_v,A^t)\right),
\end{gather*}
 et, pour chaque groupe abélien $B$, $\overline{B}$ désigne le quotient de $B$ par son sous-groupe divisible maximal.
\end{itemize}
\end{theorem}

\begin{remarque}
L'énoncé précédent est moins général que les énoncés qui seront démontrés dans l'article: il est en fait possible d'affaiblir l'hypothèse de connexité des fibres spéciales des modèles de Néron. On remarquera aussi que l'hypothèse ne concerne que les places de mauvaise réduction.
\end{remarque}
\vspace{5pt}
\hspace{4ex} Les parties 3 et 5 sont consacrées à une généralisation de (i) du théorème \ref{thintro} aux variétés abéliennes sur $\mathbb{C}((t_0))...((t_d))$ ou $k'((t_2))...((t_d))$ avec $k'$ corps $p$-adique. Plus précisément, elles permettent de construire un accouplement entre la cohomologie d'une variété abélienne et la cohomologie d'un certain faisceau qui lui est associé puis:
\begin{itemize}
\item[$\bullet$] de démontrer que ledit accouplement induit toujours une dualité parfaite modulo divisibles (corollaires \ref{surj} et \ref{surjQp} et théorèmes \ref{modulo divisibles} et \ref{modulo divisibles Qp}),
\item[$\bullet$] de déterminer quand c'est un accouplement parfait (sans quotienter par les sous-groupes divisbles) (corollaire \ref{cormaj}, théorèmes \ref{nullité}, \ref{nullitéQp} et \ref{nullitéQpp2} et proposition \ref{nullitéQpp}),
\item[$\bullet$] de calculer dans certains cas les noyaux à gauche et à droite de l'accouplement (théorèmes \ref{noyaubis} et \ref{noyauQp}).
\end{itemize}
\vspace{10pt}

\hspace{4ex} Les parties 4 et 6 sont consacrées à une généralisation de (ii) du théorème \ref{thintro} aux variétés abéliennes sur $k(X)$ où $k=\mathbb{C}((t_0))...((t_d))$ ou $k=k'((t_2))...((t_d))$ avec $k'$ corps $p$-adique et $X$ est une courbe projective lisse géométriquement intègre sur $k$ (théorèmes \ref{cornrsbis}, \ref{th1} et \ref{cornrsQp} et corollaires \ref{corglobbis} et \ref{corglobQp}).\\

\hspace{4ex} Finalement, dans la septième partie, on s'intéresse à la finitude du premier groupe de Tate-Shafarevich d'une variété abélienne sur $k(X)$ où $k=\mathbb{C}((t_0))...((t_d))$ ou $k=k'((t_2))...((t_d))$ avec $k'$ corps $p$-adique et $X$ est une courbe projective lisse géométriquement intègre sur $k$.

\subsection{Remerciements} 

Je tiens à remercier en premier lieu David Harari pour son soutien et ses conseils, ainsi que sa lecture soigneuse de ce texte: sans lui, ce travail n'aurait pas pu voir le jour. Je suis aussi très reconnaissant à Jean-Louis Colliot-Thélène et à Tamás Szamuely pour leurs commentaires et leurs remarques. Je voudrais finalement remercier l'École Normale Supérieure pour ses excellentes conditions de travail.

\subsection{Notations}

\textbf{Corps.} Si $l$ est un corps, on notera $l^s$ sa clôture séparable. Si de plus $l$ est un corps de valuation discrète complet, on notera $l^{nr}$ son extension non ramifiée maximale.\\

\textbf{Groupes abéliens.} Pour $M$ un groupe topologique abélien (éventuellement muni de la topologie discrète), $n>0$ un entier et $\ell$ un nombre premier, on notera:
\begin{itemize}
\item[$\bullet$] $M_{tors}$ la partie de torsion de $M$.
\item[$\bullet$] ${_n}M$ la partie de $n$-torsion de $M$.
\item[$\bullet$] $M\{\ell\}$ la partie de torsion $\ell$-primaire de $M$.
\item[$\bullet$] $M_{\text{non}-\ell}=\bigoplus_{p \neq \ell} M\{p\}$ où $p$ décrit les nombres premiers différents de $\ell$.
\item[$\bullet$] $M^{(\ell)}$ le complété pour la topologie $\ell$-adique de $M$.
\item[$\bullet$] $M^{\wedge}$ la limite projective des $M/nM$.
\item[$\bullet$] $T_{\ell}M$ la limite projective des ${_{\ell^r}}M$.
\item[$\bullet$] $M_{div}$ le sous-groupe divisible maximal de $M$.
\item[$\bullet$] $\overline{M}=M/M_{div}$ le quotient de $M$ par son sous-groupe divisible maximal.
\item[$\bullet$] $M^D$ le groupe des morphismes continus $M \rightarrow \mathbb{Q}/\mathbb{Z}$.
\end{itemize}
Un groupe abélien de torsion sera dit de type cofini si, pour tout entier naturel $n>0$, sa $n$-torsion est finie. La partie $\ell$-primaire d'un tel groupe est la somme directe d'un $\ell$-groupe abélien fini et d'une puissance finie de $\mathbb{Q}_{\ell}/\mathbb{Z}_{\ell}$.
\vspace{5 mm}

\textbf{Faisceaux et cohomologie.} Sauf indication du contraire, tous les faisceaux sont considérés pour le petit site étale. Soit $r \geq 0$. Pour $F$ et $G$ deux faisceaux fppf sur un schéma $X$, on note $\underline{\text{Ext}}^r_X(F,G)$ (ou $\underline{\text{Ext}}^r(F,G)$ s'il n'y a pas d'ambigüité) le faisceau associé pour la topologie étale au préfaisceau $T \mapsto \text{Ext}^r_{T_{fppf}}(F,G)$. On rappelle qu'avec cette définition, la formule de Barsotti-Weil garantit que, si $A$ est une variété abélienne sur un corps $k$, alors  la variété abélienne duale $A^t$ représente le faisceau $\underline{\text{Ext}}^1_k(A,\mathbb{G}_m)$ (voir par exemple le théorème III.18.1 de \cite{Oort}). Par ailleurs, en mimant les notations pour les groupes abéliens, on pose, pour $F$ un faisceau sur un schéma $X$ et $l$ un nombre premier, $H^r(X,T_lF) = \varprojlim_n H^r(X,{_{l^n}}F)$ et $H^r(X,F\{l\}) = \varinjlim_n H^r(X,{_{l^n}}F)$.\\

\textbf{Catégories dérivées.} Nous serons amenés quelques fois à considérer des catégories dérivées. On notera alors $- \otimes^{\textbf{L}} -$ le produit tensoriel dérivé.\\

\textbf{Corps locaux supérieurs.} Les corps 0-locaux sont par définition les corps finis et le corps $\mathbb{C}((t))$. Pour $d \geq 1$, un corps $d$-local est un corps complet pour une valuation discrète dont le corps résiduel est $(d-1)$-local. \emph{On remarquera que cette définition est plus générale que la définition standard.} Lorsque $k$ est un corps $d$-local, on notera $k_0$, $k_1$, ..., $k_d$ les corps tels que $k_0$ est fini ou $\mathbb{C}((t))$, $k_d=k$, et pour chaque $i$ le corps $k_i$ est le corps résiduel de $k_{i+1}$. On rappelle le théorème de dualité sur un corps $d$-local $k$: pour tout $\text{Gal}(k^s/k)$-module fini $M$ d'ordre $n$ premier à $\text{Car}(k_1)$, on a un accouplement parfait de groupes finis $H^r(k,M) \times H^{d+1-r}(k,\text{Hom}(M,\mu_n^{\otimes d})) \rightarrow H^{d+1}(k,\mu_n^{\otimes d}) \cong \mathbb{Z}/n\mathbb{Z}$. Ce théorème est énoncé et démontré dans \cite{MilADT} (théorème 2.17) lorsque $k_0 \neq \mathbb{C}((t))$. Il se prouve exactement de la même manière dans ce dernier cas: en effet, il suffit de procéder par récurrence à l'aide du lemme 2.18 de \cite{MilADT}, l'initialisation étant réduite à la dualité évidente $H^r(k_{-1},M) \times H^{-r}(k_{-1},\text{Hom}(M,\mathbb{Z}/n\mathbb{Z})) \rightarrow H^0(k_{-1},\mathbb{Z}/n\mathbb{Z})\cong\mathbb{Z}/n\mathbb{Z}$ pour le corps ``$-1$-local'' $k_{-1} = \mathbb{C}$.\\

\textbf{Groupes de Tate-Shafarevich.} Lorsque $L$ est le corps des fonctions d'une variété projective lisse géométriquement intègre $Y$ sur un corps $l$ et $M$ est un $\text{Gal}(L^s/L)$-module discret, le $r$-ième groupe de Tate-Shafarevich de $M$ est, par définition, le groupe $\Sha^r(L,M) = \text{Ker}(H^r(L,M) \rightarrow \prod_{v \in Y^{(1)}} H^r(L_v,M))$.\\

\textbf{Groupe de Brauer.} Lorsque $Z$ est un schéma, on note $\text{Br}(Z)$ le groupe de Brauer cohomologique $H^2(Z,\mathbb{G}_m)$. Si $Z$ est une $l$-variété géométriquement intègre pour un certain corps $l$, on note $\text{Br}_{1}(Z)$ le groupe de Brauer algébrique $\text{Ker}(\text{Br}(Z) \rightarrow \text{Br}(Z \times_l l^s))$. \\

\textbf{Cadre.} Dans toute la suite, $d$ désignera un entier naturel fixé (éventuellement nul), $k$ un corps $d$-local et $X$ une courbe projective lisse géométriquement intègre sur $k$. On notera $X^{(1)}$ l'ensemble de ses points de codimension 1 et $K$ son corps des fonctions. \emph{Lorsque $k_0$ est fini, on supposera que le corps $k_1$ est de caractéristique 0}: autrement dit, ou bien $k_0 = \mathbb{C}((t))$, ou bien $d\geq 1$ et $k_1$ est un corps $p$-adique. Lorsque $M$ est un $\text{Gal}(K^s/K)$-module discret, on notera parfois $\Sha^r(M)$ au lieu de $\Sha^r(K,M)$.\\

\textbf{Cohomologie à support compact.} Pour $j: U \hookrightarrow X$ une immersion ouverte et $\mathcal{F}$ un faisceau sur $U$, le $r$-ième groupe de cohomologie à support compact est, par définition, le groupe $H^r_c(U,\mathcal{F}) = H^r(X,j_!\mathcal{F})$. On remarquera que, contrairement à la définition classique de la cohomologie à support compact, cette définition ne dépend pas du choix d'une compactification lisse de $U$ (plus précisément, nous avons choisi $X$ comme compactification lisse de $U$).\\

\textbf{Tores algébriques.} On dit qu'un groupe algébrique $T$ sur un corps $l$ est un tore si $T \times_l l^s$ est isomorphe à $\mathbb{G}_m^r$ pour un certain $r \geq 0$. On rappelle que le foncteur $T \mapsto \hat {T}=\text{Hom}(T \times l^s,\mathbb{G}_{m,l^s})$ établit une équivalence de catégories entre les tores algébriques sur $l$ et les $\text{Gal}(l^s/l)$-modules qui, en tant que groupes abéliens, sont libres de type fini. Si $T$ est un tore sur $l$, on appelle rang de $T$ la dimension d'un sous-tore déployé maximal: c'est aussi le rang de $H^0(l,\hat{T})$. Par ailleurs, on dit qu'un tore $T$ sur $l$ est quasi-trivial s'il est isomorphe à $R_{L/l}\mathbb{G}_m$ pour une certaine $l$-algèbre finie séparable $L$.

\section{\textsc{Préliminaires}}

\subsection{Groupes de torsion de type cofini}

Dans cet article, nous étudierons souvent la structure des groupes de cohomologie par des arguments de comptage. Ainsi, les lemmes qui suivent, qui ne sont que des exercices d'algèbre élémentaire, seront utilisés très souvent sans référence explicite:

\begin{lemma}
Soient $A$ et $A'$ deux groupes de torsion de type cofini. Si $|{_n}A|=|{_n}A'|$ pour tout entier naturel $n$, alors $A$ et $A'$ sont (non canoniquement) isomorphes.
\end{lemma}

\begin{lemma}
Soit $A$ un groupe de torsion de type cofini. Pour chaque nombre premier $\ell$, il existe un groupe fini $F_{\ell}$ et un entier naturel $r_{\ell}$ tels que $A\{\ell\} \cong F_{\ell} \oplus (\mathbb{Q}_{\ell}/\mathbb{Z}_{\ell})^{r_{\ell}}$. De plus, pour tout entier naturel $n$, on a:
$$\frac{|{_n}A|}{|A/n|}=\prod_{\ell} \ell^{r_{\ell}v_{\ell}(n)}.$$
Ici, $v_{\ell}(n)$ désigne la valuation $\ell$-adique de $n$.
\end{lemma}

\subsection{Caractéristique d'Euler-Poincaré}

\begin{proposition}
Soit $F$ un $\text{Gal}(k^s/k)$-module fini. 
\begin{itemize}
\item[(i)] Si $k_0$ est fini et $d\geq 2$, alors $\chi (k,F) = 1$.
\item[(ii)] Si $k_0 =\mathbb{C}((t))$, alors $\chi (k,F) = 1$.
\end{itemize}
\end{proposition}

\begin{proof}
\begin{itemize}
\item[(i)] Notons $\kappa$ le corps résiduel de $k$ et procédons par récurrence sur $d$.
\begin{itemize}
\item[$\bullet$] Supposons que $d=2$. On dispose de la suite spectrale $H^r(\kappa,H^s(k^{nr},F)) \Rightarrow H^{r+s}(k,F)$, qui dégénère en une suite exacte longue:
$$... \rightarrow H^r(\kappa,H^0(k^{nr},F)) \rightarrow H^r(k,F) \rightarrow H^{r-1}(\kappa,H^1(k^{nr},F)) \rightarrow ...$$
On a donc $\chi (k,F) = \frac{\chi (\kappa,H^0(k^{nr},F))}{\chi (\kappa,H^1(k^{nr},F))}$. Or $H^0(k^{nr},F)$ et $H^1(k^{nr},F)$ ont même cardinal puisque $\text{Gal}(k^s/k^{nr}) \cong \hat{\mathbb{Z}}$ (cela découle aisément de la proposition 1.7.7(i) de \cite{CNF}). Par conséquent, d'après le théorème 2.8 de \cite{MilADT}, on a $\chi (\kappa,H^0(k^{nr},F)) = \chi (\kappa,H^1(k^{nr},F))$, et donc $\chi (k,F) = 1$.
\item[$\bullet$] Soit $d>2$ et supposons que la proposition soit vraie pour tout corps $(d-1)$-local. Comme avant, la suite spectrale $H^r(\kappa,H^s(k^{nr},F)) \Rightarrow H^{r+s}(k,F)$ dégénère en une suite exacte longue:
$$... \rightarrow H^r(\kappa,H^0(k^{nr},F)) \rightarrow H^r(k,F) \rightarrow H^{r-1}(\kappa,H^1(k^{nr},F)) \rightarrow ...$$
On a donc $\chi (k,F) = \frac{\chi (\kappa,H^0(k^{nr},F))}{\chi (\kappa,H^1(k^{nr},F))}$. Par hypothèse de récurrence, on a $\chi (\kappa,H^0(k^{nr},F)) = \chi (\kappa,H^1(k^{nr},F))=1$, et donc $\chi (k,F) = 1$.
\end{itemize}
\item[(ii)] L'énoncé est vrai pour $d=0$ puisque $\text{Gal}(\mathbb{C}((t))^s/\mathbb{C}((t))) \cong \hat{\mathbb{Z}}$. On procède ensuite par récurrence comme dans (i).
\end{itemize}
\end{proof}

\subsection{Tores algébriques}

Pour commencer, on rappelle le lemme d'Ono, qui sera utilisé à de nombreuses reprises dans la suite:

\begin{lemma} \textbf{(Lemme d'Ono - théorème 1.5.1 de \cite{Ono})}\\
Soient $l$ un corps et $T$ un tore sur $l$. Il existe un entier naturel non nul $m$, des $l$-tores quasi-triviaux $T_0$ et $R$ et un $l$-schéma en groupes fini commutatif $F$ tels que l'on ait une suite exacte:
$$0 \rightarrow F \rightarrow R \rightarrow T^m \times_l T_0 \rightarrow 0.$$
\end{lemma}

\begin{proposition}\label{tor0}
Supposons que $k_0=\mathbb{C}((t))$. Soient $T$ un tore sur $k$ et $\rho$ son rang, c'est-à-dire la dimension d'un sous-tore déployé maximal. On a alors pour chaque entier naturel naturel $n$:
$$\frac{|{_n}T(k)|}{|T(k)/n|}=n^{-d\rho}.$$ 
\end{proposition}

\begin{proof}
\begin{itemize}
\item[$\bullet$] Montrons d'abord la proposition pour $T=\mathbb{G}_m$. On remarque que:
$$\frac{|{_n}\mathbb{G}_m(k)|}{|\mathbb{G}_m(k)/n|}=\frac{|H^0(k,\mu_n)|}{|H^1(k,\mu_n)|}=\frac{n}{|H^1(k,\mu_n)|}.$$
En notant $\kappa$ le corps résiduel de $k$, on a une suite exacte:
$$0 \rightarrow H^1(\kappa,\mu_n) \rightarrow H^1(k,\mu_n) \rightarrow H^0(\kappa,\mathbb{Z}/n\mathbb{Z}) \rightarrow 0.$$
On en déduit que $|H^1(k,\mu_n)|=n|H^1(\kappa,\mu_n)|$. Par récurrence, cela montre que $|H^1(k,\mu_n)|=n^{d+1}$, et donc que la proposition est vraie pour $T=\mathbb{G}_m$. Cela entraîne bien sûr que la proposition est vraie pour tout tore quasi-trivial.
\item[$\bullet$] On se place maintenant dans le cas général. Soient $m$ un entier naturel non nul et $T_0$ un tore quasi-trivial sur $k$ tels que l'on a une suite exacte:
$$0 \rightarrow F \rightarrow R \rightarrow T^m \times_k T_0 \rightarrow 0$$
avec $F$ un schéma en groupes fini commutatif sur $k$ et $R$ un tore quasi-trivial sur $k$. Notons $\rho_R$ le rang de $R$. En passant à la cohomologie, on obtient une suite exacte:
$$0 \rightarrow F(k) \rightarrow R(k) \rightarrow T(k)^m\times T_0(k) \rightarrow H^1(k,F) \rightarrow 0.$$
Comme $F$ est fini, on déduit du lemme du serpent que:
$$\left( \frac{|{_n}T(k)|}{|T(k)/n|}\right) ^m \times \frac{|{_n}T_0(k)|}{|T_0(k)/n|}=  \frac{|{_n}R(k)|}{|R(k)/n|}.$$
Or nous avons montré que $\frac{|{_n}T_0(k)|}{|T_0(k)/n|} =n^{-d(\rho_R-\rho m)} $ et $\frac{|{_n}R(k)|}{|R(k)/n|}=n^{-d\rho_R}$. On en déduit que $ \frac{|{_n}T(k)|}{|T(k)/n|} = n^{-d\rho}$.
\end{itemize}
\end{proof}

\begin{remarque}\label{tor0qp}
Lorsque $k_0$ est un corps fini de caractéristique $p$, on montre de manière analogue que:
$$\frac{|{_n}\mathbb{G}_m(k)|}{|\mathbb{G}_m(k)/n|}=n^{-d}p^{-[k_1:\mathbb{Q}_p]v_p(n)},$$
pour tout $n \geq 1$, et si $p$ ne divise pas $n$:
$$ \frac{|{_n}T(k)|}{|T(k)/n|}=n^{-d\rho}, $$
pour tout tore $T$. Si de plus $d=1$, alors $ \frac{|{_n}T(k)|}{|T(k)/n|}=n^{-d\rho}p^{-[k_1:\mathbb{Q}_p]\dim T \cdot v_p(n)} $ pour tout $n$ et pour tout $T$.
\end{remarque}

\begin{proposition}\label{tor0tors}
Supposons que $k_0=\mathbb{C}((t))$. Soit $T$ un tore de rang $\rho$ sur $k$. On a alors pour chaque entier naturel naturel $n$:
$$\frac{|{_n}T(k)|}{|T(k)_{tors}/n|}=n^{\rho}.$$ 
\end{proposition}

\begin{proof}
\begin{itemize}
\item[$\bullet$] La propriété est évidente pour $T=\mathbb{G}_m$. Elle est donc aussi vraie pour tout tore quasi-trivial.
\item[$\bullet$] On se place maintenant dans le cas général. Soient $m$ un entier naturel non nul et $T_0$ un tore quasi-trivial sur $k$ tels que l'on a une suite exacte:
$$0 \rightarrow F \rightarrow R \rightarrow T^m \times_k T_0 \rightarrow 0$$
avec $F$ un schéma en groupes fini commutatif sur $k$ et $R$ un tore quasi-trivial sur $k$. Comme $F(k^s)$ est fini, on déduit une suite exacte:
$$0 \rightarrow F(k^s)_{tors}\rightarrow R(k^s)_{tors}\rightarrow T^m(k^s)_{tors}\times T_0(k^s)_{tors} \rightarrow 0.$$
En passant à la cohomologie, on obtient l'exactitude de:
$$0 \rightarrow F(k)_{tors} \rightarrow R(k)_{tors} \rightarrow T(k)_{tors}^m\times T_0(k)_{tors} \rightarrow H^1(k,F(k^s)_{tors}).$$
Or $F(k^s)_{tors}$ est fini. Donc il en est de même du groupe $H^1(k,F(k^s)_{tors})$, et il existe une suite exacte:
$$0 \rightarrow F(k)_{tors} \rightarrow R(k)_{tors} \rightarrow T(k)_{tors}^m\times T_0(k)_{tors} \rightarrow Q \rightarrow 0,$$
où $Q$ est fini. En notant $\rho_R$ le rang de $R$, on déduit du lemme du serpent que:
$$\left( \frac{|{_n}T(k)|}{|T(k)_{tors}/n|}\right) ^m \times \frac{|{_n}T_0(k)|}{|T_0(k)_{tors}/n|}=  \frac{|{_n}R(k)|}{|R(k)_{tors}/n|}.$$
Or nous avons montré que $\frac{|{_n}T_0(k)|}{|T_0(k)_{tors}/n|} =n^{\rho_R-\rho m} $ et $\frac{|{_n}R(k)|}{|R(k)_{tors}/n|}=n^{\rho_R}$. On en déduit que $ \frac{|{_n}T(k)|}{|T(k)/n|} = n^{\rho}$.
\end{itemize}
\end{proof}

\begin{remarque}\label{tor0torsqp}
\begin{itemize}
\item[(i)] Comme $T(k)_{tors}$ est un groupe de torsion de type cofini, on déduit de la proposition précédente que $T(k)_{tors} \cong F \oplus (\mathbb{Q}/\mathbb{Z})^{\rho}$ pour un certain groupe abélien fini $F$ dépendant de $T$.
\item[(ii)] Lorsque $k_0$ est un corps fini de caractéristique $p$, on a $\frac{|{_n}T(k)|}{|T(k)_{tors}/n|}=1$ pour tout tore $T$ puisque $T(k)_{tors}$ est fini.
\end{itemize}
\end{remarque}

\begin{proposition}\label{torr}
Supposons que $k_0=\mathbb{C}((t))$. Soit $T$ un tore de rang $\rho$ sur $k$. Soit $r \geq 1$. On note $c_{r,d} = 0$ si $r=1$ et $c_{r,d} = \binom{d+1}{r}$ si $r>1$. Il existe un groupe abélien fini $F$ (qui dépend de $r$ et de $T$) tel que:
$$H^r(k,T) \cong F \oplus (\mathbb{Q}/\mathbb{Z})^{c_{r,d}\rho }.$$ 
De plus, si $T = \mathbb{G}_m$ (ou si $T$ est quasi-trivial), alors $F=0$.
\end{proposition}

\begin{proof} On procède en deux étapes:
\begin{itemize}
\item[(A)] Montrons d'abord la proposition pour $T=\mathbb{G}_m$ en procédant par double récurrence sur $d$ et $r$. Pour $d=0$, la proposition est vraie par les théorème de Hilbert 90 et par dimension cohomologique. Supposons-la donc vraie pour un certain $d \geq 0$, et considérons $k$ un corps $d+1$-local. 
\begin{itemize}
\item[$\bullet$] Pour $r=1$, on a bien $H^1(k,\mathbb{G}_m)=0$ par le théorème de Hilbert 90. 
\item[$\bullet$] Pour $r=2$, si l'on note $\kappa$ le corps résiduel de $k$, on a pour chaque $n\geq 1$ la suite exacte (voir le paragraphe 2 de l'annexe du chapitre II de \cite{Ser}):
$$0 \rightarrow H^{2}(\kappa,\mu_n) \rightarrow H^{2}(k,\mu_n) \rightarrow H^{1}(\kappa,\mathbb{Z}/n\mathbb{Z}) \rightarrow 0.$$
Comme $H^2(\kappa,\mu_n)={_n}H^2(\kappa,\mathbb{G}_m)$, $H^2(k,\mu_n)={_n}H^2(k,\mathbb{G}_m)$ et $|H^1(\kappa,\mathbb{Z}/n\mathbb{Z})|=n^{d+1}$ d'après la preuve de la proposition précédente, on obtient que $|{_n}H^2(k,\mathbb{G}_m)|=n^{d+1}|{_n}H^2(k,\mathbb{G}_m)|=n^{d+1+\binom{d+1}{2}}= n^{\binom{d+2}{2}}$. On en déduit que $H^2(k,\mathbb{G}_m) \cong \mathbb{Q}/\mathbb{Z}^{\binom{d+2}{2}}$.
\item[$\bullet$] Supposons que l'on ait montré que $H^r(k,\mathbb{G}_m) \cong \mathbb{Q}/\mathbb{Z}^{\binom{d+2}{r}}$ pour un certain $r \geq 2$. On a alors la suite exacte:
$$0 \rightarrow H^{r+1}(\kappa,\mu_n) \rightarrow H^{r+1}(k,\mu_n) \rightarrow H^{r}(\kappa,\mathbb{Z}/n\mathbb{Z}) \rightarrow 0.$$
Par hypothèse de récurrence, les groupes $H^r(\kappa,\mathbb{G}_m)$, $H^r(k,\mathbb{G}_m)$ et $H^{r-1}(\kappa,\mathbb{G}_m)$ sont divisibles, et donc on a $H^{r+1}(\kappa,\mu_n) = {_n}H^{r+1}(\kappa,\mathbb{G}_m)$, $H^{r+1}(k,\mu_n) = {_n}H^{r+1}(k,\mathbb{G}_m)$ et $H^{r}(\kappa,\mathbb{Z}/n\mathbb{Z}) \cong H^{r}(\kappa,\mu_n) = {_n}H^r(\kappa,\mathbb{G}_m)$. On en déduit, toujours à l'aide de l'hypothèse de récurrence, que $|{_n}H^{r+1}(k,\mathbb{G}_m)| = n^{ \binom{d+1}{r+1}+  \binom{d+1}{r}}=n^{ \binom{d+2}{r+1}}$, et donc que $H^{r+1}(k,\mathbb{G}_m) \cong \mathbb{Q}/\mathbb{Z}^{\binom{d+2}{r+1}}$.
\end{itemize}
Cela achève la démonstration de la proposition pour $T=\mathbb{G}_m$. Le lemme de Shapiro montre alors que la proposition est vraie pour tout tore quasi-trivial.
\item[(B)] On se place maintenant dans le cas général. Soient $m$ un entier naturel non nul et $T_0$ un tore quasi-trivial sur $k$ tels que l'on a une suite exacte:
$$0 \rightarrow F \rightarrow R \rightarrow T^m \times_k T_0 \rightarrow 0$$
avec $F$ un schéma en groupes fini commutatif sur $k$ et $R$ un tore quasi-trivial sur $k$. Notons $\rho_R$ le rang de $R$. En passant à la cohomologie, on obtient une suite exacte:
$$H^r(k,F) \rightarrow H^r(k,R) \rightarrow H^r(k,T)^m\times H^r(k,T_0) \rightarrow H^{r+1}(k,F).$$
Comme $F$ est fini, on déduit du lemme du serpent que:
$$\left( \frac{|{_n}H^r(k,T)|}{|H^r(k,T)/n|}\right) ^m \times \frac{|{_n}H^r(k,T_0)|}{|H^r(k,T_0)/n|} = \frac{|{_n}H^r(k,R)|}{|H^r(k,R)/n|}.$$
Par conséquent:
$$\left(\frac{|{_n}H^r(k,T)|}{|H^r(k,T)/n|}\right) ^m =n^{c_{r,d}(-\rho_R+m\rho)} n^{c_{r,d}\rho_R})=n^{mc_{r,d}\rho},$$
et on a $ \frac{|{_n}H^r(k,T)|}{|H^r(k,T)/n|}=n^{c_{r,d}\rho}$. On en déduit que $H^r(k,T) \cong F \oplus \mathbb{Q}/\mathbb{Z}^{c_{r,d}\rho }$ pour un certain groupe abélien fini $F$.
\end{itemize}
\end{proof}

\begin{remarque}\label{torrqp}
Lorsque $k_0$ est un corps fini de caractéristique $p$, on montre de manière analogue que, pour $r \geq 3$:
\begin{gather*}
\frac{|{_n}H^2(k,\mathbb{G}_m)|}{|H^r(k,\mathbb{G}_m)/n|} = n^d \cdot p^{(d-1)[k_1:\mathbb{Q}_p]v_p(n)},\\
\frac{|{_n}H^r(k,\mathbb{G}_m)|}{|H^r(k,\mathbb{G}_m)/n|} = p^{c_{r-1,d-2}[k_1:\mathbb{Q}_p]v_p(n)}.
\end{gather*}
pour tout $n \geq 1$, et si $p$ ne divise pas $n$ ou si $d=1$:
\begin{gather*}
\frac{|{_n}H^2(k,T)|}{|H^2(k,T)/n|} = n^{d\rho} ,\\
\frac{|{_n}H^r(k,T)|}{|H^r(k,T)/n|} = 1.
\end{gather*}
pour tout tore $T$ et tout $r \geq 3$.
\end{remarque}

\begin{proposition}\label{tor1tors}
Supposons que $k_0=\mathbb{C}((t))$. Soit $T$ un tore de rang $\rho$ sur $k$. On a alors pour chaque entier naturel naturel $n$:
$$\frac{|{_n}H^1(k,T(k^s)_{tors})|}{|H^1(k,T(k^s)_{tors})/n|}=n^{(d+1)\rho}.$$ 
\end{proposition}

\begin{proof}
Pour $T=\mathbb{G}_m$, on a $H^1(k,\mu_n) = {_n}H^1(k,\mathbb{G}_m(k^s)_{tors})$, et donc $H^1(k,\mathbb{G}_m(k^s)_{tors}) \cong (\mathbb{Q}/\mathbb{Z})^{d+1}$. La formule est donc vraie pour les tores quasi-triviaux. En procédant comme dans les propositions précédentes, on obtient le résultat désiré.
\end{proof}

\begin{remarque}\label{tor1torsqp}
Lorsque $k_0$ est un corps fini de caractéristique $p$, on montre de manière analogue que:
$$\frac{|{_n}H^1(k,\mathbb{G}_m(k^s)_{tors})|}{|H^1(k,\mathbb{G}_m(k_s)_{tors})/n|} = n^d \cdot p^{[k_1:\mathbb{Q}_p]v_p(n)}.$$
pour tout $n \geq 1$, et si $p$ ne divise pas $n$:
$$\frac{|{_n}H^1(k,T(k^s)_{tors})|}{|H^1(k,T(k_s)_{tors})/n|} = n^{d\rho} .$$
pour tout tore $T$. Si de plus $d=1$, alors $$\frac{|{_n}H^1(k,T(k^s)_{tors})|}{|H^1(k,T(k_s)_{tors})/n|} = n^{d\rho}\cdot p^{[k_1:\mathbb{Q}_p]\dim T \cdot v_p(n)} $$ pour tout $n$ et pour tout $T$.\\ Plus généralement, si $r \in \mathbb{N}$, $i \in \mathbb{Z}$, $\ell$ est un nombre premier différent de $p$ et $t \in \mathbb{N}$, alors:
\begin{equation*} \frac{|{_{\ell^t}}H^r(k,\mathbb{Q}_{\ell}/\mathbb{Z}_{\ell}(i))|}{|H^r(k,\mathbb{Q}_{\ell}/\mathbb{Z}_{\ell}(i))/{\ell^t}|} = \left\lbrace
\begin{array}{cc}
\ell^{tc_{i,d-1}}  & \text{si } r\in \{i,i+1\} \\
1 & \text{sinon,}
\end{array}\right. 
\end{equation*}
\begin{equation*} \frac{|{_{\ell^t}}H^r(k,\varinjlim_s {_{\ell^s}}T(k^s) \otimes \mathbb{Z}/\ell^s\mathbb{Z}(i))|}{|H^r(k,\varinjlim_s {_{\ell^s}}T(k^s) \otimes \mathbb{Z}/\ell^s\mathbb{Z}(i))/{\ell^t}|} = \left\lbrace
\begin{array}{cc}
\ell^{tc_{i+1,d-1}\rho}  & \text{si } r\in \{i+1,i+2\} \\
1 & \text{sinon.}
\end{array}\right. 
\end{equation*}
\end{remarque}

\subsection{Variétés abéliennes}

\subsubsection{Réduction des variétés abéliennes}

Soient $l$ un corps local et $A$ une variété abélienne sur $l$. Notons $\mathcal{A}$ le modèle de Néron de $A$ et $A_0$ sa fibre spéciale. Soit $A_0^0$ la composante connexe du neutre dans $A_0$. On rappelle que $A_0/A_0^0$ est un groupe algébrique fini et qu'il existe une suite exacte:
$$0 \rightarrow U \times_l T \rightarrow A_0^0 \rightarrow B \rightarrow 0,$$
où $U$ est un groupe abélien unipotent, $T$ un tore et $B$ une variété abélienne. Dans le cas où $l$ est de caractéristique résiduelle nulle, $U$ est une puissance de $\mathbb{G}_a$.

\begin{definition}
On dit que $A$ a \textbf{très mauvaise réduction} si $T=0$ et $B=0$. 
\end{definition}

\begin{theorem} \textbf{(Théorème de Ogg - Théorème 1 de \cite{Ogg})}\\
Soient $l$ un corps algébriquement clos de caractéristique 0 et $A$ une variété abélienne sur $l((t))$. Soient $\mathcal{A}$ le modèle de Néron de $A$ et $A_0$ la fibre spéciale de $\mathcal{A}$. Soit $A_0^0$ la composante connexe du neutre dans $A_0$. On considère une suite exacte $0 \rightarrow U \times_l T \rightarrow A_0^0 \rightarrow B \rightarrow 0$ où $U$ est une puissance de $\mathbb{G}_a$, $T$ est un tore et $B$ une variété abélienne. Soient $r$ la dimension de $T$, $s$ la dimension de $U$ et $\epsilon = r+2s$. Alors $H^1(l,A) \cong (\mathbb{Q}/\mathbb{Z})^{2\dim A - \epsilon}$. En particulier, le groupe $H^1(l,A) \cong (\mathbb{Q}/\mathbb{Z})^{2\dim A - \epsilon}$ est nul si, et seulement si, la variété abélienne $A$ a très mauvaise réduction.
\end{theorem}

\begin{definition}
On appellera $\epsilon$ l'\textbf{entier de Ogg} de $A$.
\end{definition}

\subsubsection{Variétés abéliennes sur un corps fini}

Soit $\mathbb{F}$ un corps fini de caractéristique $p$ et de cardinal $q$. Soient $\ell$ un nombre premier différent de $p$ et $i$ un entier relatif. Le but de ce paragraphe est d'établir la proposition suivante:

\begin{proposition}\label{vafini}
Soit $A$ une variété abélienne sur $\mathbb{F}$. On note $A\{\ell\}(i)$ le module galoisien $\varinjlim_r {_{\ell^r}}A(\mathbb{F}^s) \otimes \mathbb{Z}/\ell^r\mathbb{Z}(i)$. Alors le groupe $H^0(\mathbb{F},A\{\ell\}(i))$ est fini.
\end{proposition}

\begin{proof}
Supposons que $H^0(\mathbb{F},A\{\ell\}(i))$ soit infini. Cela signifie que $H^0(\mathbb{F},A\{\ell\}(i))$ possède un sous-groupe isomorphe à $\mathbb{Q}_{\ell}/\mathbb{Z}_{\ell}$, et donc que $A(\mathbb{F}^s)$ possède un sous-module galoisien isomorphe à $\mathbb{Q}_{\ell}/\mathbb{Z}_{\ell}(-i)$. Par conséquent, le module de Tate $T_{\ell} A(\mathbb{F}^s)$ contient $\mathbb{Z}_{\ell}(-i)$ comme module galoisien. Comme le Frobenius géométrique agit sur $\mathbb{Z}_{\ell}(1)$ par  multiplication par $q^{-1}$, on déduit que le Frobenius géométrique sur $T_{\ell} A(\mathbb{F}^s) \otimes_{\mathbb{Z}_{\ell}} \mathbb{Q}_{\ell}$ possède une valeur propre de module complexe $q^{i}$. Mais $T_{\ell} A(\mathbb{F}^s)$ est isomorphe à $H^1(A^t,\mathbb{Z}_{\ell}) \otimes_{\mathbb{Z}_{\ell}} \mathbb{Z}_{\ell}(1)$, et d'après les conjectures de Weil (théorème IV.1.2 de \cite{FK}), les valeurs propres du Frobenius géométrique sur $H^1(A^t,\mathbb{Z}_{\ell}) \otimes_{\mathbb{Z}_{\ell}} \mathbb{Z}_{\ell}(1)\otimes_{\mathbb{Z}_{\ell}} \mathbb{Q}_{\ell}$ sont de module complexe $q^{-1/2}$: absurde! Donc $H^0(\mathbb{F},A\{\ell\}(i))$ est fini.
\end{proof}

\begin{remarque}
Ce résultat sera notamment utile dans la section \ref{Qp} afin de déterminer quand il y a un bon théorème de dualité pour les groupes de cohomologie d'une variété abélienne sur un corps de la forme $\mathbb{Q}_p((t_1))...((t_{d-1}))$.
\end{remarque}

\section{\textsc{Variétés abéliennes sur le corps des fonctions d'une courbe sur $\mathbb{C}((t))$}}\label{C}

\subsection{Étude locale}

Le but de cette partie est d'établir un théorème de dualité pour les variétés abéliennes sur $\mathbb{C}((t_0))((t_1))$ (théorème \ref{1-local}). Il s'agit d'un résultat analogue aux théorèmes de dualité pour les variétés abéliennes sur $\mathbb{Q}_p$ et sur $\mathbb{F}_p((t))$. En fait, c'est essentiellement le "dernier" cas de corps local de dimension cohomologique 2 non compris, et il se trouve qu'il pose certaines difficultés supplémentaires par rapport aux cas déjà connus.\\
On se place donc dans le cas où $k = \mathbb{C}((t_0))((t_1))$ (et alors $d=1$).  Soient $\kappa = \mathbb{C}((t_0))$ et $A$ une variété abélienne sur $k$. Soit $A^t$ sa variété abélienne duale, qui, d'après le théorème de Barsotti-Weil, représente le faisceau $\underline{\text{Ext}}^1_k(A,\mathbb{G}_m)$ (on rappelle que $\underline{\text{Ext}}^r_k(A,\mathbb{G}_m)$ est le faisceau associé à $T \mapsto \text{Ext}^r_{T_{fppf}}(F,G)$). Comme $\underline{\text{Ext}}^r_k(A,\mathbb{G}_m)=0$ pour $r \neq 1$, on dispose d'un accouplement $A \otimes^{\mathbf{L}} A^t \rightarrow \mathbb{G}_m[1],$
d'où un accouplement:
$$H^r(k,A) \times H^{1-r}(k,A^t) \rightarrow \text{Br} \; k \cong \mathbb{Q}/\mathbb{Z}.$$

\begin{lemma}
Pour chaque entier naturel $n$ et chaque entier $r$, le morphisme $ H^{r-1}(k,A)/n \rightarrow ({_n}H^{2-r}(k,A^t))^D$ est injectif et le morphisme ${_n}H^{r}(k,A) \rightarrow (H^{1-r}(k,A^t)/n)^D$ est surjectif.
\end{lemma}

\begin{proof}
On remarque que $\text{\underline{Hom}}({_n}A,\mathbb{Q}/\mathbb{Z}(1))= {_n}A^t$. On en déduit que, dans le diagramme commutatif:\\
\centerline{\xymatrix{
0 \ar[r]& H^{r-1}(k,A)/n \ar[d] \ar[r] & H^r(k,{_n}A) \ar[r]\ar[d] & {_n}H^r(k,A) \ar[r]\ar[d] & 0\\
0 \ar[r]& ({_n}H^{2-r}(k,A^t))^D  \ar[r] & H^{2-r}(k,{_n}A^t)^D \ar[r] & (H^{1-r}(k,A^t)/n)^D \ar[r] & 0,
}}
la flèche verticale centrale est un isomorphisme. Par conséquent, $ H^{r-1}(k,A)/n \rightarrow ({_n}H^{2-r}(k,A^t))^D$ est injectif et ${_n}H^{r}(k,A) \rightarrow (H^{1-r}(k,A^t)/n)^D$ est surjectif.
\end{proof}

Notons $\mathcal{A}$ le modèle de Néron de $A$ et $A_0$ sa fibre spéciale. On dispose alors d'une filtration $A_0 \supseteq A_0^{0} \supseteq A_0^1$ de $A_0$,
où:
\begin{itemize} 
\item[$\bullet$] $A_0^0$ est la composante connexe de l'élément neutre de $A_0$, 
\item[$\bullet$] $F=A_0/A_0^0$ est un groupe fini,
\item[$\bullet$] $A_0^1$ est de la forme $U \times T$ où $U$ est un groupe additif (c'est-à-dire une puissance de $\mathbb{G}_a$) et $T$ un tore,
\item[$\bullet$] $B=A_0^0/A_0^1$ est une variété abélienne.
\end{itemize}
On note $\epsilon$ l'entier de Ogg de $B$.\\
De même,  on note $\mathcal{A}^*$ le modèle de Néron de $A^t$ et $A_0^*$ sa fibre spéciale. On dispose alors d'une filtration $A_0^* \supseteq A_0^{0*} \supseteq A_0^{1*}$ de $A_0^*$,
où:
\begin{itemize} 
\item[$\bullet$] $A_0^{0*}$ est la composante connexe de l'élément neutre de $A_0^*$, 
\item[$\bullet$] $F^*=A_0^*/A_0^{0*}$ est un groupe fini,
\item[$\bullet$] $A_0^{1*}$ est de la forme $U^* \times T^*$ où $U^*$ est un groupe additif et $T^*$ un tore,
\item[$\bullet$] $B^*=A_0^{0*}/A_0^{1*}$ est une variété abélienne.
\end{itemize}
On note $\epsilon^*$ l'entier de Ogg de $B^*$.

\begin{remarque}
Ainsi, toutes les variétés indexées par 0 sont définies sur $\kappa$.
\end{remarque}

\begin{theorem} \label{1-local}
On a une suite exacte:
$$0 \rightarrow (\mathbb{Q}/\mathbb{Z})^{m(A)} \rightarrow H^1(k,A) \rightarrow (H^0(k,A^t)^{\wedge})^D \rightarrow 0$$
avec $m(A) = 2\cdot (\text{dim}\; B + \text{dim}\; B^*) - \epsilon - \epsilon^*$.
\end{theorem}

\begin{remarque}
La multiplication par $n$ sur $A^t$ est étale puisque $k$ est de caractéristique 0. Donc, d'après le théorème des fonctions implicites, le groupe $nA^t(k)$ est ouvert dans $A^t(k)$. De plus, il est d'indice fini. On en déduit que $H^0(k,A^t)^{\wedge}$ coïncide avec le complété profini de $H^0(k,A^t)$.
\end{remarque}

\begin{proof}
La surjectivité du morphisme $ H^1(k,A) \rightarrow (H^0(k,A^t)^{\wedge})^D$ découle immédiatement du lemme précédent par passage à la limite inductive et de la remarque précédente. Nous voulons maintenant calculer son noyau $N$. \\
On a l'égalité $A(k) = \mathcal{A}(\mathcal{O}_k)$, ainsi qu'une suite exacte:
$$0 \rightarrow D \rightarrow \mathcal{A}(\mathcal{O}_k) \rightarrow A_0(\kappa) \rightarrow 0,$$
où $D$ désigne un groupe abélien uniquement divisible. Le lemme du serpent impose donc que:
$$\frac{|A(k)/n|}{|{_n}A(k)|}=\frac{|A_0(\kappa)/n|}{|{_n}A_0(\kappa)|}.$$
Nous allons maintenant dévisser $A_0$.
\begin{itemize}
\item[$\bullet$] En exploitant la suite exacte $0 \rightarrow A^0_0(\kappa) \rightarrow A_0(\kappa) \rightarrow F(\kappa)$, le lemme du serpent et la finitude de $F(\kappa)$, on obtient que:
$$\frac{|A_0(\kappa)/n|}{|{_n}A_0(\kappa)|}=\frac{|A_0^0(\kappa)/n|}{|{_n}A_0^0(\kappa)|}.$$
\item[$\bullet$] Comme $H^1(\kappa,U)=0$ et $H^1(\kappa,T)=0$ (puisque d'une part $H^1(\kappa,T)$ est d'exposant fini d'après le théorème de Hilbert 90 et d'autre part c'est un groupe divisible car $\kappa$ est de dimension cohomologique 1), on a une suite exacte: $$0 \rightarrow U(\kappa) \times T(\kappa) \rightarrow A_0^0(\kappa) \rightarrow B(\kappa) \rightarrow 0.$$ 
Par conséquent:
$$\frac{|A_0^0(\kappa)/n|}{|{_n}A_0^0(\kappa)|} = \frac{|B(\kappa)/n|}{|{_n}B(\kappa)|} \cdot \frac{|U(\kappa)/n|}{|{_n}U(\kappa)|} \cdot \frac{|T(\kappa)/n|}{|{_n}T(\kappa)|}.$$
Or:
\begin{itemize}
\item[$\circ$] on a $ \frac{|U(\kappa)/n|}{|{_n}U(\kappa)|} =1$;
\item[$\circ$] comme $\chi (\kappa,{_n}T)=1$ et $H^1(\kappa,T)=0$), on a $\frac{|T(\kappa)/n|}{|{_n}T(\kappa)|}=\frac{1}{|{_n}H^1(\kappa,T)|}=1$;
\item[$\circ$] d'après le théorème de Ogg, on a $H^1(\kappa,B) \cong (\mathbb{Q}/\mathbb{Z})^{2\cdot \text{dim}\; B - \epsilon}$; comme $\chi (\kappa,{_n}B)=1$, on obtient $\frac{|B(\kappa)/n|}{|{_n}B(\kappa)|} = \frac{1}{|{_n}H^1(\kappa,B)|} = \frac{1}{n^{2\cdot \text{dim}\; B - \epsilon}}$. 
\end{itemize}
On en déduit que:
$$\frac{|A(k)/n|}{|{_n}A(k)|}=\frac{|A_0^0(\kappa)/n|}{|{_n}A_0^0(\kappa)|} = \frac{1}{n^{2\cdot \text{dim}\; B - \epsilon}}.$$

\end{itemize}
On montre de même que:
$$\frac{|A^t(k)/n|}{|{_n}A^t(k)|}=\frac{1}{n^{2\cdot \text{dim}\; B^* - \epsilon^*}}.$$
On a alors:
\begin{align*} 
\chi (k,{_n}A)&= \frac{|{_n}A(k)||H^2(k,{_n}A)|}{|A(k)/n||{_n}H^1(k,A)|}\\
&= n^{2\cdot \text{dim}\; B - \epsilon} \frac{|H^2(k,{_n}A)|}{|{_n}H^1(k,A)|}\\
&=  n^{2\cdot \text{dim}\; B - \epsilon} \frac{|{_n}A^t(k))|}{|{_n}H^1(k,A)|}\\
&= n^{2\cdot (\text{dim}\; B + \text{dim}\; B^*) - \epsilon - \epsilon^*}\frac{|A^t(k)/n|}{|{_n}H^1(k,A)|},
\end{align*}
et donc $|{_n}H^1(k,A)| =  n^{2\cdot (\text{dim}\; B + \text{dim}\; B^*) - \epsilon - \epsilon^*}|A^t(k)/n|$. Par conséquent, $|{_n}N|=n^{2\cdot (\text{dim}\; B + \text{dim}\; B^*) - \epsilon - \epsilon^*}$. Cela étant vrai pour tout $n$, comme $N$ est de torsion de type cofini, on a $N \cong (\mathbb{Q}/\mathbb{Z})^{2\cdot (\text{dim}\; B + \text{dim}\; B^*) - \epsilon - \epsilon^*}$, ce qui achève la preuve avec $m(A) = 2\cdot (\text{dim}\; B + \text{dim}\; B^*) - \epsilon - \epsilon^*$.
\end{proof}

\begin{remarque}
Comme $A$ et $A^t$ sont isogènes (paragraphe 10 de \cite{MilAV}), on a:
$$\frac{|A^t(k)/n|}{|{_n}A^t(k)|}=\frac{|A(k)/n|}{|{_n}A(k)|}$$
pour tout $n$. Cela montre que $2\cdot \text{dim}\; B^* - \epsilon^* = 2\cdot \text{dim}\; B - \epsilon$, et donc $m (A) = 2(2\cdot \text{dim}\; B - \epsilon)$
\end{remarque}

\begin{corollary} \label{cor local}
\begin{itemize}
\item[(i)] Si $A$ a bonne réduction, alors $m(A) = 4\cdot \text{dim}\; A - 2\epsilon$. Si de plus $A_0$ a bonne réduction (resp. très mauvaise réduction), alors $m(A) = 4\cdot \text{dim}\; A$ (resp. $m(A) =0$), et il y a donc une suite exacte de groupes de torsion de type cofini:
$$0 \rightarrow (\mathbb{Q}/\mathbb{Z})^{4\cdot \text{dim}\; A} \rightarrow H^1(k,A) \rightarrow (H^0(k,A^t)^{\wedge})^D \rightarrow 0$$
(resp. un isomorphisme $H^1(k,A) \cong (H^0(k,A^t)^{\wedge})^D$).
\item[(ii)] En général, on a un isomorphisme $\overline{H^1(k,A)} \cong \overline{(H^0(k,A^t)^{\wedge})^D}$.
\item[(iii)] Le morphisme $H^1(k,A) \rightarrow (H^0(k,A^t)^{\wedge})^D$ est un isomorphisme si, et seulement si, $B$ a très mauvaise réduction.
\end{itemize}
\end{corollary}

Dans certains cas, il est possible d'expliciter le noyau de $H^1(k,A) \rightarrow (H^0(k,A^t)^{\wedge})^D$.  Pour ce faire, il convient d'établir quelques résultats préliminaires:

\begin{lemma}\label{cohonr}
Le morphisme naturel $H^1(\mathcal{O}_k,\mathcal{A}) \rightarrow H^1(k,A)$ est injectif d'image le sous-groupe $H^1(k^{nr}/k,A(k^{nr}))$ de $ H^1(k,A)$.
\end{lemma}

\begin{proof}
Notons $g: \text{Spec} \; k \rightarrow \text{Spec} \; \mathcal{O}_k$ et $i: \text{Spec} \; \kappa \rightarrow \text{Spec} \; \mathcal{O}_k$. Comme $\mathcal{A}\cong g_*A$ (d'après la propriété universelle du modèle de Néron), il suffit de remarquer que:
$$H^1(\mathcal{O}_k,\mathcal{A}) \cong H^1(\mathcal{O}_k, g_*A) \cong H^1(\kappa,i^*g_*A) \cong H^1(k^{nr}/k,A(k^{nr})).$$
\end{proof}

\begin{proposition}\label{prelC}
Soit $\ell$ un nombre premier ne divisant pas $|F|$.
\begin{itemize}
\item[(i)] Les groupes $A(k^{nr})$ et $A^t(k^{nr})$ sont $\ell$-divisibles.
\item[(ii)] Il existe un morphisme fonctoriel injectif $$(T_{\ell}H^1(\mathcal{O}_k,\mathcal{A}^*))^D \rightarrow (\varprojlim_r H^1(k^{nr}/k,{_{\ell^r}}A^t(k^{nr})))^D.$$
\end{itemize}
\end{proposition}

\begin{proof}
\item[(i)] Montrons d'abord que $A(k^{nr})$ est $\ell$-divisible. Comme $A(k^{nr})=\mathcal{A}(\mathcal{O}_{k^{nr}})$ et comme il existe un morphisme surjectif $\mathcal{A}(k^{nr}) \rightarrow A_0(\kappa^s)$ à noyau divisible, cela revient à montrer que $A_0(\kappa^s)$ est $\ell$-divisible. En exploitant la suite exacte $0 \rightarrow A_0^0 \rightarrow A_0 \rightarrow F \rightarrow 0$ et le fait que $\ell$ ne divise pas $|F|$, on remarque alors qu'il suffit de prouver que $A_0^0(\kappa^s)$ est $\ell$-divisible. Mais cela découle immédiatement de l'exactitude de $0 \rightarrow U \times T \rightarrow A_0^0 \rightarrow B \rightarrow 0$. Donc $A(k^{nr})$ est $\ell$-divisible.\\
D'après le paragraphe IX.11.3 de \cite{SGA7}, on a $|F^*|=|F|$. On en déduit que $\ell$ ne divise pas $|F^*|$ et donc que le groupe $A^t(k^{nr})$ est $\ell$-divisible.
\item[(ii)] D'après (i), on a une suite exacte:
$$0 \rightarrow {_{\ell^r}}A^t(k^{nr}) \rightarrow A^t(k^{nr}) \rightarrow A^t(k^{nr}) \rightarrow 0.$$
Comme ${_{\ell}}H^1(k^{nr}/k,A^t(k^{nr}))$ et $H^1(k^{nr}/k,{_{\ell^r}}A^t(k^{nr}))$ sont finis, en passant à la limite projective, on obtient une surjection $\varprojlim_r H^1(k^{nr}/k,{_{\ell^r}}A^t(k^{nr}))\rightarrow T_{\ell}H^1(k^{nr}/k,A^t(k^{nr}))$. En utilisant \ref{cohonr}, cela permet de réaliser $(T_{\ell}H^1(k^{nr}/k,A^t(k^{nr})))^D = (T_{\ell}H^1(\mathcal{O}_k,\mathcal{A}^*))^D $ comme un sous-groupe de $(\varprojlim_r H^1(k^{nr}/k,{_{\ell^r}}A^t(k^{nr})))^D$.
\end{proof}

\begin{lemma}
On peut obtenir un isomorphisme $\iota_{\ell}: H^0(k^{nr}/k,H^1(k^{nr},A))\{\ell\} \rightarrow (\varprojlim_r H^1(k^{nr}/k,{_{\ell^r}}A^t(k^{nr})))^D$ par composition des isomorphismes naturels:
\begin{align*}
 H^0(k^{nr}/k,H^1(k^{nr},A))\{\ell\} & \xrightarrow{\sim} \varinjlim_r H^0(k^{nr}/k,{_{\ell^r}}H^1(k^{nr},A)) \\
&\xleftarrow{\sim} \varinjlim_r H^0(k^{nr}/k,H^1(k^{nr},{_{\ell^r}}A)) \\
&\xrightarrow{\sim} \varinjlim_r H^0(k^{nr}/k,{_{\ell^r}}A^t(k^{nr})^D)  \\
&\xrightarrow{\sim} (\varprojlim_r H^1(k^{nr}/k,{_{\ell^r}}A^t(k^{nr})))^D .
\end{align*} 
\end{lemma}

\begin{proof}
Le premier isomorphisme est évident. Le deuxième découle du fait que $A(k^{nr})$ est $\ell$-divisible. Les deux derniers sont obtenus par dualité sur $\text{Gal}(k^s/k^{nr}) \cong \text{Gal}(k^{nr}/k) \cong \hat{\mathbb{Z}}$ (voir exemple 1.10 de \cite{MilADT}).
\end{proof}

Nous sommes maintenant en mesure d'introduire la définition suivante:

\begin{definition}\label{HnrsC}
Soit $\ell$ un nombre premier ne divisant pas $|F|$. On appelle \textbf{$\ell$-groupe de cohomologie non ramifiée symétrisé de $A$} le groupe:
$$H^1_{nrs}(k,A,\ell) := (\iota_{\ell} \circ \text{Res})^{-1}((T_{\ell}H^1(\mathcal{O}_k,\mathcal{A}^*))^D) \subseteq H^1(k,A)\{\ell\}$$
où $\text{Res}: H^1(k,A) \rightarrow H^0(k^{nr}/k,H^1(k^{nr},A))$ désigne la restriction et $\iota_{\ell}$ l'isomorphisme du lemme précédent. On a alors une suite exacte:
$$0 \rightarrow H^1(\mathcal{O}_k,\mathcal{A})\{\ell\} \rightarrow H^1_{nrs}(k,A,\ell) \rightarrow (T_{\ell}H^1(\mathcal{O}_k,\mathcal{A}^*))^D \rightarrow 0.$$
\end{definition}

\begin{remarque}
\begin{itemize}
\item[$\bullet$] Le $\ell$-groupe de cohomologie non ramifiée symétrisé de $A$ est bien défini quel que soit $\ell$ lorsque $F$ est trivial, c'est-à-dire lorsque $A_0$ est connexe. 
\item[$\bullet$] La suite exacte:
$$0 \rightarrow H^1(\mathcal{O}_k,\mathcal{A})\{\ell\} \rightarrow H^1_{nrs}(k,A,\ell) \rightarrow (T_{\ell}H^1(\mathcal{O}_k,\mathcal{A}^*))^D \rightarrow 0$$
s'identifie à la suite exacte de groupes abstraits:
$$0 \rightarrow (\mathbb{Q}_{\ell}/\mathbb{Z}_{\ell})^{m(A)/2} \rightarrow (\mathbb{Q}_{\ell}/\mathbb{Z}_{\ell})^{m(A)} \rightarrow (\mathbb{Q}_{\ell}/\mathbb{Z}_{\ell})^{m(A)/2} \rightarrow 0.$$
\end{itemize}
\end{remarque}

\begin{theorem}\label{noyau}
Pour $\ell$ premier ne divisant pas $|F|$, la partie $\ell$-primaire du noyau de $H^1(k,A) \rightarrow (H^0(k,A^t)^{\wedge})^D$ est $H^1_{nrs}(k,A,\ell)$.
\end{theorem}

\begin{proof}
Considérons le diagramme suivant:\\
\centerline{\xymatrix{
& \varinjlim_r H^0(k^{nr}/k,{_{\ell^r}}H^1(k^{nr},A)) & \\
H^1(k,A)\{\ell\} \ar[ru]^{Res} \ar[ddd]^{f_4} & & \varinjlim_r H^0(k^{nr}/k,H^1(k^{nr},{_{\ell^r}}A))\ar[lu]^{\cong}_{f_1}\ar[ddd]^{\cong}_{f_6}\\
& \varinjlim_r H^1(k,{_{\ell^r}}A)  \ar[ru]^{f_3} \ar@{->>}[lu]^{f_2}\ar[d]^{f_5} &\\
& (\varprojlim_rH^1(k,{_{\ell^r}}A^t))^D \ar[dl]^{f_7} \ar[dr]^{f_8} & \\
(H^0(k,A^t)^{(\ell)})^D & & (\varprojlim_r H^1(k^{nr}/k,{_{\ell^r}}A^t(k^{nr})))^D \ar[dl]^{f_9}\\
& (H^0(k^{nr}/k,A^t(k^{nr}))^{(\ell)})^D \ar@{=}[ul] &
}}
où le morphisme $f_6$ est obtenu par composition des isomorphismes $$\varinjlim_r H^0(k^{nr}/k,H^1(k^{nr},{_{\ell^r}}A)) \xrightarrow{\sim} \varinjlim_r H^0(k^{nr}/k,{_{\ell^r}}A^t(k^{nr})) \xrightarrow{\sim} \varinjlim_r (H^1(k^{nr}/k,{_{\ell^r}}A^t(k^{nr})))^D$$ provenant de la dualité pour la cohomologie du groupe profini $\hat{\mathbb{Z}}$. On vérifie aisément que ce diagramme est commutatif.\\
Soit maintenant $x \in {_{\ell^r}}H^1(k,A)$. Comme $f_2$ est surjectif, on peut relever $x$ en $z \in H^1(k,{_{\ell^r}}A)$. On remarque alors que $f_3(z)=f_1^{-1}(Res(x))$. Donc, par définition de $\iota_{\ell}$: 
\begin{align*}
f_9(\iota_{\ell}(Res(x)))&=f_9(f_6(f_1^{-1}(Res(x)))=f_9(f_6(f_3(z)))\\ &=f_9(f_8(f_5(z)))=f_7(f_5(z))=f_4(x).
\end{align*}
On en déduit que $x \in \text{Ker}(f_4)$ si, et seulement si,
$$\iota_{\ell}(Res(x)) \in \text{Ker}(f_9) = (T_{\ell}H^1(\mathcal{O}_k,\mathcal{A}^*))^D.$$
\end{proof}

\begin{remarque}
Pour $\ell$ divisant $|F|$, le noyau de $H^1(k,A) \rightarrow (H^0(k,A^t)^{\wedge})^D$ ne contient pas forcément $H^1(\mathcal{O}_k,\mathcal{A})\{\ell\} $: par exemple, si $\ell$ divise $|F(\kappa)|$ et $A$ a très mauvaise réduction, le groupe  $H^1(\mathcal{O}_k,\mathcal{A})\{\ell\} $ est non trivial alors que le morphisme $H^1(k,A) \rightarrow (H^0(k,A^t)^{\wedge})^D$ est injectif. En fait, pour $\ell$ divisant $|F|$, il semble difficile de caractériser la partie $\ell$-primaire du noyau de $H^1(k,A) \rightarrow (H^0(k,A^t)^{\wedge})^D$: en particulier, il serait intéressant de déterminer si elle contient forcément $H^1(\mathcal{O}_k,\mathcal{A})\{\ell\}_{div}$.
\end{remarque}

Pour alléger les notations dans la section suivante, nous noterons:
$$H^1_{nrs}(k,A) := \bigoplus_{\ell \wedge |F|=1} H^1_{nrs}(k,A,\ell).$$
C'est le groupe de torsion dont la partie $\ell$-primaire est $H^1_{nrs}(k,A,\ell)$ si $\ell$ ne divise pas $|F|$, triviale sinon.

\subsection{Étude globale}\label{globC((t))}

Supposons maintenant que $k=\mathbb{C}((t))$ (et donc que $d=0$). Soient $A$ une variété abélienne sur $K=k(X)$ et $A^t$ sa variété abélienne duale. Le but de ce paragraphe est d'établir un théorème de dualité à la Cassels-Tate pour $A$: plus précisément, nous voulons déterminer, sous certaines hypothèses géométriques et modulo divisibles, le dual du groupe de Tate-Shafarevich $\Sha^1(A)$.

Pour chaque $v \in X^{(1)}$, notons:
\begin{itemize}
\item[$\bullet$] $\mathcal{A}_v$ le modèle de Néron de $A$ sur $\mathcal{O}_v$,
\item[$\bullet$] $F_v$ le groupe algébrique fini des composantes connexes de la fibre spéciale de $\mathcal{A}_v$,
\item[$\bullet$] $B_v$ la variété abélienne qui apparaît comme composante de la fibre spéciale de $\mathcal{A}_v$. 
\end{itemize}
Notons aussi $U$ l'ouvert de bonne réduction de $A$, de sorte que le modèle de Néron $\mathcal{A}$ de $A$ sur $U$ est un schéma abélien. Soit $\mathcal{A}^t$ le schéma abélien dual. \\
Fixons maintenant un nombre premier $\ell$ et faisons l'hypothèse suivante:

\begin{hypol}\label{HC}
\begin{minipage}[t]{12.72cm}
pour chaque $v \in X \setminus U$, au moins l'une des deux affirmations suivantes est vérifiée:
\begin{itemize}
\item[$\bullet$] $\ell$ ne divise pas $|F_v|$,
\item[$\bullet$] $B_v$ a très mauvaise réduction.
\end{itemize}
\end{minipage}
\end{hypol}

\begin{remarque}
Étant donnée une variété abélienne $A$, l'hypothèse précédente est vérifiée pour presque tout $\ell$. Par conséquent, les résultats que nous allons montrer sont vrais pour presque tout $\ell$.
\end{remarque}

Soit $Z$ l'ensemble des $v \in X^{(1)}$ tels que $B_v$ a très mauvaise réduction. Pour chaque ouvert $V$ de $U$, on introduit les groupes suivants:
\begin{multline*}
\Sha^1_{nr}(V,A):= \text{Ker} \left(  H^1(K,A) \rightarrow \prod_{v \in X \setminus V} H^1(K_v,A) \times \prod_{v \in V^{(1)}} H^1(K_v,A)/ H^1_{nr}(K_v,A))\right) ,  
\end{multline*}
\begin{multline*}
\Sha^1_{nrs}(V,A^t):= \text{Ker} \left(  H^1(K,A^t) \rightarrow \prod_{v \in Z \setminus V} H^1(K_v,A^t) \times \prod_{v \in X \setminus (V \cup Z)} H^1(K_v,A^t)/ H^1_{nrs}(K_v,A^t) \right. \\ \times \left. \prod_{v \in V^{(1)}} H^1(K_v,A^t)/ H^1_{nr}(K_v,A^t))\right) , 
\end{multline*}
\begin{multline*}
\Sha^1_{nrs}(A^t):= \text{Ker} \left(  H^1(K,A^t) \rightarrow \prod_{v \in Z} H^1(K_v,A^t) \times \prod_{v \in X^{(1)} \setminus Z} H^1(K_v,A^t)/ H^1_{nrs}(K_v,A^t) \right).
\end{multline*}

Ici, $H^1_{nr}(K_v,A^t)$ désigne $H^1(\mathcal{O}_v,\mathcal{A}_v)=H^1(K_v^{nr}/K_v,A(K_v^{nr}))$.

\begin{remarque}
\begin{itemize}
\item[$\bullet$] L'intersection $Z \cap U$ n'est pas forcément vide.
\item[$\bullet$] La définition de $\Sha^1_{nrs}(A^t)$ est prise de telle sorte que $\Sha^1_{nrs}(A^t) = \bigcup_{V \subseteq U} \Sha^1_{nrs}(V,A^t)$.
\end{itemize}
\end{remarque}

Fixons $V$ un ouvert non vide de $U$.

\begin{lemma}\label{cofC}
\begin{itemize}
\item[(i)] Pour $r>0$, le groupe $H^r(V,\mathcal{A})$ est de torsion de type cofini.
\item[(ii)] Le groupe $H^2_c(V,\mathcal{A})$ est de torsion de type cofini.
\end{itemize}
\end{lemma}

\begin{proof}
\begin{itemize}
\item[(i)] On note $g: \text{Spec} \; K \rightarrow V$. On sait que $\mathcal{A}$ représente le faisceau $g_*A$ sur $V$. On peut alors écrire la suite spectrale de Leray:
$$H^r(V,R^sg_*A) \Rightarrow H^{r+s}(K,A).$$
En calculant les tiges de $R^sg_*A$, on prouve aisément que, pour $s>0$, $R^sg_*A$ est un faisceau de torsion. Comme $V$ est quasi-compact, cela entraîne que $H^r(V,R^sg_*A)$ est de torsion pour $r \geq 0$ et $s >0$. De plus, $H^r(K,A)$ est de torsion pour $r>0$. La suite spectrale entraîne alors que $H^r(V,\mathcal{A})$ est bien de torsion.\\
Reste à prouver que ${_n}H^r(V,\mathcal{A})$ est fini pour chaque $n \geq 1$. La suite exacte:
$$0 \rightarrow {_n}\mathcal{A} \rightarrow \mathcal{A} \rightarrow \mathcal{A} \rightarrow 0.$$
montre que ${_n}H^r(V,\mathcal{A})$ est un quotient de $H^r(V,{_n}\mathcal{A})$. Or ce dernier est fini. Donc ${_n}H^r(V,\mathcal{A})$ est fini, et $H^r(V,\mathcal{A})$ est de torsion de type cofini.
\item[(ii)] On a une suite exacte:
$$...  \rightarrow \bigoplus_{v \in X^{(1)}\setminus V} H^1(K_v,A) \rightarrow H^2_c(V,\mathcal{A}) \rightarrow H^2(V,\mathcal{A}) \rightarrow ...$$
Comme les $H^1(K_v,A)$ sont de torsion de type cofini, grâce à (i), on conclut que $H^2_c(V,\mathcal{A})$ est de torsion de type cofini.
\end{itemize}
\end{proof}

\begin{lemma} \label{exact}
Il existe des suites exactes:
$$ 0 \rightarrow H^0(V,\mathcal{A}^t) \otimes_{\mathbb{Z}} (\mathbb{Q}/\mathbb{Z})\{\ell\}\rightarrow H^1(V,\mathcal{A}^t\{\ell\}) \rightarrow H^1(V,\mathcal{A}^t)\{\ell\} \rightarrow 0 ,$$
$$0 \rightarrow H^1_c(V,\mathcal{A})^{(\ell)}\rightarrow H^2_c(V,T_{\ell}\mathcal{A}) \rightarrow  T_{\ell}H^2_c(V,\mathcal{A}) \rightarrow 0.$$
Ici, $H^1(V,\mathcal{A}^t\{\ell\})$ et $H^2_c(V,T_{\ell}\mathcal{A})$ désignent $\varinjlim_n H^1(V,{_{\ell^n}}\mathcal{A}^t)$ et $\varprojlim_n H^2_c(V,{_{\ell^n}}\mathcal{A})$ respectivement.
\end{lemma}

\begin{proof}
\begin{itemize}
\item[$\bullet$] En utilisant la suite de Kummer, pour chaque entier naturel $r$ on dispose d'une suite exacte: $$0 \rightarrow H^0(V,\mathcal{A}^t)/\ell^r \rightarrow H^1(V,{_{\ell^r}}\mathcal{A}^t) \rightarrow {_{\ell^r}}H^1(V,\mathcal{A}^t) \rightarrow 0.$$ En prenant la limite inductive, on obtient la suite exacte:
$$ 0 \rightarrow H^0(V,\mathcal{A}^t) \otimes_{\mathbb{Z}} (\mathbb{Q}/\mathbb{Z})\{\ell\}\rightarrow H^1(V,\mathcal{A}^t\{\ell\}) \rightarrow H^1(V,\mathcal{A}^t)\{\ell\} \rightarrow 0.$$
\item[$\bullet$] Pour chaque entier naturel $r$ on dispose d'une suite exacte:
 $$0 \rightarrow H^1_c(V,\mathcal{A})/{{\ell^r}} \rightarrow H^2_c(V,{_{\ell^r}}\mathcal{A}) \rightarrow {_{\ell^r}}H^2_c(V,\mathcal{A}) \rightarrow 0.$$
Par passage à la limite projective, on dispose d'une suite exacte:
$$0 \rightarrow H^1_c(V,\mathcal{A})^{(\ell)}\rightarrow H^2_c(V,T_{\ell}\mathcal{A}) \rightarrow  T_{\ell}H^2_c(V,\mathcal{A}) \rightarrow 0.$$
\end{itemize}
\end{proof}

\begin{lemma} \label{acc}
Il existe un accouplement canonique:
$$ H^1(V,\mathcal{A}^t\{\ell\}) \times H^2_c(V,T_{\ell}\mathcal{A}) \rightarrow \mathbb{Q}/\mathbb{Z}$$
qui est non dégénéré.
\end{lemma}

\begin{proof}
Le théorème 2.1 de \cite{Izq1} fournit pour chaque $r \geq 0$ un accouplement parfait de groupes finis:
$$ H^1(V,{_{\ell^r}}\mathcal{A}^t) \times H^2_c(V,{_{\ell^r}}\mathcal{A}) \rightarrow \mathbb{Q}/\mathbb{Z}.$$
Il suffit alors de passer à la limite pour obtenir un accouplement non dégénéré:
$$ H^1(V,\mathcal{A}^t\{\ell\}) \times H^2_c(V,T_{\ell}\mathcal{A}) \rightarrow \mathbb{Q}/\mathbb{Z}.$$
\end{proof}

De plus, d'après la formule de Barsotti-Weil et la nullité de $\underline{\text{Hom}}_V(\mathcal{A},\mathbb{G}_m)$, on a un accouplement canonique $\mathcal{A}^t \otimes^{\mathbf{L}} \mathcal{A} \rightarrow \mathbb{G}_m[1]$ qui induit donc un accouplement:
$$ H^1(V,\mathcal{A}^t) \times H^1_c(V,\mathcal{A}) \rightarrow H^3_c(V,\mathbb{G}_m) \cong \mathbb{Q}/\mathbb{Z}.$$
Posons maintenant: $$D^1(V,\mathcal{A}) = \text{Im}(H^1_c(V,\mathcal{A}) \rightarrow H^1(V,\mathcal{A})) = \text{Ker}(H^1(V,\mathcal{A}) \rightarrow \bigoplus_{v \in X \setminus V} H^1(K_v,A)),$$
$$D^1_{nrs}(V,\mathcal{A}^t) =  \text{Ker}\left( H^1(V,\mathcal{A}^t) \rightarrow \bigoplus_{v \in Z \setminus V} H^1(K_v,A^t) \oplus \bigoplus_{v \in X \setminus (V \cup Z)} H^1(K_v,A^t)/ H^1_{nrs}(K_v,A^t))\right).$$
 Ce sont bien sûr des groupes de torsion de type cofini.

\begin{lemma}\label{Sha}
\begin{itemize}
\item[(i)] L'application naturelle $H^1(V,\mathcal{A}) \rightarrow H^1(K,A)$ induit un isomorphisme $D^1(V,\mathcal{A}) \cong \Sha^1_{nr}(V,A)$.
\item[(ii)] L'application naturelle $H^1(V,\mathcal{A}^t) \rightarrow H^1(K,A^t)$ induit un isomorphisme $$D^1_{nrs}(V,\mathcal{A}^t) \cong \Sha^1_{nrs}(V,A).$$
\end{itemize}
\end{lemma}

\begin{proof}
Soit $g: \text{Spec} \; K \rightarrow V$. La suite spectrale de Leray s'écrit:
$$H^r(V,R^sg_*A) \Rightarrow H^{r+s}(K,A).$$
Cela fournit alors une suite exacte courte:
$$0 \rightarrow H^1(V,\mathcal{A}) \rightarrow H^1(K,A) \rightarrow H^0(V,R^1g_*A).$$
Soit $P$ un ensemble de points géométriques tels que, pour tout $v \in U$, il existe un unique élément de $P$ d'image $v$. On sait alors que $R^1g_*A$ s'injecte dans $\prod_{u \in P} u_*u^*R^1g_*A$ (c'est le premier terme de la résolution de Godement). On en déduit que $H^0(U,R^1g_*A)$ s'injecte dans $\prod_{u \in P} u_*u^*R^1g_*A(U) = \prod_{v \in U} (R^1g_*A)_{\overline{v}} = \prod_{v \in U \setminus \{ \eta \}} H^1(K_v^{nr},A)$. On obtient donc une suite exacte:
$$0 \rightarrow H^1(V,\mathcal{A}) \rightarrow H^1(K,A) \rightarrow \prod_{v \in V^{(1)}} H^1(K_v^{nr},A).$$
De même, on a une suite exacte:
$$0 \rightarrow H^1(V,\mathcal{A}^t) \rightarrow H^1(K,A^t) \rightarrow \prod_{v \in V^{(1)}} H^1(K_v^{nr},A^t).$$
Le lemme découle alors aisément des deux suites précédentes et des suites d'inflation-restriction:
\begin{gather*}
0 \rightarrow H^1_{nr}(K_v,A) \rightarrow H^1(K_v,A) \rightarrow H^1(K_v^{nr},A),\\
0 \rightarrow H^1_{nr}(K_v,A) \rightarrow H^1(K_v,A^t) \rightarrow H^1(K_v^{nr},A^t),
\end{gather*}
pour $v \in V^{(1)}$.
\end{proof}

Afin d'établir un théorème de dualité pour les groupes de Tate-Shafarevich, il convient donc d'établir un théorème de dualité pour les groupes $D^1(U,\mathcal{A})$ et $D^1_{nrs}(V,\mathcal{A}^t)$:

\begin{proposition}\label{d1}
Il existe un accouplement canonique:
$$ \overline{D^1_{nrs}(V,\mathcal{A}^t)}\{\ell\} \times \overline{D^1(V,\mathcal{A})}\{\ell\} \rightarrow  \mathbb{Q}/\mathbb{Z}$$
qui est non dégénéré.
\end{proposition}

Il convient d'établir préalablement le lemme suivant:

\begin{lemma}\label{7aux}
La suite:
$$ \bigoplus_{v \in X^{(1)} \setminus V} H^0(K_v,A)^{(\ell)} \rightarrow H^1_c(V,\mathcal{A})^{(\ell)} \rightarrow  D^1(V,\mathcal{A})^{(\ell)} \rightarrow 0$$
est exacte.
\end{lemma}

\begin{proof} 
Nous disposons d'une suite exacte:
$$ \bigoplus_{v \in X^{(1)} \setminus V} H^0(K_v,A) \rightarrow H^1_c(U,\mathcal{A}) \rightarrow  D^1(U,\mathcal{A}) \rightarrow 0,$$
d'où des suites exactes pour tout $r$:
$$ \bigoplus_{v \in X^{(1)} \setminus V} H^0(K_v,A)/{{\ell^r}} \rightarrow H^1_c(U,\mathcal{A})/\ell^r \rightarrow  D^1(U,\mathcal{A})/\ell^r \rightarrow 0.$$
En passant à la limite projective on obtient l'exactitude de:
$$ \bigoplus_{v \in X^{(1)} \setminus V} H^0(K_v,A)^{(\ell)} \rightarrow H^1_c(V,\mathcal{A})^{(\ell)} \rightarrow  D^1(V,\mathcal{A})^{(\ell)} \rightarrow 0.$$
\end{proof}

\begin{proof} \textit{(De la proposition \ref{d1})}
\begin{itemize}
\item[$\bullet$] Rappelons que, d'après le lemme \ref{acc}, nous disposons d'un accouplement non dégénéré:
$$ H^1(V,\mathcal{A}^t\{\ell\}) \times H^2_c(V,T_{\ell}\mathcal{A}) \rightarrow \mathbb{Q}/\mathbb{Z},$$
d'où un isomorphisme $H^1(V,\mathcal{A}^t\{\ell\}) \rightarrow H^2_c(V,T_{\ell}\mathcal{A})^*$. On dispose aussi d'un accouplement:
$$H^1(V,\mathcal{A}^t)  \times H^1_c(V,\mathcal{A}) \rightarrow \mathbb{Q}/\mathbb{Z}$$
qui induit pour chaque entier naturel $n$ un accouplement:
$${_{\ell^n}}H^1(V,\mathcal{A}^t)  \times H^1_c(V,\mathcal{A})/\ell^n \rightarrow \mathbb{Q}/\mathbb{Z}$$
d'où un accouplement obtenu par passage à la limite:
$$H^1(V,\mathcal{A}^t)\{\ell\}  \times H^1_c(V,\mathcal{A})^{(\ell)} \rightarrow \mathbb{Q}/\mathbb{Z}.$$
Ainsi on obtient un diagramme commutatif à lignes exactes (que l'on appellera diagramme (1)):\\
\centerline{\xymatrix{
 0 \ar[r]& H^0(V,\mathcal{A}^t) \otimes_{\mathbb{Z}} (\mathbb{Q}/\mathbb{Z})\{\ell\} \ar[r]\ar[d]& H^1(V,\mathcal{A}^t\{\ell\}) \ar[r]\ar[d]^{\cong}& H^1(V,\mathcal{A}^t)\{\ell\} \ar[r]\ar[d]& 0\\
 0 \ar[r]& (T_{\ell}H^2_c(V,\mathcal{A}))^D\ar[r]& (H^2_c(V,T_{\ell}\mathcal{A}))^D \ar[r]&  (H^1_c(V,\mathcal{A})^{(\ell)})^D \ar[r]& 0
}} 

De plus, nous disposons aussi d'un autre diagramme commutatif à lignes exactes (que l'on appellera diagramme (2)):\\
\centerline{\xymatrix{
0 \ar[r] & D^1_{nrs}(V,\mathcal{A}^t)\{\ell\}  \ar[d]\ar[r] & H^1(V,\mathcal{A}^t)\{\ell\} \ar[d] \ar[r] & W \ar[d] \\
0 \ar[r] & (D^1(V,\mathcal{A})^{(\ell)})^D   \ar[r] & (H^1_c(V,\mathcal{A})^{(\ell)})^D  \ar[r] & \bigoplus_{v \in X \setminus V} (H^0(K_v,A)^{(\ell)})^D, 
}}
où $W = \bigoplus_{v \in Z \setminus V} H^1(K_v,A^t)\{\ell\} \oplus \bigoplus_{v \in X \setminus (V \cup Z)} H^1(K_v,A^t)\{\ell\}/ H^1_{nrs}(K_v,A^t)\{\ell\}$.
La flèche verticale de droite est un isomorphisme d'après le corollaire \ref{cor local}, le théorème \ref{noyau} et l'hypothèse (H \ref{HC})$_{\ell}$.
\item[$\bullet$] Montrons que $\text{Ker}(D^1_{nrs}(V,\mathcal{A}^t)\{\ell\} \rightarrow (D^1(V,\mathcal{A})^{(\ell)})^D)$ est divisible. En utilisant les diagrammes (1) et (2) et le lemme du serpent, on obtient:
\begin{align*}
\text{Ker}(D^1_{nrs}(V,\mathcal{A}^t)\{\ell\} \rightarrow &(D^1(V,\mathcal{A})^{(\ell)})^D) \cong \text{Ker}( H^1(V,\mathcal{A}^t)\{\ell\} \rightarrow (H^1_c(V,\mathcal{A})^{(\ell)})^D)\\& \cong \text{Coker}(H^0(V,\mathcal{A}^t) \otimes_{\mathbb{Z}} (\mathbb{Q}/\mathbb{Z})\{\ell\} \rightarrow (T_{\ell}H^2_c(V,\mathcal{A}))^D).
\end{align*}
Or le groupe $(T_{\ell}H^2_c(V,\mathcal{A}))^D$ est divisible (puisque $T_{\ell}H^2_c(V,\mathcal{A})$ est un $\mathbb{Z}_{\ell}$-module de type fini sans torsion), et il en est donc de même de $\text{Ker}(D^1_{nrs}(V,\mathcal{A}^t)\{\ell\} \rightarrow (D^1(V,\mathcal{A})^{(\ell)})^D)$.
\item[$\bullet$] Remarquons maintenant que $D^1(V,\mathcal{A})$ est de torsion de type cofini. Cela entraîne que le morphisme naturel $D^1(V,\mathcal{A})\{\ell\} \rightarrow D^1(V,\mathcal{A})^{(\ell)}$ induit un isomorphisme $\overline{D^1(V,\mathcal{A})}\{\ell\} \cong D^1(V,\mathcal{A})^{(\ell)}$. Ce groupe étant fini, le noyau de $D^1_{nrs}(V,\mathcal{A}^t)\{\ell\} \rightarrow (D^1(V,\mathcal{A})^{(\ell)})^D$ est $(D^1_{nrs}(V,\mathcal{A}^t)\{\ell\})_{\text{div}}$, et on a bien un accouplement non dégénéré:
$$ \overline{D^1_{nrs}(V,\mathcal{A}^t)}\{\ell\} \times \overline{D^1(V,\mathcal{A})}\{\ell\} \rightarrow  \mathbb{Q}/\mathbb{Z}$$
\end{itemize}
\end{proof}

Nous sommes maintenant en mesure d'établir le théorème suivant:

\begin{theorem} \label{C((t))(X)} On rappelle que $k=\mathbb{C}((t))$, que $K=k(X)$ est le corps des fonctions de la courbe $X$ et que $V$ est un ouvert non vide de $X$ On suppose (H \ref{HC})$_{\ell}$. Alors il existe un accouplement non dégénéré de groupes de torsion:
$$\overline{\Sha^1_{nrs}(V,A^t)}\{\ell\} \times \overline{\Sha^1_{nr}(V,A)}\{\ell\} \rightarrow \mathbb{Q}/\mathbb{Z}.$$
De plus, $\Sha^1_{nrs}(V,A^t)$ et $\Sha^1_{nr}(V,A)$ sont de torsion de type cofini.
\end{theorem}

\begin{proof}
Cela découle immédiatement de la proposition \ref{d1} et du lemme \ref{Sha}.
\end{proof}

\begin{corollary}\label{cornrs}
On rappelle que $k=\mathbb{C}((t))$ et que $K=k(X)$ est le corps des fonctions de la courbe $X$. On suppose (H \ref{HC})$_{\ell}$. Alors il existe un accouplement non dégénéré de groupes finis:
$$\overline{\Sha^1_{nrs}(A^t)}\{\ell\} \times \overline{\Sha^1(A)}\{\ell\} \rightarrow \mathbb{Q}/\mathbb{Z}.$$
\end{corollary}

\begin{proof}
Pour $V \subseteq V'$ deux ouverts de $U$, on remarque que $\Sha^1_{nr}(V,A)\{\ell\}$ et $ \Sha^1_{nr}(V',A)\{\ell\}$ sont des sous-groupes du groupe de torsion de type cofini $\Sha^1_{nr}(U,A)\{\ell\}$ tels que $\Sha^1_{nr}(V,A)\{\ell\} \subseteq \Sha^1_{nr}(V',A)\{\ell\}$. Comme toute suite décroissante de sous-groupes d'un groupe de torsion de type cofini $\ell$-primaire est stationnaire (Lemme 3.7 de \cite{HS1}), on en déduit qu'il existe un ouvert non vide $V_0$ de $U$ tel que, pour tout ouvert non vide $V$ de $V_0$, on a $\Sha^1_{nr}(V,A)\{\ell\} = \Sha^1_{nr}(V_0,A)\{\ell\}$. Cela implique que $\Sha^1_{nr}(V_0,A)\{\ell\}=\Sha^1(A)\{\ell\}$.\\
Par ailleurs, on remarque que, pour $V \subseteq V'$ deux ouverts non vides de $V_0$, on a un diagramme commutatif:\\
\centerline{\xymatrix{
0 \ar[r] & \Sha^1_{nrs}(V',A^t)\{\ell\}_{div} \ar[r]\ar@{^{(}->}[d] & \Sha^1_{nrs}(V',A^t)\{\ell\} \ar[r] \ar@{^{(}->}[d] & \overline{\Sha^1_{nrs}(V',A^t)}\{\ell\} \ar[d]^{\cong} \ar[r] & 0\\
0 \ar[r] & \Sha^1_{nrs}(V,A^t)\{\ell\}_{div} \ar[r] & \Sha^1_{nrs}(V,A^t)\{\ell\} \ar[r]  & \overline{\Sha^1_{nrs}(V,A^t)}\{\ell\}  \ar[r] & 0
}}
Comme $\Sha^1_{nrs}(A^t)\{\ell\} = \bigcup_{V \subseteq V_0} \Sha^1_{nrs}(V,A^t)\{\ell\}$, en passant à la limite inductive, on obtient que l'injection naturelle $\Sha^1_{nrs}(V_0,A^t)\{\ell\} \hookrightarrow \Sha^1_{nrs}(A^t)\{\ell\}$ induit un isomorphisme $\overline{\Sha^1_{nrs}(V_0,A^t)}\{\ell\} \xrightarrow{\sim} \overline{\Sha^1_{nrs}(A^t)}\{\ell\}$. Par conséquent, d'après le théorème \ref{C((t))(X)}, il existe un accouplement non dégénéré de groupes finis:
$$\overline{\Sha^1_{nrs}(A^t)}\{\ell\} \times \overline{\Sha^1(A)}\{\ell\} \rightarrow \mathbb{Q}/\mathbb{Z}.$$
\end{proof}

On peut aussi obtenir un énoncé symétrique en $A$ et $A^t$:\

\begin{corollary}\label{corglob}
On rappelle que $k=\mathbb{C}((t))$ et que $K=k(X)$ est le corps des fonctions de la courbe $X$. On suppose (H \ref{HC})$_{\ell}$ et on note $i: \Sha^1(A) \hookrightarrow \Sha^1_{nrs}(A)$ (resp. $i^t: \Sha^1(A^t) \hookrightarrow \Sha^1_{nrs}(A^t)$) l'injection canonique. Alors il existe un accouplement non dégénéré de groupes finis:
$$\Sha^1(A^t)\{\ell\}/(i^t)^{-1}(\Sha^1_{nrs}(A^t)\{\ell\}_{div}) \times \Sha^1(A)\{\ell\}/i^{-1}(\Sha^1_{nrs}(A)\{\ell\}_{div}) \rightarrow \mathbb{Q}/\mathbb{Z}.$$
\end{corollary}

\begin{proof}
Il suffit de montrer que le diagramme:\\
\centerline{\xymatrix{
\overline{\Sha^1_{nrs}(A^t)}\{\ell\} \ar@{}[r]|{\times}& \overline{\Sha^1(A)}\{\ell\}\ar[d]^i \ar[r]& \mathbb{Q}/\mathbb{Z}\ar@{=}[d]\\
\overline{\Sha^1(A^t)}\{\ell\}\ar[u]^{i^t} \ar@{}[r]|{\times}& \overline{\Sha^1_{nrs}(A)}\{\ell\} \ar[r]& \mathbb{Q}/\mathbb{Z}
}}
commute. On définit des accouplements $\text{CT}$ et $\text{CT}^t$ par les diagrammes suivants:\\
\centerline{\xymatrix{
\text{CT}: & \overline{\Sha^1(A^t)}\{\ell\} \ar@{}[r]|{\times}& \overline{\Sha^1(A)}\{\ell\}\ar[d]^i \ar[r]& \mathbb{Q}/\mathbb{Z}\ar@{=}[d]\\
&\overline{\Sha^1(A^t)}\{\ell\}\ar@{=}[u] \ar@{}[r]|{\times}& \overline{\Sha^1_{nrs}(A)}\{\ell\} \ar[r]& \mathbb{Q}/\mathbb{Z},
}}
\centerline{\xymatrix{
& \overline{\Sha^1_{nrs}(A^t)}\{\ell\} \ar@{}[r]|{\times}& \overline{\Sha^1(A)}\{\ell\}\ar@{=}[d] \ar[r]& \mathbb{Q}/\mathbb{Z}\ar@{=}[d]\\
\text{CT}^t: &\overline{\Sha^1(A^t)}\{\ell\}\ar[u]^{i^t} \ar@{}[r]|{\times}& \overline{\Sha^1(A)}\{\ell\} \ar[r]& \mathbb{Q}/\mathbb{Z}.
}}
Pour établir le corollaire, il suffit de montrer que $\text{CT}$ et $\text{CT}^t$ coïncident. En procédant comme dans \ref{cornrs}, on choisit un ouvert $V$ de $U$ tel que $D^1(V,\mathcal{A})\{\ell\}=\Sha^1(A)\{\ell\}$ et $D^1(V,\mathcal{A}^t)\{\ell\}=\Sha^1(A^t)\{\ell\}$. Puis en procédant comme pour \ref{d1}, on a des diagrammes commutatifs:\\
\centerline{\xymatrix{
0 \ar[r] & D^1(V,\mathcal{A})\{\ell\}  \ar[d]^j\ar[r] & H^1(V,\mathcal{A})\{\ell\} \ar[d] \ar[r] & \bigoplus_{v \in X \setminus V} H^1(K_v,A)\{\ell\} \ar[d] \\
0 \ar[r] & (D^1(V,\mathcal{A}^t)^{(\ell)})^D   \ar[r] & (H^1_c(V,\mathcal{A}^t)^{(\ell)})^D  \ar[r] & \bigoplus_{v \in X \setminus V} (H^0(K_v,A^t)^{(\ell)})^D, 
}}
\centerline{\xymatrix{
0 \ar[r] & D^1(V,\mathcal{A}^t)\{\ell\}  \ar[d]^{j^t}\ar[r] & H^1(V,\mathcal{A}^t)\{\ell\} \ar[d] \ar[r] & \bigoplus_{v \in X \setminus V} H^1(K_v,A^t)\{\ell\} \ar[d] \\
0 \ar[r] & (D^1(V,\mathcal{A})^{(\ell)})^D   \ar[r] & (H^1_c(V,\mathcal{A})^{(\ell)})^D  \ar[r] & \bigoplus_{v \in X \setminus V} (H^0(K_v,A)^{(\ell)})^D. 
}}
On vérifie alors aisément que $\text{CT}$ est induit par $j$ et que $\text{CT}^t$ est induit par $j^t$. Il suffit donc d'établir le lemme qui suit.
\end{proof}

\begin{lemma}
Soient $r,s \geq 0$. On a un diagramme commutatif:
\begin{equation}
\begin{gathered}
\xymatrix{
H^r_c(V,\mathcal{A}) \ar[d]\ar@{}[r]|{\times} & H^s(V,\mathcal{A}^t)\ar[r] & H^{r+s}_c(V,\mathcal{A}\otimes^{{\mathbf{L}}} \mathcal{A}^t) \ar@{=}[d]  \\
H^r(V,\mathcal{A})  \ar@{}[r]|{\times} & H^s_c(V,\mathcal{A}^t)\ar[u] \ar[r] & H^{r+s}_c(V,\mathcal{A}\otimes^{{\mathbf{L}}} \mathcal{A}^t).
}
\end{gathered}\tag{$\ast$}\label{*}
\end{equation}
\end{lemma}

\begin{proof}
On note $j: V \rightarrow X$ l'immersion ouverte et on fait les identifications suivantes:
\begin{gather*}
H^r_c(V,\mathcal{A}) = \text{Hom}_{D(X)}(\mathbb{Z},j_!\mathcal{A}[r]), \;\;\;\;\; H^s_c(V,\mathcal{A}^t) = \text{Hom}_{D(X)}(\mathbb{Z},j_!\mathcal{A}^t[s]),\\
H^r(V,\mathcal{A}) = \text{Hom}_{D(V)}(\mathbb{Z},\mathcal{A}[r]), \;\;\;\;\; H^s(V,\mathcal{A}^t) = \text{Hom}_{D(V)}(\mathbb{Z},\mathcal{A}^t[s]),\\
H^{r+s}_c(V,\mathcal{A}\otimes^{\mathbf{L}} \mathcal{A}^t)=\text{Hom}_{D(X)}(\mathbb{Z},j_!(\mathcal{A}\otimes^{{\mathbf{L}}} \mathcal{A}^t)[r+s]),
\end{gather*}
où $D(U)$ et $D(X)$ désignent les catégories dérivées de faisceaux étales sur $U$ et sur $X$ respectivement. La commutativité de (\ref{*}) revient alors à montrer que, si $\alpha \in \text{Hom}_{D(X)}(\mathbb{Z},j_!\mathcal{A}[r])$ et $\beta \in  \text{Hom}_{D(X)}(\mathbb{Z},j_!\mathcal{A}^t[s])$, alors le diagramme suivant commute dans $D(X)$:\\
\centerline{\xymatrix{
\mathbb{Z} \ar[r]^-{\alpha}\ar[d]^{\beta} & j_!\mathcal{A}[r] \ar[r]^-{\cong} & (j_!\mathcal{A} \otimes^{{\mathbf{L}}} j_!\mathbb{Z})[r] \ar[d]^{j_!j^*\beta} \\
j_!{\mathcal{A}}^t[s] \ar[r]^-{\cong} & (j_!\mathcal{A}^t \otimes^{{\mathbf{L}}} j_!\mathbb{Z})[s] \ar[r]^-{j_!j^*\alpha} & (j_!\mathcal{A} \otimes^{{\mathbf{L}}} j_!\mathcal{A}^t)[r+s].
}}
Mais cette commutativité est évidente, ce qui achève la preuve.
\end{proof}

\begin{example}
\begin{itemize}
\item[$\bullet$] Les variétés abéliennes ayant bonne réduction partout vérifient les hypothèses des corollaires précédents. C'est par exemple le cas des variétés abéliennes isotriviales.
\item[$\bullet$] Supposons que $X=\mathbb{P}^1_k$, c'est-à-dire que $K=\mathbb{C}((t))(u)$. La courbe elliptique d'équation $y^2=x^3+u$ vérifie les hypothèses des corollaires.
\end{itemize}
\end{example}

\section{\textsc{Variétés abéliennes sur $\mathbb{C}((t_0))...((t_d))$}}\label{Cbis}

On suppose dans cette section que $d \geq 2$ et que $k= \mathbb{C}((t_0))...((t_d))$. Soit $A$ une variété abélienne sur $k$ de variété abélienne duale $A^t$. 

\subsection{Dualité modulo divisibles}

La formule de Barsotti-Weil impose que $A^t=\underline{\text{Ext}}^1_k(A,\mathbb{G}_m)$. De plus, $\underline{\text{Hom}}_k(A,\mathbb{G}_m)=0$. On dispose donc d'un morphisme dans la catégorie dérivée:
$$A \otimes^{\mathbf{L}} A^t \rightarrow \mathbb{G}_m[1],$$
induisant un accouplement:
$$H^r(k,A) \times H^{d-r}(k,A^t) \rightarrow H^{d+1}(k,\mathbb{G}_m) \cong H^{d+1}(k,\mathbb{Q}/\mathbb{Z}(d))\cong \mathbb{Q}/\mathbb{Z}.$$

\begin{lemma}\label{surj}
Pour chaque entier naturel $n$ et chaque entier $r$, le morphisme $ H^{r-1}(k,A)/n \rightarrow ({_n}H^{d+1-r}(k,A^t))^D$ est injectif et le morphisme ${_n}H^{r}(k,A) \rightarrow (H^{d-r}(k,A^t)/n)^D$ est surjectif.
\end{lemma}

\begin{proof}
On a un diagramme commutatif à lignes exactes:\\
\centerline{\xymatrix{
0 \ar[r]& H^{r-1}(k,A)/n \ar[d] \ar[r] & H^r(k,{_n}A) \ar[r]\ar[d] & {_n}H^r(k,A) \ar[r]\ar[d] & 0\\
0 \ar[r]& ({_n}H^{d+1-r}(k,A^t))^D  \ar[r] & H^{d+1-r}(k,{_n}A^t)^D \ar[r] & (H^{d-r}(k,A^t)/n)^D \ar[r] & 0,
}}
où le morphisme vertical central est un isomorphisme d'après le théorème 2.17 de \cite{MilADT} car ${_n}A^t \cong \underline{\text{Hom}}_k({_n}A, \mathbb{Z}/n\mathbb{Z}(d))$. On en déduit que le morphisme $ H^{r-1}(k,A)/n \rightarrow ({_n}H^{d+1-r}(k,A^t))^D$ est injectif et le morphisme ${_n}H^{r}(k,A) \rightarrow (H^{d-r}(k,A^t)/n)^D$ est surjectif.
\end{proof}

\begin{notation}
Pour $r \in \mathbb{Z}$, $n \in \mathbb{N}^*$ et $M$ un $\text{Gal}(k^s/k)$-module discret tel que ${_n}H^r(k,M)$ et $H^r(k,M)/n$ sont finis, on note $\lambda_r(k,n,M) = \frac{|{_n}H^r(k,M)|}{|H^r(k,M)/n|}$.
\end{notation}

\begin{proposition}
Pour $r \geq 0$, il existe des familles d'entiers naturels $(\beta_{r,{\ell}})_{\ell}$, $(\beta_{r,{\ell}}^*)_{\ell}$, $(\beta_{r,{\ell}}^t)_{\ell}$, $(\beta_{0,{\ell}}^{tors})_{\ell}$ et $(\beta_{0,{\ell}}^{t,tors})_{\ell}$ indexées par les nombres premiers telles que, pour tout $n \in \mathbb{N}^*$, on a:
\begin{gather*}
\lambda_r(k,n,A) = \prod_{\ell} {\ell}^{\beta_{r,{\ell}}v_{\ell}(n)}, \\
\lambda_r(k,n,A^t)=\prod_{\ell} {\ell}^{\beta_{r,{\ell}}^tv_{\ell}(n)},\\
\lambda_0(k,n,A(k^s)_{tors})= \prod_{\ell} {\ell}^{\beta_{0,{\ell}}^{tors}v_{\ell}(n)},\\
\lambda_0(k,n,A^t(k^s)_{tors})= \prod_{\ell} {\ell}^{\beta_{0,{\ell}}^{t,tors}v_{\ell}(n)}.
 \end{gather*}
\end{proposition}

\begin{proof}
Les deux dernières égalités sont évidentes car $A(k)_{tors}$ et $A(k)_{tors}$ sont de torsion de type cofini. Montrons les deux premières. Pour $r \geq 1$, elles sont évidentes, puisque les groupes $H^r(k,A)$ et $H^r(k,A^t)$ sont de torsion de type cofini. Le cas $r=0$ découle alors des formules suivantes:
\begin{gather*}
1 = \chi (k, {_n}A) = \prod_{r=0}^{d+1} \lambda_r(k,n,A)^{(-1)^r},\\
1 = \chi (k, {_n}A^t) = \prod_{r=0}^{d+1} \lambda_r(k,n,A^t)^{(-1)^r}.
\end{gather*}
\end{proof}

\begin{theorem}\label{modulo divisibles}
Pour $r \geq 1$, le noyau du morphisme $H^{r}(k,A) \rightarrow (H^{d-r}(k,A^t)^{\wedge})^D$ est un groupe de torsion de type cofini divisible.
\end{theorem}

\begin{proof}
Soit $s \in \{-1,0,...,d+1\}$. On note $N_s$ le noyau de $H^{s}(k,A) \rightarrow (H^{d-s}(k,A^t)^{\wedge})^D$. On calcule la caractéristique d'Euler-Poincaré de ${_n}A$ pour chaque entier naturel $n$:
\begin{align*}
1 & = \chi (k, {_n}A)\\
& = \prod_{r=0}^{s-1} \lambda_r(k,n,A)^{(-1)^r} \cdot \prod_{r=s+1}^{d+1} |H^r(k,{_n}A)|^{(-1)^r} \cdot |{_n}H^{s}(k,A)|^{(-1)^{s}}\\
& = \prod_{r=0}^{s-1} \lambda_r(k,n,A)^{(-1)^r} \cdot \prod_{r=0}^{d-s} |H^{r}(k,{_n}A^t)|^{(-1)^{d+1-r}} \cdot |{_n}H^{s}(k,A)|^{(-1)^{s}}\\
& = \prod_{r=0}^{s-1} \lambda_r(k,n,A)^{(-1)^r} \cdot \prod_{r=0}^{d-s} \lambda_r(k,n,A^t)^{(-1)^{d+1-r}}\cdot |{_n}H^{s}(k,A)|^{(-1)^{s}}|H^{d-s}(k,A^t)/n|^{(-1)^{s+1}}\\
& = |{_n}H^{s}(k,A)|^{(-1)^{s}}|H^{d-s}(k,A^t)/n|^{(-1)^{s+1}} \cdot \prod_{\ell} {\ell}^{(-1)^{s+1}\gamma_{s,{\ell}} v_{\ell}(n)},
\end{align*}
où $\gamma_{s,{\ell}} = \sum_{r=0}^{s-1}(-1)^{r+s+1}\beta_{r,{\ell}}+\sum_{r=0}^{d-s}(-1)^{d+s-r}\beta_{r,{\ell}}^t$. On obtient donc, pour tout $s \in \{-1,0,...,d+1\}$:
$$ |{_n}N_s| = \frac{|{_n}H^{s}(k,A)|}{|H^{d-s}(k,A^t)/n|} = \prod_{\ell} {\ell}^{\gamma_{s,{\ell}} v_{\ell}(n)},$$
ce qui prouve que $N_s$ est divisible à condition que $s\neq 0$.
\end{proof}

En reprenant les notations de la preuve précédente, on a alors:
$$N_s \cong \bigoplus_{\ell} (\mathbb{Q}_{\ell}/\mathbb{Z}_{\ell})^{\gamma_{s,{\ell}}},$$
et nous voulons calculer les $\gamma_{s,{\ell}} = \sum_{r=0}^{s-1}(-1)^{r+s+1}\beta_{r,{\ell}}+\sum_{r=0}^{d-s}(-1)^{d+s-r}\beta_{r,{\ell}}^t$. Avant de passer à la suite, il est utile d'établir des équations reliant les différentes variables que nous avons introduites ($\beta_{r,{\ell}}, \beta_{r,{\ell}}^t,\beta_{0,{\ell}}^{tors},\beta_{0,{\ell}}^{t,tors},\gamma_{r,{\ell}}$).

\begin{proposition}\label{equa0}
Soit ${\ell}$ un nombre premier. Les entiers $(\beta_{r,{\ell}})_r$ et $(\beta^t_{r,{\ell}})_r$ vérifient les équations:\\
$$\left
\{
\begin{array}{c}
\gamma_{r,{\ell}}=\beta_{r,{\ell}} \;\;\;\; \forall r \in \{-1\} \cup \{1,2,...,d-1\} \cup \{d+1\}\\
\gamma_{0,{\ell}} = \beta_{0,{\ell}}^{tors}\\
\beta_{r,{\ell}}=\beta^t_{d+1-r,{\ell}} \;\;\;\; \forall r \in \{2,3,...,d-1\}\\ 
\beta_{1,{\ell}} - \beta_{0,{\ell}} = \beta_{d,{\ell}}^t - \beta_{d+1,{\ell}}^t\\
\beta_{0,{\ell}}^{tors}=\beta_{d+1,{\ell}}^t\\
\sum_{r=0}^{d+1} (-1)^r\beta_{r,{\ell}}=0\\
\beta_{r,{\ell}}=\beta^t_{r,{\ell}}=0\;\;\;\; \forall r \geq d+2 
\end{array}
\right.$$
\end{proposition}

\begin{proof}
Exactement comme dans la démonstration de \ref{modulo divisibles}, on a pour chaque $r \in \{-1,0,1,...,d+1\}$ la relation:
$$1  = |{_n}H^{r}(k,A)|^{(-1)^{r}}|H^{d-r}(k,\tilde{A})/n|^{(-1)^{r+1}} \cdot \prod_p p^{(-1)^{r+1}\gamma_{r,{\ell}} v_p(n)}$$
avec $\gamma_{r,{\ell}} = \sum_{s=0}^{r-1}(-1)^{r+s+1}\beta_{s,{\ell}}+\sum_{s=0}^{d-r}(-1)^{d+r-s}\beta_{s,{\ell}}^t$. Cela montre immédiatement que:
\begin{itemize}
\item[$\bullet$] $\gamma_{{\ell},p}=\beta_{{\ell},p}$ pour $r \in \{1,2,...,d-1\}\cup\{-1,d+1\}$,
\item[$\bullet$] $\beta_{0,{\ell}}^{tors} =\gamma_{0,{\ell}}$,
\item[$\bullet$] $\beta_{d,{\ell}} - \beta_{0,\ell} - \beta_0^{t,tors} =\gamma_{d,{\ell}}$,
\end{itemize}
ce qui achève la preuve.
\end{proof}

\begin{proposition}\label{isogénie}
Pour tout premier ${\ell}$, pour tout entier $r$, on a $\beta_{r,{\ell}} = \beta_{r,{\ell}}^t$. On a aussi $\beta_{0,{\ell}}^{tors} = \beta_{0,{\ell}}^{t,tors}$.
\end{proposition}

\begin{proof}
Les variétés abéliennes $A$ et $A^t$ sont isogènes. Il existe donc une suite exacte $0 \rightarrow F \rightarrow A \rightarrow A^t \rightarrow 0$ où $F$ est un schéma en groupes abélien fini. En passant à la cohomologie, on obtient donc un morphisme $H^r(k,A) \rightarrow H^r(k,A^t)$ de noyau et conoyau finis pour chaque $r$. Cela montre grâce au lemme du serpent que $\beta_{r,{\ell}} = \beta_{r,{\ell}}^t$ pour tout $r$.\\
Pour montrer que $\beta_{0,{\ell}}^{tors} = \beta_{0,{\ell}}^{t,tors}$, on procède de la même façon en remarquant que l'on a un morphisme $A(k)_{tors} \rightarrow A^t(k)_{tors}$ de noyau et conoyau finis.
\end{proof}

Des deux propositions précédentes, on déduit:

\begin{corollary}\label{equa}
Soit ${\ell}$ un nombre premier. Les entiers $(\beta_{r,{\ell}})_r$ vérifient les équations:\\
$$\left
\{
\begin{array}{c}
\beta_{r,{\ell}}=\beta_{d+1-r,{\ell}} \;\;\;\; \forall r \in \{2,3,...,d-1\}\\ 
\beta_{1,{\ell}} - \beta_{0,{\ell}} = \beta_{d,{\ell}} - \beta_{d+1,{\ell}}\\
\beta_{0,{\ell}}^{tors}=\beta_{d+1,{\ell}}\\
\beta_{r,{\ell}}=0\;\;\;\; \forall r \geq d+2 
\end{array}
\right.$$
\end{corollary}

\subsection{Étude de $\beta_{0,{\ell}}$ et de $\beta_{0,{\ell}}^{tors}$}\label{notations}

Pour chaque $i \in \{0,1,...,d\}$, on considère $\mathcal{A}_i$ (resp. $A_i$, $F_i$, $U_i$, $T_i$, $B_i$) un schéma en groupes commutatifs sur $\text{Spec} \; \mathcal{O}_{k_i}$ (resp. sur $\text{Spec} \; k_i$) tels que:
\begin{itemize}
\item[$\bullet$] on a $A_d=A$, $U_d=0$ et $T_d=0$,
\item[$\bullet$] pour $i \in \{0,1,...,d\}$, $B_i$ est une variété abélienne, 
\item[$\bullet$] pour $i \in \{0,1,...,d\}$, $\mathcal{A}_i$ est le modèle de Néron de $B_i$,
\item[$\bullet$] pour $i \in \{0,1,...,d-1\}$, $A_i$ est la fibre spéciale de $\mathcal{A}_{i+1}$,
\item[$\bullet$] pour $i \in \{0,1,...,d\}$, $F_i$ (resp. $U_i$, $T_i$, $B_i$) est le groupe fini (resp. le groupe additif, le tore, la variété abélienne) apparaissant dans la filtration de $A_{i}$.
\end{itemize}
On note aussi $A_{-1}$ la fibre spéciale de $\mathcal{A}_{0}$, et $F_{-1}$, $U_{-1}$, $T_{-1}$, $B_{-1}$ les parties finie, unipotente, torique et abélienne apparaissant dans la filtration de $A_{-1}$. 
De même, pour chaque $i \in \{0,1,...,d\}$, on considère $\mathcal{A}_i^*$ (resp. $A_i^*$, $F_i^*$, $U_i^*$, $T_i^*$, $B_i^*$) un schéma en groupes commutatifs sur $\text{Spec} \; \mathcal{O}_{k_i}$ (resp. sur $\text{Spec} \; k_i$) tels que:
\begin{itemize}
\item[$\bullet$] on a $A_d^*=A^t$, $U_d^*=0$ et $T_d^*=0$,
\item[$\bullet$] pour $i \in \{0,1,...,d\}$, $B_i^*$ est une variété abélienne, 
\item[$\bullet$] pour $i \in \{0,1,...,d\}$, $\mathcal{A}_i^*$ est le modèle de Néron de $B_i^*$,
\item[$\bullet$] pour $i \in \{0,1,...,d-1\}$, $A_i^*$ est la fibre spéciale de $\mathcal{A}_{i+1}^*$,
\item[$\bullet$] pour $i \in \{0,1,...,d\}$, $F_i^*$ (resp. $U_i^*$, $T_i^*$, $B_i^*$) est le groupe fini (resp. le groupe additif, le tore, la variété abélienne) apparaissant dans la filtration de $A_{i}^*$.
\end{itemize}
On note aussi $A_{-1}^*$ la fibre spéciale de $\mathcal{A}_{0}^*$, et $F_{-1}^*$, $U_{-1}^*$, $T_{-1}^*$, $B_{-1}^*$ les parties finie, unipotente, torique et abélienne apparaissant dans la filtration de $A_{-1}^*$.

\begin{proposition} \label{beta0}
\begin{itemize}
\item[(i)] On a $\lambda_0(k,n,A) = \lambda_0(k_{-1},n,B_{-1}) \cdot \prod_{r=-1}^{d-1} \lambda_0(k_r,n,T_r)$.
\item[(ii)] Le nombre $\frac{\lambda_0(k_{-1},n,B_{-1}(k_{-1}^s)_{tors}) \cdot \prod_{r=-1}^{d-1} \lambda_0(k_r,n,T_r(k_r^s)_{tors})}{\lambda_0(k,n,A(k^s)_{tors})}$ est entier. Si $B_i$ est à réduction scindée pour $i \in \{1,...,d\}$, alors $$\lambda_0(k,n,A(k^s)_{tors})=\lambda_0(k_{-1},n,B_{-1}(k_{-1}^s)_{tors}) \cdot \prod_{r=-1}^{d-1} \lambda_0(k_r,n,T_r(k_r^s)_{tors}).$$
\end{itemize}
\end{proposition}

\begin{proof}
\begin{itemize}
\item[(i)] Exactement comme dans le théorème \ref{1-local}, on montre que:
$$\lambda_0(k,n,A) = \lambda_0(k_{d-1},n,B_{d-1})\lambda_0(k_{d-1},n,T_{d-1})\lambda_0(k_{d-1},n,U_{d-1}).$$
Bien sûr, $\lambda_0(k_{d-1},n,U_{d-1}) = 1$, et donc:
$$\lambda_0(k,n,A) = \lambda_0(k_{d-1},n,B_{d-1})\lambda_0(k_{d-1},n,T_{d-1}).$$
Il suffit alors de procéder par récurrence.
\item[(ii)] Comme le morphisme $\mathcal{A}_d(\mathcal{O}_k) \rightarrow A_{d-1}(k_{d-1})$ est surjectif à noyau uniquement divisible, on a $\lambda_0(k,n,A(k^s)_{tors}) = \lambda_0(k_{d-1},n,A_{d-1}(k_{d-1}^s)_{tors})$. On procède maintenant par dévissage.
\begin{itemize}
\item[$\bullet$] Comme $0 \rightarrow A^0_{d-1} \rightarrow A_{d-1} \rightarrow F_{d-1} \rightarrow 0$ est exacte, le morphisme $A^0_{d-1}(k_{d-1})_{tors} \rightarrow A_{d-1}(k_{d-1})_{tors}$ est injectif à conoyau fini, et donc: $$\lambda_0(k_{d-1},n,A_{d-1}(k_{d-1}^s)_{tors}) = \lambda_0(k_{d-1},n,A_{d-1}^0(k_{d-1}^s)_{tors}).$$
\item[$\bullet$] On a une suite exacte $0 \rightarrow U_{d-1} \times T_{d-1} \rightarrow A^0_{d-1} \rightarrow B_{d-1} \rightarrow 0$. Comme $U_{d-1}(k_{d-1}^s)$ et $T_{d-1}(k_{d-1}^s)$ sont divisibles, on en déduit l'exactitude de $0 \rightarrow T_{d-1}(k_{d-1}^s)_{tors} \rightarrow A^0_{d-1}(k_{d-1}^s)_{tors} \rightarrow B_{d-1}(k_{d-1}^s)_{tors} \rightarrow 0$. En passant à la cohomologie, on obtient une suite exacte:
$$0 \rightarrow T_{d-1}(k_{d-1})_{tors} \rightarrow A^0_{d-1}(k_{d-1})_{tors} \rightarrow B_{d-1}(k_{d-1})_{tors} \rightarrow H^1(k_{d-1},T_{d-1}(k_{d-1}^s)_{tors}). $$
Donc $\lambda_0(k_{d-1},n,A_{d-1}^0(k_{d-1}^s)_{tors})$ divise $\lambda_0(k_{d-1},n,T_{d-1}(k_{d-1}^s)_{tors}) \lambda_0(k_{d-1},n,B_{d-1}(k_{d-1}^s)_{tors})$. 
\end{itemize}
En procédant par récurrence, $\lambda_0(k_{d-1},n,A_{d-1}^0(k_{d-1}^s)_{tors})$ divise $\lambda_0(k_{-1},n,B_{-1}(k_{-1}^s)_{tors}) \cdot \prod_{r=-1}^{d-1} \lambda_0(k_r,n,T_r(k_{r}^s)_{tors})$.\\
Si $A=B_d$ est à réduction scindée, on remarque que la flèche $B_{d-1}(k_{d-1})_{tors} \rightarrow H^1(k_{d-1},T_{d-1}(k_{d-1}^s)_{tors})$ est nulle, et on a donc une suite exacte courte:
$$0 \rightarrow T_{d-1}(k_{d-1})_{tors} \rightarrow A^0_{d-1}(k_{d-1})_{tors} \rightarrow B_{d-1}(k_{d-1})_{tors} \rightarrow 0.$$\\
Donc en procédant par récurrence, si $B_i$ est à réduction scindée pour $i \in \{1,...,d\}$, on obtient l'égalité: $$\lambda_0(k,n,A(k^s)_{tors})=\lambda_0(k_{-1},n,B_{-1}(k_{-1}^s)_{tors}) \cdot \prod_{r=-1}^{d-1} \lambda_0(k_r,n,T_r(k_r^s)_{tors}).$$
\end{itemize}
\end{proof}

\begin{corollary}\label{0}
Notons $\rho_r$ le rang du tore $T_r$ pour $r \in \{-1,0,...,d-1\}$. On a pour tout premier ${\ell}$:
$$\beta_{0,{\ell}} = 2\text{dim} B_{-1} - \sum_{r=-1}^{d-1} r\rho_r,$$
$$\beta_{0,{\ell}}^{tors}=\beta_{d+1,{\ell}}\leq 2\text{dim} B_{-1} + \sum_{r=-1}^{d-1} \rho_r.$$
Si $B_i$ est à réduction scindée pour $i \in \{1,...,d\}$, alors pour tout premier $p$: $$\beta_{0,{\ell}}^{tors} = \beta_{d+1,{\ell}}= 2\text{dim} B_{-1} + \sum_{r=-1}^{d-1} \rho_r.$$
\end{corollary}

\begin{proof}
Cela découle de la proposition \ref{beta0}, du corollaire \ref{equa} et des propositions \ref{tor0} et \ref{tor0tors}.
\end{proof}

\begin{corollary}
On a $2\text{dim} B_{-1} - \sum_{r=-1}^{d-1} r\rho_r = 2\text{dim} B_{-1}^* - \sum_{r=-1}^{d-1} r\rho_r^*$. Si $B_i$ et $B_i^*$ sont à réduction scindée pour $i \in \{1,...,d\}$, alors $2\text{dim} B_{-1} + \sum_{r=-1}^{d-1} \rho_r = 2\text{dim} B_{-1}^* + \sum_{r=-1}^{d-1} \rho_r^*$.
\end{corollary}

\begin{proof}
Cela découle immédiatement du corollaire \ref{0} et du lemme \ref{isogénie}.
\end{proof}

\begin{remarque}
Plus généralement, la quantité $2\text{dim} B_{-1} - \sum_{r=-1}^{d-1} r\rho_r$ est invariante par isogénie.
\end{remarque}

\subsection{Majorations des $\beta_{r,\ell}$}

\begin{lemma}\label{tech}
Les parties divisibles de $(\varprojlim_m H^{d-r}(k_{d-1},{_m}A^t(k^{nr})))^D$ et de $H^{d-r}(k_{d-1},A(k^{nr})_{tors})$ sont (non canoniquement) isomorphes.
\end{lemma}

\begin{proof}
Fixons un nombre premier $\ell$. Pour chaque entier naturel $s$, on a une suite exacte:
$$0 \rightarrow {_{\ell^s}}A^t(k^{nr}) \rightarrow A^t(k^{nr})_{tors} \rightarrow A^t(k^{nr})_{tors} \rightarrow A^t(k^{nr})_{tors}/\ell^s \rightarrow 0.$$
En notant $Q_{\ell^s}$ le groupe $\ell^sA^t(k^{nr})_{tors}$, on a des suites exactes:
$$0 \rightarrow {_{\ell^s}}A^t(k^{nr}) \rightarrow A^t(k^{nr})_{tors} \rightarrow Q_{\ell^s} \rightarrow 0,$$
$$0\rightarrow Q_{\ell^s} \rightarrow A^t(k^{nr})_{tors} \rightarrow A^t(k^{nr})_{tors}/\ell^s \rightarrow 0.$$
On a alors un diagramme commutatif à colonne exacte dont les flèches diagonales sont la multiplication par $\ell^s$:\\
\centerline{\xymatrix{
H^{d-r-1}(k_{d-1},A^t(k^{nr})_{tors}) \ar[d]\ar[dr]^{\ell^s} & \\
H^{d-r-1}(k_{d-1},Q_{\ell^s}) \ar[d]\ar[r] & H^{d-r-1}(k_{d-1},A^t(k^{nr})_{tors})\\
H^{d-r}(k_{d-1}, {_{\ell^s}}A^t(k^{nr})) \ar[d]&\\
H^{d-r}(k_{d-1},A^t(k^{nr})_{tors}) \ar[d]\ar[dr]^{\ell^s}&\\
H^{d-r}(k_{d-1},Q_{\ell^s})\ar[r] & H^{d-r}(k_{d-1},A^t(k^{nr})_{tors}).
}}
Quand on fait varier $s$, le module galoisien $A^t(k^{nr})_{tors}/\ell^s$ ne peut prendre qu'un nombre fini de valeurs à isomorphisme près. Par conséquent, il existe une constante entière $C_{\ell}>0$ telle que, pour tout $s \geq 0$:
\begin{gather*}
|\text{Ker}(H^{d-r-1}(k_{d-1},Q_{\ell^s}) \rightarrow H^{d-r-1}(k_{d-1},A^t(k^{nr})_{tors}))| <C_{\ell},\\
|\text{Coker}(H^{d-r-1}(k_{d-1},Q_{\ell^s}) \rightarrow H^{d-r-1}(k_{d-1},A^t(k^{nr})_{tors}))| <C_{\ell},\\
|\text{Ker}(H^{d-r}(k_{d-1},Q_{\ell^s}) \rightarrow H^{d-r}(k_{d-1},A^t(k^{nr})_{tors}))| <C_{\ell},\\
|\text{Coker}(H^{d-r}(k_{d-1},Q_{\ell^s}) \rightarrow H^{d-r}(k_{d-1},A^t(k^{nr})_{tors}))| <C_{\ell}.
\end{gather*}
On en déduit que:
\begin{gather*} |\text{Coker}(H^{d-r}(k_{d-1}, {_{\ell^s}}A^t(k^{nr})) \rightarrow {_{\ell^s}}H^{d-r}(k_{d-1},A^t(k^{nr})_{tors}))| \leq |{_{\ell^s}}H^{d-r}(k_{d-1},A^t(k^{nr})_{tors})/C_{\ell}!|,\\
|\text{Ker}(H^{d-r}(k_{d-1}, {_{\ell^s}}A^t(k^{nr})) \rightarrow {_{\ell^s}}H^{d-r}(k_{d-1},A^t(k^{nr})_{tors}))| \leq C_{\ell}|H^{d-r-1}(k_{d-1},A^t(k^{nr})_{tors})/\ell^s|,
\end{gather*}
d'où l'existence d'une contante $D_{\ell}$ telle que, pour tout $s \geq 0$:
\begin{gather*}
|\text{Coker}(H^{d-r}(k_{d-1}, {_{\ell^s}}A^t(k^{nr})) \rightarrow {_{\ell^s}}H^{d-r}(k_{d-1},A^t(k^{nr})_{tors}))| \leq D_{\ell},\\
|\text{Ker}(H^{d-r}(k_{d-1}, {_{\ell^s}}A^t(k^{nr})) \rightarrow {_{\ell^s}}H^{d-r}(k_{d-1},A^t(k^{nr})_{tors}))| \leq D_{\ell}.
\end{gather*}
On en déduit que les groupes: 
\begin{gather*}
\text{Ker}(\varprojlim_s H^{d-r}(k_{d-1}, {_{\ell^s}}A^t(k^{nr})) \rightarrow \varprojlim_s{_{\ell^s}}H^{d-r}(k_{d-1},A^t(k^{nr})_{tors})),\\
\text{Coker}(\varprojlim_s H^{d-r}(k_{d-1}, {_{\ell^s}}A^t(k^{nr})) \rightarrow \varprojlim_s{_{\ell^s}}H^{d-r}(k_{d-1},A^t(k^{nr})_{tors})),
\end{gather*}
sont finis. Cela montre que les parties divisibles de $(\varprojlim_m H^{d-r}(k_{d-1},{_m}A^t(k^{nr})))^D$ et de $H^{d-r}(k_{d-1},A^t(k^{nr})_{tors})$ sont isomorphes. Comme $A$ et $A^t$ sont isogènes, les parties divisibles de $H^{d-r}(k_{d-1},A^t(k^{nr})_{tors})$ et de $H^{d-r}(k_{d-1},A(k^{nr})_{tors})$ sont isomorphes, ce qui achève la preuve.
\end{proof}

\begin{lemma}\label{lemme1}
Pour chaque entier naturel $r <d-1$, les parties divisibles de groupes de type cofini $H^{r}(k_{d-1},H^1(k^{nr},A))$ et $H^{d-r}(k_{d-1},A_{d-1}^0(k_{d-1}^s))$ sont (non canoniquement) isomorphes. Pour $r = d-1$ et $r=d$, les parties divisibles de groupes de type cofini $H^{r}(k_{d-1},H^1(k^{nr},A))$ et $H^{d-r}(k_{d-1},A_{d-1}^0(k_{d-1}^s)_{tors})$ sont (non canoniquement) isomorphes.
\end{lemma}

\begin{proof}
Soit $r\geq 0$. D'après \cite{Ogg}, on a un isomorphisme $H^1(k^{nr},A) \cong (\varprojlim_m {_m}A^t(k^{nr}))^D$.
On calcule alors:
\begin{align*}
H^{r}(k_{d-1},H^1(k^{nr},A)) &= H^{r}(k_{d-1},(\varprojlim_m {_m}A^t(k^{nr}))^D)\\
& \cong \varinjlim_m H^{r}(k_{d-1}, {_m}A^t(k^{nr})^D)\\
& \cong \varinjlim_m H^{d-r}(k_{d-1},{_m}A^t(k^{nr}))^D\\
& \cong (\varprojlim_m H^{d-r}(k_{d-1},{_m}A^t(k^{nr})))^D.
\end{align*}
Par conséquent, les parties divisibles de $H^{r}(k_{d-1},H^1(k^{nr},A))$ et de $H^{d-r}(k_{d-1},A(k^{nr})_{tors}) \cong H^{d-r}(k_{d-1},A_{d-1}(k_{d-1}^s)_{tors})$ sont isomorphes. \\
On remarque maintenant que l'on a la suite exacte:
$$0 \rightarrow A_{d-1}^0(k_{d-1}^s)  \rightarrow A_{d-1}(k_{d-1}^s) \rightarrow F_{d-1}(k_{d-1}^s) \rightarrow 0,$$
où $A_{d-1}^0$ désigne la composante connexe du neutre dans $A_{d-1}$. Il existe donc un $\text{Gal}(k_{d-1}^s/k_{d-1})$-module fini $F$ tel que:
$$0 \rightarrow A_{d-1}^0(k_{d-1}^s)_{tors}  \rightarrow A_{d-1}(k_{d-1}^s)_{tors} \rightarrow F \rightarrow 0.$$
On en déduit que les parties divisibles des groupes de type cofini $H^{d-r}(k_{d-1},A_{d-1}(k_{d-1}^s)_{tors})$ et $H^{d-r}(k_{d-1},A_{d-1}^0(k_{d-1}^s)_{tors})$ sont (non canoniquement) isomorphes. \\
Supposons maintenant que $r < d-1$. On a une suite exacte:
$$0 \rightarrow U_{d-1}(k_{d-1}^s) \times T(k_{d-1}^s) \rightarrow A_{d-1}^0(k_{d-1}^s) \rightarrow B_{d-1}(k_{d-1}^s) \rightarrow 0,$$
qui montre que $A_{d-1}^0(k_{d-1}^s)$ est divisible. On en déduit que $A_{d-1}^0(k_{d-1}^s)/A_{d-1}^0(k_{d-1}^s)_{tors}$ est uniquement divisible, et donc que $H^{d-r}(k_{d-1},A_{d-1}^0(k_{d-1}^s)_{tors})\cong H^{d-r}(k_{d-1},A_{d-1}^0(k_{d-1}^s))$, ce qui achève la preuve.
\end{proof}

\begin{lemma}\label{lemme2}
Pour chaque entier $r>0$, on a un isomorphisme $H^r(k_{d-1},A(k^{nr})) \cong H^r(k_{d-1},A_{d-1})$.
\end{lemma}

\begin{proof}
Comme l'extension $k^{nr}/k$ est non ramifiée par définition, on sait que $A(k^{nr}) = \mathcal{A}_d(\mathcal{O}_{k^{nr}})$ et que le morphisme $\mathcal{A}_d(\mathcal{O}_{k^{nr}}) \rightarrow A_{d-1}(k_{d-1}^s)$ est surjectif de noyau uniquement divisible. On en déduit que $H^r(k_{d-1},A(k^{nr})) \cong H^r(k_{d-1},A_{d-1})$.
\end{proof}

\begin{theorem}\label{majoration}
On a, pour $r \geq 2$:
$$\beta_{r,{\ell}} \leq  \binom{d+1}{r} \left( \sum_{e \geq 0} \rho_{e-1} + 2 \dim B_{-1}\right) .$$
Pour $r=1$:
$$\beta_{r,{\ell}} \leq \sum_{e\geq 0} (d+1-e)\rho_{e-1} + 2(d+1)\dim B_{-1}.$$
\end{theorem}

\begin{proof}
Procédons par récurrence sur $d+r$. \\
Pour $d+r=0$ (c'est-à-dire $d=-1$ et $r=1$), le théorème est évident.\\
Soit $s \geq 0$ tel que le théorème est vrai pour $r$ et $d$ tels que $r+d\leq s$. Soient $r\geq 1$ et $d \geq -1$ des entiers tels que $d+r=s+1$. La suite spectrale $H^r(k_{d-1},H^s(k^{nr},A)) \Rightarrow H^{r+s}(k,A)$ dégénère en une suite exacte longue:
$$... \rightarrow H^r(k_{d-1},A(k^{nr})) \rightarrow H^r(k,A) \rightarrow H^{r-1}(k_{d-1},H^1(k^{nr},A)) \rightarrow ... \text{.}$$
Étudions les termes $H^r(k_{d-1},A(k^{nr}))$ et $H^{r-1}(k_{d-1},H^1(k^{nr},A))$.
\begin{itemize}
\item[$\bullet$] D'après \ref{lemme2}, la partie divisible de $H^r(k_{d-1},A(k^{nr}))$ est isomorphe à celle de $H^r(k_{d-1},A^0_{d-1})$. De plus, la suite exacte $0 \rightarrow U_{d-1} \times T_{d-1} \rightarrow A_{d-1}^0 \rightarrow B_{d-1} \rightarrow 0$ montre l'exactitude de:
$$ H^{r}(k_{d-1},T_{d-1}) \rightarrow H^{r}(k_{d-1},A_0^{d-1}) \rightarrow H^{r}(k_{d-1},B_{d-1}),$$
et on a vu dans \ref{torr} que $\lambda_r(k_{d-1},n,T_{d-1}))$ vaut $n^{\binom{d}{r}\rho_{d-1}}$ si $r>1$ et $1$ si $r=1$.
\item[$\bullet$] Supposons que $r<d$. Alors, d'après \ref{lemme1}, la partie divisible de $H^{r-1}(k_{d-1},H^1(k^{nr},A))$ est isomorphe à celle de $H^{d-r+1}(k_{d-1},A^0_{d-1})$, et  $H^{d-r+1}(k_{d-1},A^0_{d-1})$ s'insère dans une suite exacte:
$$ H^{d-r+1}(k_{d-1},T_{d-1}) \rightarrow H^{d-r+1}(k_{d-1},A_{d-1}^0) \rightarrow H^{d-r+1}(k_{d-1},B_{d-1}),$$
où $\lambda_{d-r+1}(k_{d-1},n,T_{d-1}))=n^{\binom{d}{r-1}\rho_{d-1}}$ d'après \ref{torr}.
\item[$\bullet$] D'après \ref{lemme1}, la partie divisible de $H^{d-1}(k_{d-1},H^1(k^{nr},A))$ est isomorphe à celle de $H^{1}(k_{d-1},A^0_{d-1}(k_{d-1}^s)_{tors})$. L'exactitude de la suite $0 \rightarrow T_{d-1}(k_{d-1}^s)_{tors} \rightarrow A^0_{d-1}(k_{d-1}^s)_{tors} \rightarrow B_{d-1}(k_{d-1}^s)_{tors} \rightarrow 0$ entraîne l'exactitude de la suite:
$$H^1(k_{d-1}, T_{d-1}(k_{d-1}^s)_{tors}) \rightarrow H^1(k_{d-1},A^0_{d-1}(k_{d-1}^s)_{tors}) \rightarrow H^1(k_{d-1},B_{d-1}(k_{d-1}^s)_{tors}),$$
où on a $\lambda_{1}(k_{d-1},n,T_{d-1}(k_{d-1}^s)_{tors}))=n^{d\rho_{d-1}}$ d'après \ref{tor1tors}. Pour calculer $\lambda_{1}(k_{d-1},n,B_{d-1}(k_{d-1}^s)_{tors}))$, on remarque que l'on a une suite exacte:
$$0 \rightarrow B_{d-1}(k_{d-1})_{tors} \rightarrow B_{d-1}(k_{d-1}) \rightarrow D \rightarrow H^1(k_{d-1}, B_{d-1}(k_{d-1}^s)_{tors}) \rightarrow H^1(k_{d-1}, B_{d-1}) \rightarrow 0,$$
où $D$ est uniquement divisible. Donc, en utilisant \ref{0}, le nombre $\lambda_{1}(k_{d-1},n,B_{d-1}(k_{d-1}^s)_{tors}))$ divise $n^{\sum_{e=0}^{d-1} e\rho_{e-1}}\lambda_1(k_{d-1},n,B_{d-1})$.
\item[$\bullet$] D'après \ref{lemme1}, la partie divisible de $H^{d}(k_{d-1},H^1(k^{nr},A))$ est isomorphe à celle de $A^0_{d-1}(k_{d-1})_{tors}$. L'exactitude de la suite $0 \rightarrow T_{d-1}(k_{d-1}^s)_{tors} \rightarrow A^0_{d-1}(k_{d-1}^s)_{tors} \rightarrow B_{d-1}(k_{d-1}^s)_{tors} \rightarrow 0$ entraîne l'exactitude de la suite:
$$0 \rightarrow T_{d-1}(k_{d-1})_{tors} \rightarrow A^0_{d-1}(k_{d-1})_{tors} \rightarrow B_{d-1}(k_{d-1})_{tors},$$
où on a $\lambda_{0}(k_{d-1},n,T_{d-1}(k_{d-1}^s)_{tors}))=n^{\rho_{d-1}}$ d'après \ref{tor0tors} et $\lambda_{0}(k_{d-1},n,B_{d-1}(k_{d-1}^s)_{tors}))$ divise $n^{\sum_{e =0}^{d-1} \rho_{e-1} + 2 \dim B_{-1}}$ d'après \ref{0}.
\end{itemize}
On obtient donc, par hypothèse de récurrence:
\begin{itemize}
\item[$\bullet$] si $r=1 $:
\begin{align*}
\beta_{1,{\ell}} \leq   \sum_{e = 0}^{d-1} (d-e)\rho_{e-1} + 2d \dim B_{-1}+ \rho_{d-1}+ \sum_{e = 0}^{d-1} \rho_{e-1} + 2 \dim B_{-1}
\end{align*}
et donc:
$$\beta_{1,{\ell}} \leq  \sum_{e\geq 0} (d+1-e)\rho_{e-1} + 2(d+1)\dim B_{-1}.$$
\item[$\bullet$] si $1<r<d $:
\begin{align*}
\beta_{r,{\ell}} \leq & \binom{d}{r}\rho_{d-1} + \binom{d}{r} \left( \sum_{e = 0}^{d-1} \rho_{e-1} + 2 \dim B_{-1}\right)\\
& + \binom{d}{r-1}\rho_{d-1} + \binom{d}{r-1} \left( \sum_{e = 0}^{d-1} \rho_{e-1} + 2 \dim B_{-1}\right)
\end{align*}
et donc:
$$\beta_{r,{\ell}} \leq  \binom{d+1}{r} \left( \sum_{e \geq 0} \rho_{e-1} + 2 \dim B_{-1}\right).$$
\item[$\bullet$] pour $r=d$:
\begin{align*}
\beta_{d,{\ell}} \leq & \rho_{d-1} + \left( \sum_{e = 0}^{d-1} \rho_{e-1} + 2 \dim B_{-1}\right)\\
& + d\rho_{d-1} + \sum_{e=0}^{d-1} e\rho_{e-1} + \sum_{e= 0}^{d-1} (d-e)\rho_{e-1} + 2d\dim B_{-1}
\end{align*}
 et donc:
 $$\beta_{d,{\ell}} \leq (d+1) \left( \sum_{e \geq 0} \rho_{e-1} + 2 \dim B_{-1}\right).$$
\item[$\bullet$] pour $r=d+1$:
\begin{align*}
\beta_{d+1,{\ell}} &\leq \rho_{d-1} + \sum_{e =0}^{d-1} \rho_{e-1} + 2 \dim B_{-1}\\
&\leq \sum_{e\geq 0} \rho_{e-1} + 2\dim B_{-1}.
\end{align*}
\end{itemize}
\end{proof}

\begin{corollary}\label{cormaj}
Si $\dim B_{-1}=\rho_{-1}=\rho_0=...=\rho_{d-1}=0$, alors $\beta_{0,{\ell}}^{tors}=0$ et $\beta_{r,{\ell}} = 0$ pour tout $r \geq 0$ et tout premier ${\ell}$.
\end{corollary}

\begin{proof}
Cela découle immédiatement du théorème \ref{majoration} et du corollaire \ref{0}.
\end{proof}

\subsection{Nullité des $\beta_{r,\ell}$}

\begin{theorem}\label{nullité}
Soit ${\ell}$ un nombre premier. Supposons que $\beta_{0,{\ell}}^{tors}=0$. Alors $\dim B_{-1}=\rho_{-1}=\rho_0=\rho_1=...=\rho_{d-1}=0$. 
\end{theorem}

\begin{proof}
On procède par récurrence sur $d$.
\begin{itemize}
\item[$\bullet$] Dans le cas $d=0$, le théorème est évident puisqu'on a $\beta_{0,\ell}^{tors}=2\dim B_{-1}+\rho_{-1}$ d'après le corollaire \ref{0}.
\item[$\bullet$] Supposons maintenant la propriété démontrée au rang $d-1$. Montrons-la au rang $d$. On suppose donc que $k$ est $d$-local et que $\beta_{0,\ell}^{tors}=0$. On a $A(k)_{tors} \cong \mathcal{A}_d(\mathcal{O}_k)_{tors} \cong A_{d-1}(k_{d-1})_{tors}$ puisque le morphisme $\mathcal{A}_d(\mathcal{O}_k) \rightarrow A_{d-1}(k_{d-1})$ est surjectif à noyau uniquement divisible. Par conséquent, la partie $p$-primaire de $ A_{d-1}(k_{d-1})_{tors}$ est finie. Comme la suite $0 \rightarrow T_{d-1}(k_{d-1}^s)_{tors} \rightarrow A_0^{d-1}(k_{d-1}^s)_{tors} \rightarrow B_{d-1}(k_{d-1}^s)_{tors} \rightarrow 0$ est exacte, on a une suite de cohomologie:
$$0 \rightarrow T_{d-1}(k_{d-1})_{tors} \rightarrow A_{d-1}^0(k_{d-1})_{tors} \rightarrow B_{d-1}(k_{d-1})_{tors} \rightarrow H^1(k_{d-1},T_{d-1}(k_{d-1}^s)_{tors}).$$
On en déduit que la partie $\ell$-primaire de $T_{d-1}(k_{d-1})_{tors}$ et le noyau de $B_{d-1}(k_{d-1})\{{\ell}\} \rightarrow H^1(k_{d-1},T_{d-1}(k_{d-1}^s)_{tors})\{{\ell}\}$ sont finis. Comme $\lambda_0(k_{d-1},n,T_{d-1}(k_{d-1}^s)_{tors})=n^{\rho_{d-1}}$ d'après \ref{tor0tors}, on en déduit que $\rho_{d-1}=0$. En utilisant que le noyau de $B_{d-1}(k_{d-1})\{{\ell}\} \rightarrow H^1(k_{d-1},T_{d-1}(k_{d-1}^s)_{tors})\{{\ell}\}$ est fini et en remarquant que $\lambda_1(k_{d-1},n,T_{d-1}(k_{d-1}^s)_{tors}) = n^{d\rho_{d-1}}=1$ d'après \ref{tor1tors}, on obtient que $B_{d-1}(k_{d-1})\{{\ell}\}$ est fini. Par hypothèse de récurrence, cela impose que $\dim B_{-1}=\rho_{-1}=\rho_0=\rho_1=...=\rho_{d-2}=0$, ce qui achève la preuve.
\end{itemize}
\end{proof}

\begin{corollary}
\begin{itemize}
\item[(i)] Soit ${\ell}$ un nombre premier. Supposons que $\beta_{d+1,{\ell}}=0$. Alors $\beta_{r,l}=0$ pour tout entier $r$ et tout premier $l$.
\item[(ii)] Supposons que $\dim B_{-1} = 0$ et que $\rho_r = 0$ pour tout entier $r$. Alors $\dim B_{-1}^*=0$ et $\rho_r^* = 0$ pour tout $r$.
\end{itemize}
\end{corollary}

\begin{proof}
\begin{itemize}
\item[(i)] Cela découle immédiatement du théorème \ref{nullité} et des corollaires \ref{cormaj} et \ref{equa}.
\item[(ii)] D'après le corollaire \ref{cormaj}, $\beta_{d+1,{\ell}}$ est nul pour tout premier ${\ell}$. On déduit alors du lemme \ref{isogénie} que $\beta_{d+1,{\ell}}^t$ est nul pour tout premier ${\ell}$. Le théorème \ref{nullité} et le corollaire \ref{equa} permettent donc de conclure.
\end{itemize}
\end{proof}

\begin{remarque}
Plus généralement, la propriété que $\dim B_{-1}=\rho_{-1}=\rho_0=...=\rho_{d-1}=0$ est préservée par les isogénies.
\end{remarque}

\subsection{Le noyau de $H^d(k,A) \rightarrow (H^0(k,A^t)^{\wedge})^D$}

Dans certains cas, il est possible d'expliciter le noyau de $H^d(k,A) \rightarrow (H^0(k,A^t)^{\wedge})^D$. Pour ce faire, il convient d'établir quelques propriétés préliminaires:

\begin{lemma}\label{cohonrbis}
Pour chaque $r\geq 0$, on a un isomorphisme $H^r(\mathcal{O}_k,\mathcal{A}_d) \rightarrow H^r(k^{nr}/k,A(k^{nr}))$ faisant commuter le diagramme:\\
\centerline{\xymatrix{
H^r(\mathcal{O}_k,\mathcal{A}_d) \ar[r]^{\text{Res}} \ar[d]^{\cong} & H^r(k,A)\\
H^r(k^{nr}/k,A(k^{nr})) \ar[ru]_{\text{Inf}} & .
}}
\end{lemma}

\begin{proof}
La preuve est analogue à celle de \ref{cohonr}.
\end{proof}

\begin{proposition}\label{prelCbis}
Soit $\ell$ un nombre premier ne divisant pas $|F_{d-1}|$.
\begin{itemize}
\item[(i)] Les groupes $A(k^{nr})$ et $A^t(k^{nr})$ sont $\ell$-divisibles.
\item[(ii)] Il existe un morphisme fonctoriel injectif $$(T_{\ell}H^1(\mathcal{O}_k,\mathcal{A}^*_d))^D \rightarrow (\varprojlim_r H^1(k^{nr}/k,{_{\ell^r}}A^t(k^{nr})))^D.$$
\end{itemize}
\end{proposition}

\begin{proof}
La preuve est analogue à celle de \ref{prelC}.
\end{proof}

Comme dans \ref{HnrsC}, nous sommes maintenant en mesure d'introduire la définition suivante:

\begin{definition}\label{HnrsCbis}
Soit $\ell$ un nombre premier ne divisant pas $|F_{d-1}|$. On appelle \textbf{$\ell$-groupe de cohomologie non ramifiée symétrisé de $A$} le groupe:
$$H^d_{nrs}(k,A,\ell) := (\iota_{\ell} \circ \varphi)^{-1}((T_{\ell}H^1(\mathcal{O}_k,\mathcal{A}^*_d))^D) \subseteq H^d(k,A)\{\ell\}$$
où $\varphi: H^d(k,A) \rightarrow H^{d-1}(k^{nr}/k,H^1(k^{nr},A))$ désigne le morphisme induit par la suite spectrale $H^{r}(k^{nr}/k,H^s(k^{nr},A)) \Rightarrow H^{r+s}(k,A)$ et $\iota_{\ell}$ l'isomorphisme composé: 
\begin{align*}
 H^{d-1}(k^{nr}/k,H^1(k^{nr},A))\{\ell\} &\xrightarrow{\sim} \varinjlim_r H^{d-1}(k^{nr}/k,{_{\ell^r}}H^1(k^{nr},A)) \\
&\xleftarrow{\sim} \varinjlim_r H^{d-1}(k^{nr}/k,H^1(k^{nr},{_{\ell^r}}A)) \\
&\xrightarrow{\sim} \varinjlim_r H^{d-1}(k^{nr}/k,{_{\ell^r}}A^t(k^{nr})^D) \\
&\xrightarrow{\sim} (\varprojlim_r H^1(k^{nr}/k,{_{\ell^r}}A^t(k^{nr})))^D.
\end{align*}  
On a alors une suite exacte:
$$0 \rightarrow \frac{H^d(\mathcal{O}_k,\mathcal{A}_d)}{\delta (H^{d-1}(\mathcal{O}_k,R^1g_*A))}\{\ell\} \rightarrow H^d_{nrs}(k,A,\ell) \rightarrow (T_{\ell}H^1(\mathcal{O}_k,\mathcal{A}^*_d))^D \rightarrow 0$$
où $\delta: H^{d-1}(\mathcal{O}_k,R^1g_*A)\rightarrow H^d(\mathcal{O}_k,\mathcal{A}_d)$ est le morphisme de bord provenant de la suite spectrale $H^{r}(\mathcal{O}_k,R^sg_*A)\Rightarrow H^{r+s}(k,A)$.
\end{definition}

\begin{theorem}\label{noyaubis}
Pour $\ell$ premier ne divisant pas $|F_{d-1}|$, la partie $\ell$-primaire du noyau de $H^d(k,A) \rightarrow (H^0(k,A^t)^{\wedge})^D$ est $H^d_{nrs}(k,A,\ell)$.
\end{theorem}

\begin{proof}
La preuve est très similaire à celle de \ref{noyau}. Il suffit de remarquer que le diagramme suivant est commutatif:\\
\centerline{\xymatrix{
& \varinjlim_r H^{d-1}(k^{nr}/k,{_{\ell^r}}H^1(k^{nr},A)) & \\
H^d(k,A)\{\ell\} \ar[ru] \ar[ddd] & & \varinjlim_r H^{d-1}(k^{nr}/k,H^1(k^{nr},{_{\ell^r}}A))\ar[lu]^{\cong}\ar[ddd]^{\cong}\\
& \varinjlim_r H^d(k,{_{\ell^r}}A)  \ar[ru] \ar@{->>}[lu]\ar[d] &\\
& (\varprojlim_r H^1(k,{_{\ell^r}}A^t))^D \ar[dl] \ar[dr] & \\
(H^0(k,A^t)^{(\ell)})^D & & (\varprojlim_r H^1(k^{nr}/k,{_{\ell^r}}A^t(k^{nr})))^D \ar[dl]\\
& (H^0(k^{nr}/k,A^t(k^{nr}))^{(\ell)})^D \ar@{=}[ul] &
}}
où le morphisme $ \varinjlim_r H^{d-1}(k^{nr}/k,H^1(k^{nr},{_{\ell^r}}A)) \rightarrow  \varinjlim_r H^1(k^{nr}/k,{_{\ell^r}}A^t(k^{nr}))^D$ est obtenu par composition des isomorphismes $$ \varinjlim_r H^{d-1}(k^{nr}/k,H^1(k^{nr},{_{\ell^r}}A)) \xrightarrow{\sim}  \varinjlim_r H^{d-1}(k^{nr}/k,{_{\ell^r}}A^t(k^{nr})) \xrightarrow{\sim}  \varinjlim_r H^1(k^{nr}/k,{_{\ell^r}}A^t(k^{nr}))^D.$$
\end{proof}

Pour alléger les notations dans la section suivante, nous noterons:
$$H^d_{nrs}(k,A) := \bigoplus_{\ell \wedge |F_{d-1}|=1} H^d_{nrs}(k,A,\ell).$$
C'est le groupe de torsion dont la partie $\ell$-primaire est $H^d_{nrs}(k,A,\ell)$ si $\ell$ ne divise pas $|F_{d-1}|$, triviale sinon.

\section{\textsc{Variétés abéliennes sur $\mathbb{C}((t_0))...((t_d))(u)$}}\label{Cfonctions}

Supposons dans cette partie que $d\geq 1$ et que $k=\mathbb{C}((t_0))...((t_d))$. Soient $A$ une variété abélienne sur $K=k(X)$ et $A^t$ sa variété abélienne duale. Le but de ce paragraphe est d'établir un théorème de dualité à la Cassels-Tate pour $A$: plus précisément, nous voulons déterminer, sous certaines hypothèses géométriques et modulo divisibles, le dual du groupe de Tate-Shafarevich $\Sha^1(A)$.

Pour chaque $v \in X^{(1)}$, on adopte des notations analogues à celles de la section \ref{notations} pour la variété abélienne $A_v=A \times_K K_v$ sur le corps $d+1$-local $K_v$ (on prendra garde au fait que $K_v$ n'est pas $d$-local). Ainsi, on introduit les schémas en groupes $\mathcal{A}_{v,i}, A_{v,i},F_{v,i},U_{v,i},T_{v,i},B_{v,i}$ pour $i \in \{0,1,...,d+1\}$ et les schémas en groupes $A_{v,-1},F_{v,-1},U_{v,-1},T_{v,-1},B_{v,-1}$.\\
Notons aussi $U$ l'ouvert de bonne réduction de $A$, de sorte que le modèle de Néron $\mathcal{A}$ de $A$ sur $U$ est un schéma abélien. Soit $\mathcal{A}^t$ le schéma abélien dual. \\
Fixons maintenant un nombre premier $\ell$ et faisons l'hypothèse suivante:

\begin{hypol}\label{HCbis}
\begin{minipage}[t]{12.72cm}
pour chaque $v \in X \setminus U$, au moins l'une des deux affirmations suivantes est vérifiée:
\begin{itemize}
\item[$\bullet$] $\ell$ ne divise pas $|F_{v,d}|$,
\item[$\bullet$] la variété abélienne $B_{v,-1}$ est nulle et les tores $T_{v,d},...,T_{v,-1}$ sont anisotropes.
\end{itemize}
\end{minipage}
\end{hypol}

Soit $Z$ l'ensemble des $v \in X^{(1)}$ tels que la variété abélienne $B_{v,-1}$ est nulle et les tores $T_{v,d},...,T_{v,-1}$ sont anisotropes. Pour chaque ouvert $V$ de $U$, on introduit les groupes suivants:
\begin{multline*}
\Sha^1_{nr}(V,A):= \text{Ker} \left(  H^1(K,A) \rightarrow \prod_{v \in X \setminus V} H^1(K_v,A) \times \prod_{v \in V^{(1)}} H^1(K_v,A)/ H^1_{nr}(K_v,A))\right) ,  
\end{multline*}
\begin{multline*}
\Sha^{d+1}_{nrs}(A^t):= \text{Ker} \left(  H^{d+1}(K,A^t) \rightarrow \prod_{v \in Z} H^{d+1}(K_v,A^t) \times \prod_{v \in X^{(1)} \setminus Z} H^{d+1}(K_v,A^t)/ H^{d+1}_{nrs}(K_v,A^t) \right),
\end{multline*}

où $H^1_{nr}(K_v,A^t)$ désigne $H^1(\mathcal{O}_v,\mathcal{A}_v)=H^1(K_v^{nr}/K_v,A(K_v^{nr}))$.\\

Fixons $V$ un ouvert non vide de $U$.

\begin{lemma}
\begin{itemize}
\item[(i)] Pour $r>0$, le groupe $H^r(V,\mathcal{A})$ est de torsion de type cofini.
\item[(ii)] Le groupe $H^2_c(V,\mathcal{A})$ est de torsion de type cofini.
\end{itemize}
\end{lemma}

\begin{proof}
La preuve est analogue à celle de \ref{cofC}.
\end{proof}

\begin{lemma} \label{exactbis}
Il existe des suites exactes:
$$ 0 \rightarrow H^d(V,\mathcal{A}^t) \otimes_{\mathbb{Z}} (\mathbb{Q}/\mathbb{Z})\{\ell\}\rightarrow H^{d+1}(V,\mathcal{A}^t\{\ell\}) \rightarrow H^{d+1}(V,\mathcal{A}^t)\{\ell\} \rightarrow 0 ,$$
$$0 \rightarrow H^{1}_c(V,\mathcal{A})^{(\ell)}\rightarrow H^{2}_c(V,T_{\ell}\mathcal{A}) \rightarrow  T_{\ell}H^{2}_c(V,\mathcal{A}) \rightarrow 0.$$
Ici, $H^{d+1}(V,\mathcal{A}^t\{\ell\})$ et $H^2_c(V,T_{\ell}\mathcal{A})$ désignent $\varinjlim_n H^{d+1}(V,{_{\ell^n}}\mathcal{A}^t)$ et $\varprojlim_n H^2_c(V,{_{\ell^n}}\mathcal{A})$ respectivement.
\end{lemma}

\begin{proof}
La preuve est analogue à celle de \ref{exact}.
\end{proof}

\begin{lemma} \label{accbis}
Il existe un accouplement canonique:
$$ H^1(V,\mathcal{A}\{\ell\}) \times H^{d+2}_c(V,T_{\ell}\mathcal{A}^t) \rightarrow \mathbb{Q}/\mathbb{Z}$$
qui est non dégénéré.
\end{lemma}

\begin{proof}
La preuve est analogue à celle de \ref{acc}.
\end{proof}

De plus, d'après la formule de Barsotti-Weil et la nullité de $\underline{\text{Hom}}_V(\mathcal{A},\mathbb{G}_m)$, on a un accouplement canonique $\mathcal{A}^t \otimes^{\mathbf{L}} \mathcal{A} \rightarrow \mathbb{G}_m[1]$ qui induit donc un accouplement:
$$ H^1(V,\mathcal{A}) \times H^{d+1}_c(V,\mathcal{A}^t) \rightarrow H^{d+3}_c(V,\mathbb{G}_m) \cong \mathbb{Q}/\mathbb{Z}.$$
Posons maintenant: $$D^1(V,\mathcal{A}) = \text{Im}(H^1_c(V,\mathcal{A}) \rightarrow H^1(V,\mathcal{A})) = \text{Ker}(H^1(V,\mathcal{A}) \rightarrow \bigoplus_{v \in X \setminus V} H^1(K_v,A)),$$
$$D^{d+1}_{nrs}(V,\mathcal{A}^t) =  \text{Ker}\left( H^{d+1}(V,\mathcal{A}^t) \rightarrow \bigoplus_{v \in Z \setminus V} H^{d+1}(K_v,A^t) \oplus \bigoplus_{v \in X \setminus (V \cup Z)} H^{d+1}(K_v,A^t)/ H^{d+1}_{nrs}(K_v,A^t))\right).$$
 Ce sont bien sûr des groupes de torsion de type cofini.

\begin{lemma}\label{Shabis}
L'application naturelle $H^1(V,\mathcal{A}) \rightarrow H^1(K,A)$ induit un isomorphisme $D^1(V,\mathcal{A}) \cong \Sha^1_{nr}(V,A)$.
\end{lemma}

\begin{proof}
La preuve est analogue à celle de \ref{Sha}.
\end{proof}

Afin d'établir un théorème de dualité pour les groupes de Tate-Shafarevich, il convient donc d'établir un théorème de dualité pour les groupes $D^1(V,\mathcal{A})$ et $D^{d+1}_{nrs}(V,\mathcal{A}^t)$:

\begin{proposition}\label{d1bis}
Il existe un accouplement canonique:
$$ \overline{D^{d+1}_{nrs}(V,\mathcal{A}^t)}\{\ell\} \times \overline{D^1(V,\mathcal{A})}\{\ell\} \rightarrow  \mathbb{Q}/\mathbb{Z}$$
qui est non dégénéré.
\end{proposition}

\begin{proof} 
La preuve est analogue à celle de \ref{d1}.
\end{proof}

Nous sommes maintenant en mesure d'établir le théorème suivant:

\begin{theorem}\label{cornrsbis}
On rappelle que $k=\mathbb{C}((t_0))...((t_d))$ et que $K=k(X)$ est le corps des fonctions de la courbe $X$. On suppose (H \ref{HCbis})$_{\ell}$. Alors il existe un accouplement non dégénéré de groupes finis:
$$\overline{\Sha^{d+1}_{nrs}(A^t)}\{\ell\} \times \overline{\Sha^1(A)}\{\ell\} \rightarrow \mathbb{Q}/\mathbb{Z}.$$
\end{theorem}

\begin{proof}
D'après la proposition \ref{d1bis} et le lemme \ref{Shabis}, on a un accouplement parfait de groupes finis:
$$\overline{D^{d+1}_{nrs}(V,\mathcal{A}^t)}\{\ell\} \times \overline{\Sha^1_{nr}(V,A)}\{\ell\} \rightarrow \mathbb{Q}/\mathbb{Z}.$$
Pour $V \subseteq V'$ deux ouverts de $U$, on remarque que $\Sha^1_{nr}(V,A)\{\ell\}$ et $ \Sha^1_{nr}(V',A)\{\ell\}$ sont des sous-groupes du groupe de torsion de type cofini $\Sha^1_{nr}(U,A)\{\ell\}$ tels que $\Sha^1_{nr}(V,A)\{\ell\} \subseteq \Sha^1_{nr}(V',A)\{\ell\}$. On en déduit qu'il existe un ouvert non vide $V_0$ de $U$ tel que, pour tout ouvert non vide $V$ de $V_0$, on a $\Sha^1_{nr}(V,A)\{\ell\} = \Sha^1_{nr}(V_0,A)\{\ell\}$. Cela implique que $\Sha^1(V_0,A)\{\ell\}=\Sha^1(A)\{\ell\}$.\\
Par ailleurs, on remarque que, pour $V \subseteq V'$ deux ouverts non vides de $V_0$, on a un diagramme commutatif:\\
\centerline{\xymatrix{
0 \ar[r] & D^{d+1}_{nrs}(V',\mathcal{A}^t)\{\ell\}_{div} \ar[r]\ar[d] & D^{d+1}_{nrs}(V',\mathcal{A}^t)\{\ell\} \ar[r] \ar[d] & \overline{D^{d+1}_{nrs}(V',\mathcal{A}^t)}\{\ell\} \ar[d]^{\cong} \ar[r] & 0\\
0 \ar[r] & D^{d+1}_{nrs}(V,\mathcal{A}^t)\{\ell\}_{div} \ar[r] & D^{d+1}_{nrs}(V,\mathcal{A}^t)\{\ell\} \ar[r]  & \overline{D^{d+1}_{nrs}(V,\mathcal{A}^t)}\{\ell\}  \ar[r] & 0
}}
Comme $\Sha^{d+1}_{nrs}(A^t)\{\ell\} = \varinjlim_{V \subseteq V_0} D^{d+1}_{nrs}(V,\mathcal{A}^t)\{\ell\}$, en passant à la limite inductive, on obtient que la restriction $D^{d+1}_{nrs}(V_0,\mathcal{A}^t)\{\ell\} \rightarrow \Sha^{d+1}_{nrs}(A^t)\{\ell\}$ induit un isomorphisme $\overline{D^{d+1}_{nrs}(V_0,\mathcal{A}^t)}\{\ell\} \xrightarrow{\sim} \overline{\Sha^{d+1}_{nrs}(A^t)}\{\ell\}$. On obtient donc un accouplement non dégénéré de groupes finis:
$$\overline{\Sha^{d+1}_{nrs}(A^t)}\{\ell\} \times \overline{\Sha^1(A)}\{\ell\} \rightarrow \mathbb{Q}/\mathbb{Z}.$$
\end{proof}

\begin{corollary}\label{corglobbis}
On rappelle que $k=\mathbb{C}((t_0))...((t_d))$ et que $K=k(X)$ est le corps des fonctions de la courbe $X$. On suppose (H \ref{HCbis})$_{\ell}$ et on note $i^t: \Sha^{d+1}(A^t) \hookrightarrow \Sha^{d+1}_{nrs}(A^t)$ l'injection canonique. Alors il existe un accouplement non dégénéré à gauche de groupes finis:
$$\Sha^{d+1}(A^t)\{\ell\}/(i^t)^{-1}(\Sha^{d+1}_{nrs}(A^t)\{\ell\}_{div}) \times \overline{\Sha^1(A)}\{\ell\} \rightarrow \mathbb{Q}/\mathbb{Z}.$$
\end{corollary}

\begin{proof}
La preuve est analogue à celle de \ref{corglob}.
\end{proof}

\textbf{Question:} Quel est le noyau à droite dans l'accouplement précédent?

\section{\textsc{Variétés abéliennes sur $\mathbb{Q}_p((t_2))...((t_{d}))$}}\label{Qp}

On suppose dans cette section que $d \geq 2$ et que $k=k'((t_2))...((t_d))$ avec $k'$ corps $p$-adique, $p$ étant un nombre premier fixé. Soit $A$ une variété abélienne sur $k$. On introduit, pour $i \in \{1,2,...,d\}$, les notations $\mathcal{A}_i, A_i,F_i,U_i,T_i,B_i$ analogues à celles du début de la section \ref{notations}. On note aussi $\rho_i$ le rang du tore $T_i$ pour $i \in \{1,...,d\}$.\\
On note $A_0$ la fibre spéciale de $\mathcal{A}_1$. La composante connexe du neutre $A_0^0$ de $A_0$ s'insère dans une suite exacte:
$$0 \rightarrow U_0 \times_{k_0} T_0 \rightarrow A_0 \rightarrow B_0 \rightarrow 0$$
où $U_0$ est groupe abélien unipotent, $T_0$ est un tore et $B_0$ est une variété abélienne sur $k_0$.\\
Par ailleurs, si $M$ est un module galoisien sur un corps $l$, $\ell$ un nombre premier différent de la caractéristique de $l$ et $i$ un entier, on notera $M\{\ell\}(i) = \varinjlim_r {_{\ell^r}}M \otimes \mathbb{Z}/\ell^r\mathbb{Z}(i)$. En tant que groupe abélien, il est isomorphe à $M\{\ell\}$. Par abus de notation, quand $G$ est un groupe algébrique abélien sur $l$, on écrira $G\{\ell\}(i)$ au lieu de $G(l^s)\{\ell\}(i)$.

\subsection{Dualité modulo divisibles}\label{secmoddiv}

Posons ${_{(n)}}\tilde{A} =  \text{\underline{Ext}}^1_k(A,\mathbb{Z}/n\mathbb{Z}(d))$ pour chaque entier naturel $n$ non nul et notons $\tilde{A} =\varinjlim_n {_{(n)}}\tilde{A}$.

\begin{remarque}
En tenant compte de la formule de Barsotti-Weil, il serait plus naturel de considérer $\text{\underline{Ext}}^1(A,\mathbb{Q}/\mathbb{Z}(d))$ au lieu de $\tilde{A}$. Il se trouve en fait que ces deux faisceaux coïncident, comme le montre l'annexe à la fin de ce texte.
\end{remarque}

 Comme $\text{\underline{Hom}}_k(A,\mathbb{Z}/n\mathbb{Z}(d))=0$, on a un morphisme naturel dans la catégorie dérivée ${_{(n)}}\tilde{A}\rightarrow\mathbb{R}\text{\underline{Hom}}_k(A,\mathbb{Z}/n\mathbb{Z}(d))[1]$, d'où un morphisme:
$$A \otimes^{\mathbf{L}} {_{(n)}}\tilde{A} \rightarrow \mathbb{Z}/n\mathbb{Z}(d)[1],$$
induisant un accouplement:
$$H^r(k,A) \times H^{d-r}(k,{_{(n)}}\tilde{A}) \rightarrow H^{d+1}(k,\mathbb{Z}/n\mathbb{Z}(d)) \cong \mathbb{Z}/n\mathbb{Z}.$$
En passant à la limite inductive sur $n$, on obtient un accouplement:
$$H^r(k,A) \times H^{d-r}(k,\tilde{A}) \rightarrow H^{d+1}(k,\mathbb{Q}/\mathbb{Z}(d)) \cong \mathbb{Q}/\mathbb{Z}.$$

\begin{lemma}\label{torsion dual}
Soit $n \in \mathbb{N}$ non nul. On a l'égalité:
$${_{(n)}}\tilde{A} = \text{\underline{Hom}}_k({_n}A,\mathbb{Z}/n\mathbb{Z}(d)) = {_n}A^t \otimes \mathbb{Z}/n\mathbb{Z}(d-1)$$ et la multiplication par $n$ sur $\tilde{A}$ induit une suite exacte de faisceaux:
$$0 \rightarrow {_{(n)}}\tilde{A} \rightarrow \tilde{A} \rightarrow \tilde{A} \rightarrow 0.$$
\end{lemma}

\begin{proof}
La suite exacte courte $0 \rightarrow {_n}A \rightarrow A \rightarrow A \rightarrow 0$ induit une suite exacte de cohomologie:
$$\text{\underline{Hom}}_k(A,\mathbb{Z}/n\mathbb{Z}(d)) \rightarrow \text{\underline{Hom}}_k({_n}A,\mathbb{Z}/n\mathbb{Z}(d)) \rightarrow {_{(n)}}\tilde{A} \rightarrow {_{(n)}}\tilde{A}.$$
Comme $\text{\underline{Hom}}_k(A,\mathbb{Z}/n\mathbb{Z}(d))=0$, on en déduit que: $${_{(n)}}\tilde{A} = \text{\underline{Hom}}_k({_n}A,\mathbb{Z}/n\mathbb{Z}(d)) = {_n}A^t \otimes \mathbb{Z}/n\mathbb{Z}(d-1).$$
Cela impose aussi que $\tilde{A} = \varinjlim_n {_n}A^t \otimes \mathbb{Z}/n\mathbb{Z}(d-1)$. Comme $A^t(k^s)_{tors}$ est divisible, cela montre immédiatement que la multiplication par $n$ sur $\tilde{A}$ induit une suite exacte de faisceaux:
$$0 \rightarrow {_{(n)}}\tilde{A} \rightarrow \tilde{A} \rightarrow \tilde{A} \rightarrow 0.$$
\end{proof}

Cela montre que ${_{(n)}}\tilde{A}$ est la $n$-torsion de $\tilde{A}$. On notera donc par la suite ${_{n}}\tilde{A}$ au lieu de ${_{(n)}}\tilde{A}$.

\begin{remarque}
En fait, on a montré que $\tilde{A}\{\ell\} = A^t\{\ell\}(d-1)$ pour chaque premier $\ell$.
\end{remarque}

\begin{corollary}\label{surjQp}
Pour chaque entier naturel $n$ et chaque entier $r$, le morphisme $ H^{r-1}(k,A)/n \rightarrow ({_n}H^{d+1-r}(k,\tilde{A}))^D$ est injectif et le morphisme ${_n}H^{r}(k,A) \rightarrow (H^{d-r}(k,\tilde{A})/n)^D$ est surjectif.
\end{corollary}

\begin{proof}
On a un diagramme commutatif à lignes exactes:\\
\centerline{\xymatrix{
0 \ar[r]& H^{r-1}(k,A)/n \ar[d] \ar[r] & H^r(k,{_n}A) \ar[r]\ar[d] & {_n}H^r(k,A) \ar[r]\ar[d] & 0\\
0 \ar[r]& ({_n}H^{d+1-r}(k,\tilde{A}))^D  \ar[r] & H^{d+1-r}(k,{_n}\tilde{A})^D \ar[r] & (H^{d-r}(k,\tilde{A})/n)^D \ar[r] & 0,
}}
où le morphisme vertical central est un isomorphisme d'après le lemme précédent et le théorème 2.17 de \cite{MilADT}. On en déduit que le morphisme $ H^{r-1}(k,A)/n \rightarrow ({_n}H^{d+1-r}(k,\tilde{A}))^D$ est injectif et le morphisme ${_n}H^{r}(k,A) \rightarrow (H^{d-r}(k,\tilde{A})/n)^D$ est surjectif.
\end{proof}

\begin{notation}
Pour $r \in \mathbb{Z}$, $n \in \mathbb{N}^*$ et $M$ un $\text{Gal}(k^s/k)$-module discret tel que ${_n}H^r(k,M)$ et $H^r(k,M)/n$ sont finis, on note $\lambda_r(k,n,M) = \frac{|{_n}H^r(k,M)|}{|H^r(k,M)/n|}$.
\end{notation}

\begin{proposition}
Pour $r \geq 0$, il existe des familles d'entiers naturels $(\beta_{r,\ell})_{\ell}$, $(\beta_{r,{\ell}}^*)_{\ell}$, $(\beta_{r,{\ell}}^t)_{\ell}$, $(\beta_{0,{\ell}}^{tors})_{\ell}$ et $(\beta_{0,{\ell}}^{t,tors})_{\ell}$ indexées par les nombres premiers telles que, pour tout $n \in \mathbb{N}^*$, on a:
\begin{gather*}
\lambda_r(k,n,A) = \prod_{\ell} {\ell}^{\beta_{r,{\ell}}v_{\ell}(n)}, \\
\lambda_r(k,n,\tilde{A})=\prod_{\ell} {\ell}^{\beta_{r,{\ell}}^*v_{\ell}(n)},\\
\lambda_r(k,n,A^t)=\prod_{\ell} {\ell}^{\beta_{r,{\ell}}^tv_{\ell}(n)},\\
\lambda_0(k,n,A(k^s)_{tors})= \prod_{\ell} {\ell}^{\beta_{0,{\ell}}^{tors}v_{\ell}(n)},\\
\lambda_0(k,n,A^t(k^s)_{tors})= \prod_{\ell} {\ell}^{\beta_{0,{\ell}}^{t,tors}v_{\ell}(n)}.
 \end{gather*}
\end{proposition}

\begin{proof}
Les deux dernières égalités sont évidentes car $A(k)_{tors}$ et $A(k)_{tors}$ sont de torsion de type cofini. Montrons les trois premières. Pour $r \geq 1$, elles sont évidentes, puisque les groupes $H^r(k,A)$, $H^r(k,\tilde{A})$ et $H^r(k,A^t)$ sont de torsion de type cofini. Le cas $r=0$ découle alors des formules suivantes:
\begin{gather*}
1 = \chi (k, {_n}A) = \prod_{r=0}^{d+1} \lambda_r(k,n,A)^{(-1)^r},\\
1 = \chi (k, ({_n}A)') = \prod_{r=0}^{d+1} \lambda_r(k,n,\tilde{A})^{(-1)^r},\\
1 = \chi (k, {_n}A^t) = \prod_{r=0}^{d+1} \lambda_r(k,n,A^t)^{(-1)^r}.
\end{gather*}
\end{proof}

\begin{theorem}\label{modulo divisibles Qp}
Pour $r \geq 1$, le noyau du morphisme $H^{r}(k,A) \rightarrow (H^{d-r}(k,\tilde{A})^{\wedge})^D$ est un groupe de torsion de type cofini divisible.
\end{theorem}

\begin{proof}
Soit $s \in \{-1,0,...,d+1\}$. On note $N_s$ le noyau de $H^{s}(k,A) \rightarrow (H^{d-s}(k,\tilde{A})^{\wedge})^D$. On calcule la caractéristique d'Euler-Poincaré de ${_n}A$ pour chaque $n$:
\begin{align*}
1 & = \chi (k, {_n}A)\\
& = \prod_{r=0}^{s-1} \lambda_r(k,n,A)^{(-1)^r} \cdot \prod_{r=s+1}^{d+1} |H^r(k,{_n}A)|^{(-1)^r} \cdot |{_n}H^{s}(k,A)|^{(-1)^{s}}\\
& = \prod_{r=0}^{s-1} \lambda_r(k,n,A)^{(-1)^r} \cdot \prod_{r=0}^{d-s} |H^{r}(k,{_n}\tilde{A})|^{(-1)^{d+1-r}} \cdot |{_n}H^{s}(k,A)|^{(-1)^{s}}\\
& = \prod_{r=0}^{s-1} \lambda_r(k,n,A)^{(-1)^r} \cdot \prod_{r=0}^{d-s} \lambda_r(k,n,\tilde{A})^{(-1)^{d+1-r}}\cdot |{_n}H^{s}(k,A)|^{(-1)^{s}}|H^{d-s}(k,\tilde{A})/n|^{(-1)^{s+1}}\\
& = |{_n}H^{s}(k,A)|^{(-1)^{s}}|H^{d-s}(k,\tilde{A})/n|^{(-1)^{s+1}} \cdot \prod_{\ell} {\ell}^{(-1)^{s+1}\gamma_{s,{\ell}} v_{\ell}(n)},
\end{align*}
où $\gamma_{s,{\ell}} = \sum_{r=0}^{s-1}(-1)^{r+s+1}\beta_{r,{\ell}}+\sum_{r=0}^{d-s}(-1)^{d+s-r}\beta_{r,{\ell}}^*$. On obtient donc, pour tout $s \in \{-1,0,...,d+1\}$:
$$ |{_n}N_s| = \frac{|{_n}H^{s}(k,A)|}{|H^{d-s}(k,\tilde{A})/n|} = \prod_p {\ell}^{\gamma_{s,{\ell}} v_{\ell}(n)},$$
ce qui prouve que $N_s$ est divisible à condition que $s\neq 0$.
\end{proof}

En reprenant les notations de la preuve précédente, on a alors:
$$N_s \cong \bigoplus_{\ell} (\mathbb{Q}_{\ell}/\mathbb{Z}_{\ell})^{\gamma_{s,{\ell}}},$$
et nous voulons calculer les $\gamma_{s,{\ell}} = \sum_{r=0}^{s-1}(-1)^{r+s+1}\beta_{r,{\ell}}+\sum_{r=0}^{d-s}(-1)^{d+s-r}\beta_{r,{\ell}}^*$. Avant de passer à la suite, il est utile d'établir des équations reliant les différentes variables que nous avons introduites ($\beta_{r,{\ell}}, \beta_{r,{\ell}}^*,\beta_{r,p}^t,\beta_{0,{\ell}}^{tors},\beta_{0,{\ell}}^{t,tors},\gamma_{r,{\ell}}$).

\begin{proposition}\label{equa0Qp}
Soit $\ell$ un nombre premier. Les entiers $(\beta_{r,{\ell}})_r$ et $(\beta^*_{r,{\ell}})_r$ vérifient les équations:\\
$$\left
\{
\begin{array}{c}
\gamma_{r,{\ell}}=\beta_{r,{\ell}} \;\;\;\; \forall r \in \{-1\} \cup \{1,2,...,d+1\}\\
\gamma_{0,{\ell}} = \beta_{0,{\ell}}^{tors}\\
\beta_{r,{\ell}}=\beta^*_{d+1-r,{\ell}} \;\;\;\; \forall r \in \{2,3,...,d+1\}\\ 
\beta_{1,{\ell}} - \beta_{0,{\ell}} = \beta_{d,{\ell}}^* - \beta_{d+1,{\ell}}^*\\
\beta_{0,{\ell}}^{tors}=\beta_{d+1,{\ell}}^*\\
\sum_{r=0}^{d+1} (-1)^r\beta_{r,{\ell}}=0\\
\beta_{r,{\ell}}=\beta^*_{r,{\ell}}=0\;\;\;\; \forall r \geq d+2 
\end{array}
\right.$$
\end{proposition}

\begin{proof}
Exactement comme dans la démonstration de \ref{modulo divisibles Qp}, on a pour chaque $r \in \{-1,0,1,...,d+1\}$ la relation:
$$1  = |{_n}H^{r}(k,A)|^{(-1)^{r}}|H^{d-r}(k,\tilde{A})/n|^{(-1)^{r+1}} \cdot \prod_{\ell} {\ell}^{(-1)^{r+1}\gamma_{r,{\ell}} v_{\ell}(n)}$$
avec $\gamma_{r,{\ell}} = \sum_{s=0}^{r-1}(-1)^{r+s+1}\beta_{s,{\ell}}+\sum_{s=0}^{d-r}(-1)^{d+r-s}\beta_{s,{\ell}}^*$. Cela montre immédiatement que:
\begin{itemize}
\item[$\bullet$] $\gamma_{r,{\ell}}=\beta_{r,{\ell}}$ pour $r \in \{1,2,...,d-1\}\cup\{-1,d+1\}$,
\item[$\bullet$] $\beta_{0,{\ell}}^{tors} =\gamma_{0,{\ell}}$,
\item[$\bullet$] $\beta_{d,{\ell}} = \gamma_{d,{\ell}}$ car $\tilde{A}$ est de torsion,
\end{itemize}
ce qui achève la preuve.
\end{proof}

\subsection{Étude hors de $p$}

On fixe un nombre premier $\ell$ différent de $p$.

\subsubsection{Conditions suffisantes pour la nullité des $\beta_{r,\ell}$}

On procède de manière similaire à la section \ref{Cbis}. On commence par des énoncés analogues à ceux des lemmes \ref{lemme1} et \ref{lemme2}.

\begin{lemma}\label{lemme1qp}
Pour chaque entier naturel $r$ et chaque entier $i$, les parties divisibles de groupes de type cofini $H^{r}(k_{d-1},H^1(k^{nr},A\{\ell\}(i)))$ et $H^{d-r}(k_{d-1},A_{d-1}^0\{\ell\}(d-1-i))$ sont (non canoniquement) isomorphes.
\end{lemma}

\begin{proof}
Soit $r\geq 0$. On a un isomorphisme $H^1(k^{nr},A\{\ell\}(i)) \cong (\varprojlim_s {_{\ell^s}}A^t(k^{nr})(-i))^D$.
On calcule alors:
\begin{align*}
H^{r}(k_{d-1},H^1(k^{nr},A\{\ell\}(i))) &= H^{r}(k_{d-1},(\varprojlim_s {_{\ell^s}}A^t(k^{nr})(-i))^D)\\
& \cong \varinjlim_s H^{r}(k_{d-1}, {_{\ell^s}}A^t(k^{nr})(-i)^D)\\
& \cong \varinjlim_s H^{d-r}(k_{d-1},{_{\ell^s}}A^t(k^{nr})(d-1-i)^D\\
& \cong (\varprojlim_s H^{d-r}(k_{d-1},{_{\ell^s}}A^t(k^{nr})(d-1-i)))^D.
\end{align*}
Par conséquent, en utilisant un résultat analogue à celui du lemme \ref{tech} les parties divisibles de $H^{r}(k_{d-1},H^1(k^{nr},A(k^s)\{\ell\}(i)))$ et de $H^{d-r}(k_{d-1},A(k^{nr})\{\ell\}(d-1-i)) \cong H^{d-r}(k_{d-1},A_{d-1}\{\ell\}(d-1-i))$ sont isomorphes. \\
On remarque maintenant que l'on a la suite exacte:
$$0 \rightarrow A_{d-1}^0(k_{d-1}^s)  \rightarrow A_{d-1}(k_{d-1}^s) \rightarrow F_{d-1}(k_{d-1}^s) \rightarrow 0,$$
où $A_{d-1}^0$ désigne la composante connexe du neutre dans $A_{d-1}$. Il existe donc un $\text{Gal}(k_{d-1}^s/k_{d-1})$-module fini $F$ tel que:
$$0 \rightarrow A_{d-1}^0\{\ell\}(d-1-i)  \rightarrow A_{d-1}\{\ell\}(d-1-i) \rightarrow F\{\ell\}(d-1-i) \rightarrow 0.$$
On en déduit que les parties divisibles des groupes de type cofini $H^{d-r}(k_{d-1},A_{d-1}\{\ell\}(d-1-i))$ et $H^{d-r}(k_{d-1},A_{d-1}^0\{\ell\}(d-1-i))$ sont (non canoniquement) isomorphes. 
\end{proof}

\begin{lemma}\label{lemme2qp}
Pour chaque entier $r\geq 0$, on a un isomorphisme $H^r(k_{d-1},A(k^{nr})\{\ell\}(i)) \cong H^r(k_{d-1},A_{d-1}\{\ell\}(i))$.
\end{lemma}

\begin{proof}
Comme l'extension $k^{nr}/k$ est non ramifiée par définition, on sait que $A(k^{nr}) = \mathcal{A}_d(\mathcal{O}_{k^{nr}})$ et que le morphisme $\mathcal{A}_d(\mathcal{O}_{k^{nr}}) \rightarrow A_{d-1}(k_{d-1}^s)$ est surjectif de noyau uniquement divisible par $\ell$. On en déduit que $H^r(k_{d-1},A(k^{nr})\{\ell\}(i)) \cong H^r(k_{d-1},A_{d-1}\{\ell\}(i))$.
\end{proof}

On est maintenant en mesure d'établir la proposition qui fournit des conditions suffisantes pour que les groupes de cohomologie de $A\{\ell\}(i)$ soient finis. 

\begin{proposition}\label{prop1}
Soient $r$ un entier naturel et $i$ un entier tels que $r-i-1 \not\in \{-1,0,1\}$. Le groupe $H^r(k,A\{\ell\}(i))$ est fini. 
\end{proposition}

\begin{proof}
Procédons par récurrence sur $d$. \\
Pour $d=0$, c'est la proposition \ref{vafini}.\\
Soit $d$ un entier naturel tel que la proposition est vraie au rang $d-1$. Montrons la proposition au rang $d$. La suite spectrale $H^r(k_{d-1},H^s(k^{nr},A\{\ell\}(i))) \Rightarrow H^{r+s}(k,A\{\ell\}(i))$ dégénère en une suite exacte longue:
$$... \rightarrow H^r(k_{d-1},A(k^{nr})\{\ell\}(i))) \rightarrow H^r(k,A\{\ell\}(i)) \rightarrow H^{r-1}(k_{d-1},H^1(k^{nr},A\{\ell\}(i))) \rightarrow ... \text{.}$$
Étudions les termes $H^r(k_{d-1},A(k^{nr})\{\ell\}(i)))$ et $H^{r-1}(k_{d-1},H^1(k^{nr},A\{\ell\}(i)))$.
\begin{itemize}
\item[$\bullet$] D'après \ref{lemme2qp}, la partie divisible de $H^r(k_{d-1},A(k^{nr})\{\ell\}(i)))$ est isomorphe à celle de $H^r(k_{d-1},A_{d-1}^0\{\ell\}(i))$. De plus, la suite exacte $0 \rightarrow U_{d-1} \times T_{d-1} \rightarrow A_{d-1}^0 \rightarrow B_{d-1} \rightarrow 0$ montre l'exactitude de:
$$ H^{r}(k_{d-1},T_{d-1}\{\ell\}(i)) \rightarrow H^{r}(k_{d-1},A_0^{d-1}\{\ell\}(i)) \rightarrow H^{r}(k_{d-1},B_{d-1}\{\ell\}(i)),$$
et on a vu dans \ref{torr} que $\lambda_r(k_{d-1},n,T_{d-1}\{\ell\}(i))$ vaut $1$ car $r-i-1 \not\in \{0,1\}$. Comme $H^{r}(k_{d-1},B_{d-1}\{\ell\}(i))$ est fini par hypothèse de récurrence, on conclut que $H^r(k_{d-1},A(k^{nr})\{\ell\}(i)))$ est fini.
\item[$\bullet$] D'après \ref{lemme1qp}, la partie divisible de $H^{r-1}(k_{d-1},H^1(k^{nr},A\{\ell\}(i)))$ est isomorphe à celle de $H^{d+1-r}(k_{d-1},A_{d-1}^0\{\ell\}(d-1-i))$, et $H^{d-r+1}(k_{d-1},A_{d-1}^0\{\ell\}(d-1-i))$ s'insère dans une suite exacte:\\
\centerline{\xymatrix{
H^{d-r+1}(k_{d-1},T_{d-1}\{\ell\}(d-1-i)) \ar[d] \\ H^{d-r+1}(k_{d-1},A_{d-1}^0\{\ell\}(d-1-i))\ar[d] \\ H^{d-r+1}(k_{d-1},B_{d-1}\{\ell\}(d-1-i)),}}
où $\lambda_{d-r+1}(k_{d-1},n,T_{d-1}\{\ell\}(d-1-i)))=1$ d'après \ref{torr} car $r-i-1 \not\in \{0,-1\}$. Comme $H^{d-r+1}(k_{d-1},B_{d-1}\{\ell\}(d-1-i))$ est fini par hypothèse de récurrence, on conclut que $H^{r-1}(k_{d-1},H^1(k^{nr},A\{\ell\}(i)))$ est fini.
\end{itemize}
On en déduit que $H^r(k,A\{\ell\}(i))$ est fini. 
\end{proof}

\begin{remarque}\label{rq1}
De manière tout à fait analogue, on peut montrer que:
\begin{itemize}
\item[$\bullet$] si $\rho_0=...=\rho_{d-1}=0$, alors $H^2(k,A\{\ell\}(i))$ est fini; 
\item[$\bullet$] si $\rho_0=...=\rho_{d-2}=0$, alors $H^1(k,A)\{\ell\}$ est fini;
\item[$\bullet$] pour $i \neq -1$, le groupe $H^0(k,A\{\ell\}(i))$ est fini.
\end{itemize}
\end{remarque}

Nous pouvons à présent établir le théorème suivant qui montre la nullité des $\beta_{r,\ell}$ sous certaines hypothèses.

\begin{theorem}\label{suff}
Soit $\ell \neq p$ un nombre premier. 
\begin{itemize}
\item[(i)] Pour $r \in \{ 3,...,d+1\}$, on a $\beta_{r,\ell}=0$.
\item[(ii)] Si $\rho_0=...=\rho_{d-1}=0$, alors $\beta_{2,\ell}=0$.
\item[(iii)] Si $\rho_0=...=\rho_{d-2}=0$, alors $\beta_{1,\ell}=0$.
\item[(iv)] On a $\sum_{s=0}^{d+1} (-1)^s\beta_{s,\ell}=0$, $\beta_{0,\ell}=-\sum_{e=1}^{d-1} e \rho_e$ et $\beta_{0,\ell}^{tors}=0$.
\end{itemize}
\end{theorem}

\begin{proof}
\begin{itemize}
\item[(i)] C'est un corollaire immédiat de la proposition \ref{prop1} car $H^r(k,A)\{\ell\} \cong H^r(k,A(k^s)\{\ell\})$.
\item[(ii)] C'est un corollaire immédiat de la remarque \ref{rq1} car $H^2(k,A)\{\ell\} \cong H^2(k,A(k^s)\{\ell\})$.
\item[(iii)] C'est un corollaire immédiat de la remarque \ref{rq1}.
\item[(iv)] La preuve est analogue à celles de \ref{equa0Qp}, \ref{beta0} et \ref{0} en utilisant les remarques \ref{tor0qp} et \ref{tor0torsqp}, ainsi que le lemme 3.3 et le corollaire 3.4 de \cite{MilADT}.
\end{itemize}
\end{proof}

\begin{corollary}\label{inviso}
La quantité $\sum_{e=1}^{d-1} e \rho_e$ est invariante par isogénie. En particulier, elle prend la même valeur pour $A$ et $A^t$.
\end{corollary}

\subsubsection{Conditions nécessaires pour la nullité des $\beta_{r,\ell}$}

Dans cette section, nous allons donner des réciproques partielles au théorème \ref{suff}. Nous avons d'abord besoin de trois lemmes préliminaires.

\begin{lemma}\label{lemrec}
Soit $i \in \mathbb{Z} \setminus \{0,1\}$. Si $H^1(k,A\{\ell\}(i))$ ou $H^1(k_{d-1},A_{d-1}^0\{\ell\}(i))$ est fini, alors il en est de même de $H^1(k_r,B_r\{\ell\}(i))$ et de $H^1(k_r,A_r^0\{\ell\}(i))$ pour $0 \leq r \leq d-1$.
\end{lemma}

\begin{proof}
Il suffit de procéder par récurrence descendante en remarquant $H^2(k_r,T_r\{\ell\}(d))$ est fini pour chaque $r$.
\end{proof}

\begin{lemma}\label{(C)}
Supposons que $H^1(k_{r},A^0_{r}\{\ell\}(-1))$ soit fini pour $0\leq r \leq d-2$. Alors $\rho_0=...=\rho_{d-2}=0$.
\end{lemma}

\begin{proof}
Montrons par récurrence sur $s$ que $\rho_s=0$
\begin{itemize}
\item[$\bullet$] Les groupes $H^1(k_{0},A^0_{0}\{\ell\}(-1))$ et $H^0(k_{0},B_0\{\ell\}(-1))$ sont finis. Il en est donc de même de $H^1(k_0,T_0\{\ell\}(-1))$. Donc $\rho_0=0$.
\item[$\bullet$] Soit $s \leq d-2$ tel que $\rho_0=...=\rho_{s-1}=0$. Montrons $\rho_s=0$. Le groupe $H^1(k_{s},A^0_{s}\{\ell\}(-1))$ est fini. De plus, il en est de même de $H^0(k_s,B_s\{\ell\}(-1))$ puisque $\rho_0=...=\rho_{s-1}=0$. Donc $H^1(k_{s},T_{s}\{\ell\}(-1))$ est fini et $\rho_s=0$.
\end{itemize}
\end{proof}

\begin{theorem}\label{nullitéQp}
Supposons $d \geq 2$. Soit $\ell \neq p$ un nombre premier. Les deux assertions suivantes sont équivalentes:
\begin{itemize}
\item[(i)] $\rho_{d-1}=\rho_{d-2}=...=\rho_0=0$;
\item[(ii)] $\beta_{d,\ell}=\beta_{d-1,\ell}=...=\beta_{2,\ell}=0$.
\end{itemize}
\end{theorem}

\begin{proof}
Le sens direct a déjà été prouvé. Montrons que $(ii) \Rightarrow (i)$. Le groupe $H^2(k,A)$ est fini. Il en est de même du groupe $H^3(k_{d-1},H^0(k^{nr},A(k^s)\{\ell\})$. Par conséquent, $H^1(k_{d-1},H^1(k^{nr},A(k^s)\{\ell\}))$ est fini. Le lemme \ref{lemme1qp} montre alors la finitude de $H^{d-1}(k_{d-1},A^0_{d-1}\{\ell\}(d-1))$. Comme $H^d(k_{d-1},T_{d-1}\{\ell\}(d-1))$ est fini, on en déduit qu'il en est même de $H^{d-1}(k_{d-1},B_{d-1}\{\ell\}(d-1))$ et donc de $H^{d-2}(k_{d-2}, H^1(k_{d-1}^{nr},B_{d-1}\{\ell\}(d-1)))$. Toujours avec le lemme \ref{lemme1qp}, on obtient la finitude de $H^1(k_{d-2},A^0_{d-2}\{\ell\}(-1))$.  Le lemme \ref{lemrec} montre alors que $H^1(k_{r},A^0_{r}\{\ell\}(-1))$ est fini pour $r \leq d-2$, et avec le lemme \ref{(C)}, on obtient $\rho_0=...=\rho_{d-2}$.\\
Reste à montrer que $\rho_{d-1}=0$. Pour ce faire, on remarque que $H^{d}(k_{d-1},B_{d-1}\{\ell\}(d-1))$ est fini car $\rho_0=...=\rho_{d-2}=0$. Comme $H^{d}(k_{d-1},T_{d-1}\{\ell\}(d-1))$ est aussi fini, il en est de même de $H^{d}(k_{d-1},A_{d-1}^0\{\ell\}(d-1))$ et donc de $H^{0}(k_{d-1},H^1(k^{nr},A(k^s)\{\ell\}))$. On en déduit la finitude de $H^2(k_{d-1},H^0(k^{nr},A(k^s)\{\ell\}))$ et donc aussi celle de $H^2(k_{d-1},A^0_{d-1}(k_{d-1}^s)\{\ell\})$. Le groupe $H^1(k_{d-1},B_{d-1}(k_{d-1}^s)\{\ell\})$ est fini car $\rho_0=...=\rho_{d-2}$. Il en est donc de même de $H^2(k_{d-1},T_{d-1}(k_{d-1}^s)\{\ell\})$, et $\rho_{d-1}=0$.
\end{proof}

\begin{remarque}\label{faux}
Plaçons-nous dans le cas où $k=\mathbb{Q}_p((t))$ (et $d=2$). Dans l'article \cite{Koy}, Y. Koya construit un complexe de $\text{Gal}(k^s/k)$-modules pour lequel la multiplication par $\ell^s$ induit un triangle distingué $({_{\ell^s}}A)' \rightarrow C \rightarrow C \rightarrow ({_{\ell^s}}A)'[1]$ quel que soit l'entier naturel $s$. Son théorème principal (théorème 1.1) implique que $H^0(k,C)\{\ell\}$ et $H^1(k,C)\{\ell\}$ sont finis quelle soit la variété abélienne $A$ sur $k$. De plus, sa preuve repose très fortement sur la proposition 4.1, qui impose que, pour chaque $s \geq 0$, on a $|H^0(k,C)/\ell^s|=|{_{\ell^s}}H^0(k,C)|$. Cela montre que la fonction $s \mapsto |H^0(k,C)/\ell^s|$ est bornée. Or on remarque que:
\begin{align*}\frac{|{_{\ell^s}}H^1(k,C)|}{|{_{\ell^s}}H^2(k,A)|}&= \frac{|{_{\ell^s}}H^1(k,C)|}{|H^1(k,({_{\ell^s}}A)')|}\frac{|H^1(k,({_{\ell^s}}A)')|}{|H^2(k,{_{\ell^s}}A)|}\frac{|H^2(k,{_{\ell^s}}A)|}{|{_{\ell^s}}H^2(k,A)|}\\
&=\frac{|H^1(k,A)/\ell^s|}{|H^0(k,C)/\ell^s|}.
\end{align*}
Par conséquent, $H^2(k,A)\{\ell\}$ est fini, quelle que soit la variété abélienne $A$. En utilisant \ref{nullitéQp}, cela impose que $\rho_1=0$ pour toute variété abélienne $A$. Mais cela est clairement faux: par exemple, la courbe elliptique $y^2=x^3+x^2+t$ vérifie $\rho_1=1$. On en déduit que, parmi le théorème 1.1 et la proposition 4.1 de \cite{Koy}, au moins l'un des deux énoncés est faux. En particulier, la preuve du théorème 1.1 de \cite{Koy} semble erronée et difficile à rattraper.
\end{remarque}

\subsection{Étude en $p$}

Les résultats sont plus imprécis que dans le paragraphe précédent, puisque nous ne savons pas calculer les parties $p$-primaires des groupes de cohomologie d'un tore.

\subsubsection{Cas général}

Voici une condition suffisante pour que les $\beta_{r,p}$ soient nuls:

\begin{proposition}\label{nullitéQpp}
Supposons que: $$\dim B_1=\dim T_1=\dim T_2 = ... = \dim T_{d-1}=0.$$
Alors $\beta_{r,p}=0$ pour tout entier $r$.
\end{proposition}

\begin{proof}
La preuve est analogue à celles réalisées dans les paragraphes précédents. Elle est en fait beaucoup plus facile!
\end{proof}

\subsubsection{Quelques précisions dans le cas $d=2$}\label{precisions}

Dans ce paragraphe, on suppose $d=2$.

\begin{proposition}
On a: $\frac{|H^0(k,A)/n|}{|{_n}H^0(k,A)|}=n^{\rho_1} p^{[k_1:\mathbb{Q}_p] \cdot (\dim T_1+\dim B_1) \cdot v_p(n)}$.
\end{proposition}

\begin{proof}
On a $A(k)=\mathcal{A}_2(\mathcal{O}_{k})$. On dispose en plus d'une suite exacte $0 \rightarrow D\rightarrow \mathcal{A}_2(\mathcal{O}_{k}) \rightarrow A_1(k_1) \rightarrow 0,$
où $D$ est uniquement divisible. Le lemme du serpent fournit alors des isomorphismes ${_n}H^0(k,A) \cong  {_n}H^0(k_1,A_1)$ et $H^0(k,A)/ \cong  H^0(k_1,A_1)/n$. On obtient donc $\frac{|H^0(k,A)/n|}{|{_n}H^0(k,A)|} = \frac{|H^0(k_1,A_1)/n|}{|{_n}H^0(k_1,A_1)|}$. Les suites $0 \rightarrow A_1^0 \rightarrow A_1 \rightarrow F_1 \rightarrow 0$ et $0 \rightarrow U_1 \times T_1 \rightarrow A_1^0 \rightarrow B_1 \rightarrow 0$ montrent que $\frac{|H^0(k,A)/n|}{|{_n}H^0(k,A)|} = \frac{|H^0(k_1,A^1_0)/n|}{|{_n}H^0(k_1,A^1_0)|} = \frac{|H^0(k_1,T_1)/n|}{|{_n}H^0(k_1,T_1)|}  \frac{|H^0(k_1,B_1)/n|}{|{_n}H^0(k_1,B_1)|} = n^{\rho_1} p^{[k_1:\mathbb{Q}_p] \cdot (\dim T_1+\dim B_1) \cdot v_p(n)}$.
\end{proof}

\begin{theorem} \label{nullitéQpp2}
On a:
\begin{gather*}
\beta_{0,p} - \beta_{1,p} +\beta_{2,p} - \beta_{3,p} =0\\
\beta_{0,p} = -\rho_1-[k_1:\mathbb{Q}_p](\dim T_1 + \dim B_1),
\end{gather*}
Si $\dim B_1=0$, alors $\beta_{1,p}=\beta_{3,p}=0$. 
\end{theorem}

\begin{proof}
Les preuves sont analogues à celles réalisées dans les paragraphes précédents.
\end{proof}

\begin{corollary}
On a $\dim U_1 =\dim U_1^*$. En particulier, si $A$ a très mauvaise réduction, alors il en est de même de $A^t$.
\end{corollary}

\begin{remarque}
On pourrait bien sûr remplacer $A^t$ par n'importe quelle variété abélienne isogène à $A$.
\end{remarque}

\subsection{Le noyau de $H^d(k,\tilde{A}) \rightarrow (H^0(k,A)^{\wedge})^D$}

Exactement comme dans la section \ref{secmoddiv}, on peut montrer que:

\begin{theorem}
Pour chaque $r \geq 0$, il existe un morphisme naturel surjectif $H^r(k,\tilde{A}) \rightarrow (H^{d-r}(k,A)^{\wedge})^D$ dont le noyau est de torsion de type cofini divisible.
\end{theorem}

Il se trouve que, dans certains cas, il est possible d'expliciter le noyau de $H^d(k,\tilde{A}) \rightarrow (H^0(k,A)^{\wedge})^D$. Pour ce faire, il convient de poser $\tilde{\mathcal{A}}=g_*\tilde{A}$ où $g: \text{Spec} \; k \rightarrow \text{Spec} \; \mathcal{O}_k$ désigne l'immersion ouverte, et d'établir quelques propriétés préliminaires:

\begin{lemma}\label{cohonrQp}
\begin{itemize}
\item[(i)] Le morphisme naturel $H^1(\mathcal{O}_k,\mathcal{A}_d) \rightarrow H^1(k,A)$ est injectif d'image le sous-groupe $H^1(k^{nr}/k,A(k^{nr}))$ de $ H^1(k,A)$.
\item[(ii)] Le faisceau $\tilde{\mathcal{A}}$ est de torsion. De plus, pour chaque $n \geq 1$, on a l'égalité ${_n}\tilde{\mathcal{A}} = {_n}\mathcal{A}^*_d\otimes \mathbb{Z}/n\mathbb{Z}(d-1)$.
\item[(iii)] On a un isomorphisme $H^d(\mathcal{O}_k,\tilde{\mathcal{A}}) \rightarrow H^d(k^{nr}/k,\tilde{A}(k^{nr}))$ faisant commuter le diagramme:\\
\centerline{\xymatrix{
H^d(\mathcal{O}_k,\tilde{\mathcal{A}}) \ar[r]^{\text{Res}} \ar[d]^{\cong} & H^d(k,\tilde{A})\\
H^d(k^{nr}/k,\tilde{A}(k^{nr})) \ar[ru]_{\text{Inf}} & .
}}
\end{itemize}
\end{lemma}

\begin{proof}
Les preuves de (i) et (iii) sont analogues à celle de \ref{cohonr}. Le fait que $\tilde{\mathcal{A}}$ est de torsion est évident, et pour montrer que ${_n}\tilde{\mathcal{A}} = {_n}\mathcal{A}^*_d\otimes \mathbb{Z}/n\mathbb{Z}(d-1)$, il suffit d'écrire:
$${_n}\tilde{\mathcal{A}} = g_*({_n}\tilde{A}) = g_*({_n}A^t \otimes \mathbb{Z}/n\mathbb{Z}(d-1)) = g_*({_n}A^t) \otimes \mathbb{Z}/n\mathbb{Z}(d-1) = {_n}\mathcal{A}^*_d\otimes \mathbb{Z}/n\mathbb{Z}(d-1).$$
\end{proof}

\begin{proposition}\label{prelQp}
Soit $\ell$ un nombre premier ne divisant pas $|F_{d-1}|$ (mais pouvant être éventuellement égal à $p$).
\item[(i)] Les groupes $A(k^{nr})$ et $H^0(k^{nr},\tilde{A})$ sont $\ell$-divisibles.
\item[(ii)] Il existe un morphisme fonctoriel injectif $$(T_{\ell}H^1(\mathcal{O}_k,\mathcal{A}_d))^D\rightarrow (\varprojlim_r H^1(k^{nr}/k,{_{\ell^r}}A(k^{nr})))^D.$$
\end{proposition}

\begin{proof}
\begin{itemize}
\item[(i)] On prouve que $A(k^{nr})$ et $A^t(k^{nr})$ sont $\ell$-divisibles exactement de la même manière que dans \ref{prelC}. De plus, on remarque que:
 $$H^0(k^{nr},\tilde{A})= \varinjlim_n H^0(k^{nr},{_n}\tilde{A}) \cong \varinjlim_n H^0(k^{nr},{_n}A^t) = A^t(k^{nr})_{\text{tors}}.$$
On en déduit que $H^0(k^{nr},\tilde{A})$ est $\ell$-divisible.
\item[(ii)] La preuve est analogue à celle de \ref{prelC}.
\end{itemize}
\end{proof}

Comme dans \ref{HnrsCbis}, nous sommes maintenant en mesure d'introduire la définition suivante:

\begin{definition}\label{HnrsQp}
Soit $\ell$ un nombre premier ne divisant pas $|F_{d-1}|$ (mais éventuellement égal à $p$). On appelle \textbf{$\ell$-groupe de cohomologie non ramifiée symétrisé de $\tilde{A}$} le groupe:
$$H^d_{nrs}(k,\tilde{A},\ell) := (\iota_{\ell} \circ \varphi)^{-1}((T_{\ell}H^1(\mathcal{O}_k,\mathcal{A}_d))^D) \subseteq H^d(k,\tilde{A})\{\ell\}$$
où $\varphi: H^d(k,\tilde{A}) \rightarrow H^{d-1}(k^{nr}/k,H^1(k^{nr},\tilde{A}))$ désigne le morphisme induit par la suite spectrale $H^{r}(k^{nr}/k,H^s(k^{nr},\tilde{A})) \Rightarrow H^{r+s}(k,\tilde{A})$ et $\iota_{\ell}$ l'isomorphisme composé: 
\begin{align*}
 H^{d-1}(k^{nr}/k,H^1(k^{nr},\tilde{A}))\{\ell\} &\xrightarrow{\sim} \varinjlim_r H^{d-1}(k^{nr}/k,{_{\ell^r}}H^1(k^{nr},\tilde{A})) \\
&\xleftarrow{\sim} \varinjlim_r H^{d-1}(k^{nr}/k,H^1(k^{nr},{_{\ell^r}}\tilde{A})) \\
&\xrightarrow{\sim} \varinjlim_r H^{d-1}(k^{nr}/k,({_{\ell^r}}A(k^{nr})(1-d))^D) \\
&\xrightarrow{\sim} (\varprojlim_r H^1(k^{nr}/k,{_{\ell^r}}A(k^{nr})))^D.
\end{align*}  
On a alors une suite exacte:
$$0 \rightarrow \frac{H^d(\mathcal{O}_k,\tilde{\mathcal{A}})}{\delta (H^{d-2}(\mathcal{O}_k,R^1g_*\tilde{A}))}\{\ell\} \rightarrow H^d_{nrs}(k,\tilde{A},\ell) \rightarrow (T_{\ell}H^1(\mathcal{O}_k,\mathcal{A}_d))^D \rightarrow 0$$
où $\delta: H^{d-2}(\mathcal{O}_k,R^1g_*\tilde{A})\rightarrow H^d(\mathcal{O}_k,\tilde{\mathcal{A}})$ est le morphisme de bord provenant de la suite spectrale $H^{r}(\mathcal{O}_k,R^sg_*\tilde{A})\Rightarrow H^{r+s}(k,\tilde{A})$.
\end{definition}

\begin{remarque}
Soit $\ell$ un nombre premier \emph{différent de $p$} et ne divisant pas $|F_{d-1}|$. Supposons de plus que $\rho_0=\rho_1=...=\rho_{d-3}=0$. Dans ce contexte, on a $H^1(\mathcal{O}_k,\mathcal{A}_d)\{\ell\}\cong H^1(k^{nr}/k,A(k^{nr})) \cong H^1(k_{d-1},A_{d-1})\{\ell\} \cong H^1(k_{d-1},A_{d-1}^0)\{\ell\}$. Comme $H^1(k_{d-1},T_{d-1})\{\ell\}$ et $H^1(k_{d-1},B_{d-1})\{\ell\}$ sont finis (remarque \ref{rq1}), la suite exacte $0 \rightarrow U_{d-1} \times T_{d-1} \rightarrow A_{d-1}^0 \rightarrow B_{d-1} \rightarrow 0$ impose que $H^1(k_{d-1},A_{d-1}^0)\{\ell\}$ est fini. Cela montre la finitude de $H^1(\mathcal{O}_k,\mathcal{A}_{d})\{\ell\}$ et donc la nullité de $T_{\ell}H^1(\mathcal{O}_k,\mathcal{A}_d)$. Par conséquent: $$H^d_{nrs}(k,\tilde{A},\ell) = \text{Im}(H^d(\mathcal{O}_k,\tilde{\mathcal{A}})\rightarrow H^d(k,\tilde{A}))\{\ell\} = H^d_{nr}(k,\tilde{A})\{\ell\}$$
où $H^d_{nr}(k,\tilde{A}) := \text{Im}(H^d(\mathcal{O}_k,\tilde{\mathcal{A}})\rightarrow H^d(k,\tilde{A}))$. Ainsi, il est vraiment nécessaire de parler du groupe de cohomologie non ramifié symétrisé uniquement dans le cas $\ell=p$. Mais il est quand même utile d'introduire ce groupe quel que soit $\ell$ pour deux raisons: d'une part, dans le théorème qui suit, c'est avec le groupe de cohomologie non ramifié symétrisé qu'on identifie naturellement le noyau du morphisme de la dualité locale; d'autre part, cela permet de donner des énoncés vrais pour tout $\ell$.
\end{remarque}

\begin{theorem}\label{noyauQp}
Pour $\ell$ premier ne divisant pas $|F_{d-1}|$ (éventuellement égal à $p$), la partie $\ell$-primaire du noyau de $H^d(k,\tilde{A}) \rightarrow (H^0(k,A)^{\wedge})^D$ est $H^d_{nrs}(k,\tilde{A},\ell)$.
\end{theorem}

\begin{proof}
La preuve est analogue à celle de \ref{noyaubis}.
\end{proof}

Pour alléger les notations dans la section suivante, nous noterons:
$$H^d_{nrs}(k,\tilde{A}) := \bigoplus_{\ell \wedge |F_{d-1}|=1} H^d_{nrs}(k,\tilde{A},\ell).$$
C'est le groupe de torsion dont la partie $\ell$-primaire est $H^1_{nrs}(k,\tilde{A},\ell)$ si $\ell$ ne divise pas $|F_{d-1}|$, triviale sinon.

\section{\textsc{Variétés abéliennes sur $\mathbb{Q}_p((t_2))...((t_d))(u)$}}\label{globQp}

On suppose maintenant que $d \geq 1$ et que $k=k'((t_2))...((t_d))$ avec $k'$ corps $p$-adique, $p$ étant un nombre premier fixé. Soient $A$ une variété abélienne sur $K=k(X)$ et $\tilde{A}=\varinjlim_n \underline{\text{Ext}}^1_K(A,\mathbb{Z}/n\mathbb{Z}(d+1))=\varinjlim_n {_n}A^t \otimes \mathbb{Z}/n\mathbb{Z}(d)$. Le but de ce paragraphe est d'établir, sous de bonnes hypothèses géométriques sur $A$, des théorèmes de dualité entre certains groupes de Tate-Shafarevich de $A$ et $\tilde{A}$. \\
Pour chaque $v \in X^{(1)}$, on adopte des notations analogues à celles de la section \ref{notations} pour la variété abélienne $A_v=A \times_K K_v$ sur le corps $d+1$-local $K_v$ (on prendra garde au fait que $K_v$ n'est pas $d$-local). Ainsi, on introduit les schémas en groupes $\mathcal{A}_{v,i}, A_{v,i},F_{v,i},U_{v,i},T_{v,i},B_{v,i}$ pour $i \in \{0,1,...,d+1\}$.

\subsection{Approche sans les groupes de cohomologie non ramifiée symétrisés}

On se donne un entier $r_0 \in \{0,1,...,d+1\}$ et un nombre premier $\ell$ différent de $p=\text{Car}(k_0)$ et on fait les hypothèses suivantes:
\begin{hypo}\label{H1}
\begin{minipage}[t]{12.72cm}
\begin{itemize}
\item[$\bullet$] si $r_0 = d$, alors les tores $T_{v,0},...,T_{v,d}$ sont anisotropes pour toute place $v \in X^{(1)}$;
\item[$\bullet$] si $r_0 = d+1$, alors les tores $T_{v,0},...,T_{v,d-1}$ sont anisotropes pour toute place $v \in X^{(1)}$.
\end{itemize}
\end{minipage}
\end{hypo}

\begin{remarque}
Ces hypothèses sont assez restrictives puisqu'elles portent sur toutes les places $v \in X^{(1)}$ et pas uniquement sur les places de mauvaise réduction (voir remarque \ref{restr}). Dans le paragraphe suivant, on pourra s'affranchir de ces hypothèses grâce aux groupes de cohomologie non ramifiée symétrisés.
\end{remarque}

Soit maintenant $U$ un ouvert non vide de $X$ sur lequel $A$ a bonne réduction, de sorte que le modèle de Néron $\mathcal{A}$ de $A$ sur $U$ est un schéma abélien. Soit $\tilde{\mathcal{A}}=\varinjlim_n \underline{\text{Ext}}^1_K(\mathcal{A},\mathbb{Z}/n\mathbb{Z}(d+1))$. En procédant comme dans le lemme \ref{torsion dual}, on montre pour chaque $n>0$ que:
$${_n}\tilde{\mathcal{A}}=\underline{\text{Ext}}^1_U(\mathcal{A},\mathbb{Z}/n\mathbb{Z}(d+1))=\underline{\text{Hom}}_U({_n}\mathcal{A},\mathbb{Q}/\mathbb{Z}(d+1))= {_n}\mathcal{A}^t \otimes \mathbb{Z}/n\mathbb{Z}(d)$$ et que la multiplication par $n$ sur $\tilde{\mathcal{A}}$ est surjective.

\begin{lemma}
\begin{itemize}
\item[(i)] Pour $r>0$, le groupe $H^r(U,\mathcal{A})$ est de torsion de type cofini.
\item[(ii)] Pour $r>1$, le groupe $H^r_c(U,\mathcal{A})$ est de torsion de type cofini.
\end{itemize}
\end{lemma}

\begin{proof}
La preuve est tout à fait analogue à celle de \ref{cofC}.
\end{proof}

\begin{lemma} \label{exact'}
Pour $r \geq 0$, il existe des suites exactes:
$$ 0 \rightarrow H^r(U,\mathcal{A}) \otimes_{\mathbb{Z}} (\mathbb{Q}/\mathbb{Z})\{\ell\}\rightarrow H^{r+1}(U,\mathcal{A}\{\ell\}) \rightarrow H^{r+1}(U,\mathcal{A})\{\ell\} \rightarrow 0 ,$$
$$0 \rightarrow H^r_c(U,\tilde{\mathcal{A}})^{(\ell)}\rightarrow H^{r+1}_c(U,T_{\ell}\tilde{\mathcal{A}}) \rightarrow  T_{\ell}H^{r+1}_c(U,\tilde{\mathcal{A}}) \rightarrow 0.$$
Ici, $H^{r+1}(U,\mathcal{A}\{\ell\})$ et $H^{r+1}_c(U,T_{\ell}\tilde{\mathcal{A}})$ désignent $\varinjlim_n H^{r+1}(U,{_{\ell^n}}\mathcal{A})$ et $\varprojlim_n H^2_c(U,{_{\ell^n}}\tilde{\mathcal{A}})$ respectivement.
\end{lemma}

\begin{proof}
La preuve est identique à celle de \ref{exact}.
\end{proof}

\begin{lemma} \label{acc'}
Pour chaque $r \geq 0$, il existe un accouplement canonique:
$$ H^r(U,\mathcal{A}\{\ell\}) \times H^{d+3-r}_c(U,T_{\ell}\tilde{\mathcal{A}}) \rightarrow \mathbb{Q}/\mathbb{Z}$$
qui est non dégénéré.
\end{lemma}

\begin{proof}
La preuve est analogue à celle de \ref{acc}.
\end{proof}

Comme $\underline{\text{Hom}}_U(\mathcal{A},\mathbb{Z}/n\mathbb{Z}(d+1))=0$ et $H^{d+3}_c(U,\mathbb{Q}/\mathbb{Z}(d+1))\cong \mathbb{Q}/\mathbb{Z}$ (lemme 1.3 de \cite{Izq1}), on a pour chaque $n>0$, un accouplement:
$$ H^r(U,{_n}\tilde{\mathcal{A}}) \times H^{d+2-r}_c(U,\mathcal{A}) \rightarrow H^{d+3}_c(U,\mathbb{Z}/n\mathbb{Z}(d+1)) \rightarrow H^{d+3}_c(U,\mathbb{Q}/\mathbb{Z}(d+1))\cong \mathbb{Q}/\mathbb{Z}.$$
En passant à la limite inductive sur $n$, on obtient un accouplement:
$$ H^r(U,\tilde{\mathcal{A}}) \times H^{d+2-r}_c(U,\mathcal{A}) \rightarrow \mathbb{Q}/\mathbb{Z}.$$
Posons maintenant, pour chaque $r\geq 0$:
$$D^r(U,\tilde{\mathcal{A}}) = \text{Im}(H^r_c(U,\tilde{\mathcal{A}}) \rightarrow H^r(K,\tilde{A})).$$
Ce groupe est de torsion de type cofini.

\begin{lemma}\label{Sha'}
Soit $r \geq 0$. Il existe $V_0$ un ouvert non vide de $U$ tel que, pour tout ouvert $V$ de $V_0$, le morphisme $H^r(V,\tilde{\mathcal{A}}) \rightarrow H^r(K,\tilde{A})$ induit un isomorphisme: 
$$ D^r(V,\tilde{\mathcal{A}}) \cong \Sha^r(K,\tilde{\mathcal{A}}).$$
\end{lemma}

\begin{proof}
On remarque que, si $V \subseteq V'$ sont des ouverts dans $U$, alors $D^r(V,\tilde{\mathcal{A}}) \subseteq D^r(V',\tilde{\mathcal{A}})$. Comme $D^r(U,\tilde{\mathcal{A}})$ est de torsion de type cofini, il existe $V_0$ un ouvert non vide de $U$ tel que, pour tout $V$ contenu dans $V_0$, on a: $$D^r(V,\tilde{\mathcal{A}}) = D^r(V_0,\tilde{\mathcal{A}}).$$
Un tel $V_0$ convient.
\end{proof}

Afin d'établir un théorème de dualité pour le groupe de Tate-Shafarevich, il convient donc d'établir un théorème de dualité pour le groupe $D^r(U,\tilde{\mathcal{A}})$:

\begin{proposition}\label{d1'}
On suppose (H \ref{H1}). On pose $$D^{d+2-r_0}_{sh}(U,\mathcal{A}) = \text{Ker}(H^{d+2-r_0}(U,\mathcal{A}) \rightarrow \prod_{v \in X^{(1)}} H^{d+2-r_0}(K_v,A)).$$ Il existe alors un accouplement canonique:
$$ \overline{D^{r_0}(U,\tilde{\mathcal{A}})}\{\ell\} \times \overline{D^{d+2-r_0}_{sh}(U,\mathcal{A})}\{\ell\} \rightarrow  \mathbb{Q}/\mathbb{Z}$$
qui est non dégénéré.
\end{proposition}

Il convient d'établir préalablement le lemme suivant:

\begin{lemma}
La suite:
$$ \bigoplus_{v \in X^{(1)}} H^{r_0-1}(K_v,\tilde{A})^{(\ell)} \rightarrow H^{r_0}_c(U,\tilde{\mathcal{A}})^{(\ell)} \rightarrow  D^{r_0}(U,\tilde{\mathcal{A}})^{(\ell)} \rightarrow 0$$
est exacte.
\end{lemma}

\begin{proof} 
On a $\tilde{\mathcal{A}}= \varinjlim_n {_n}\tilde{\mathcal{A}}$, et donc, en utilisant la proposition 2.4 de \cite{Izq1}, on a une suite exacte:
$$ \bigoplus_{v \in X^{(1)}} H^{r_0-1}(K_v,\tilde{A}) \rightarrow H^{r_0}_c(U,\tilde{\mathcal{A}}) \rightarrow  D^{r_0}(U,\tilde{\mathcal{A}}) \rightarrow 0.$$
Cela étant établi, la preuve est analogue à celle du lemme \ref{7aux}.
\end{proof}

\begin{proof} \textit{(De la proposition \ref{d1'})}\\
Rappelons que, d'après le lemme \ref{acc'}, nous disposons d'un accouplement non dégénéré:
$$ H^{d+2-r_0}(U,\mathcal{A}\{\ell\}) \times H^{r_0+1}_c(U,T_{\ell}\tilde{\mathcal{A}})) \rightarrow \mathbb{Q}/\mathbb{Z},$$
d'où un isomorphisme $H^{d+2-r_0}(U,\mathcal{A}\{\ell\}) \rightarrow (H^{r_0+1}_c(U,T_{\ell}\tilde{\mathcal{A}}))^D$. On dispose aussi d'un accouplement:
$$H^{d+2-r_0}(U,\mathcal{A})  \times H^{r_0}_c(U,\tilde{\mathcal{A}}) \rightarrow \mathbb{Q}/\mathbb{Z}$$
qui induit un accouplement:
$$H^{d+2-r_0}(U,\mathcal{A})\{\ell\}   \times H^{r_0}_c(U,\tilde{\mathcal{A}})^{(\ell)} \rightarrow \mathbb{Q}/\mathbb{Z}.$$
Ainsi on obtient un diagramme commutatif à lignes exactes (que l'on appellera diagramme (1)):\\
\centerline{\xymatrix{
 0 \ar[r]& H^{d+1-r_0}(U,\mathcal{A}) \otimes_{\mathbb{Z}} (\mathbb{Q}/\mathbb{Z})\{\ell\} \ar[r]\ar[d]& H^{d+2-r_0}(U,\mathcal{A}\{\ell\}) \ar[r]\ar[d]^{\cong}& H^{d+2-r_0}(U,\mathcal{A})\{\ell\} \ar[r]\ar[d]& 0\\
 0 \ar[r]& (T_{\ell}H^{r_0+1}_c(U,\tilde{\mathcal{A}}))^D\ar[r]& (H^{r_0+1}_c(U,T_{\ell}\tilde{\mathcal{A}}))^D \ar[r]&  (H^{r_0}_c(U,\tilde{\mathcal{A}})^{(\ell)})^D \ar[r]& 0
}} 
De plus, nous disposons aussi d'un autre diagramme diagramme commutatif à lignes exactes (que l'on appellera diagramme (2)):\\
\centerline{\xymatrix{
0 \ar[r] & D^{d+2-r_0}_{sh}(U,\mathcal{A})\{\ell\}  \ar[d]\ar[r] & H^{d+2-r_0}(U,\mathcal{A})\{\ell\} \ar[d] \ar[r] & \prod_{v \in X^{(1)}} H^{d+2-r_0}(K_v,A)\{\ell\} \ar[d] \\
0 \ar[r] & (D^{r_0}(U,\tilde{\mathcal{A}})^{(\ell)})^D   \ar[r] & (H^{r_0}_c(U,\tilde{\mathcal{A}})^{(\ell)})^D  \ar[r] & \prod_{v \in X^{(1)}} (H^{r_0-1}(K_v,\tilde{A})^{(\ell)})^D 
}}
En procédant exactement comme dans \ref{d1} et en utilisant la section \ref{Qp} ainsi que l'hypothèse (H \ref{H1}), on montre alors que l'on a un accouplement non dégénéré:
$$ \overline{D^{r_0}(U,\tilde{\mathcal{A}})}\{\ell\} \times \overline{D^{d+2-r_0}_{sh}(U,\mathcal{A})}\{\ell\} \rightarrow  \mathbb{Q}/\mathbb{Z}.$$
\end{proof}

Nous sommes maintenant en mesure de conclure:

\begin{theorem} \label{th1} On rappelle que $k=k'((t_2))...((t_d))$ avec $k'$ un corps $p$-adique et que $K=k(X)$ est le corps des fonctions de la courbe $X$. Soit $A$ une variété abélienne sur $K$. On suppose (H \ref{H1}). Alors il existe un accouplement non dégénéré de groupes de torsion:
$$\overline{\Sha^{r_0}(K,\tilde{A})}_{\text{non}-p} \times \overline{\Sha^{d+2-r_0}(K,A)}_{\text{non}-p} \rightarrow \mathbb{Q}/\mathbb{Z}.$$
De plus, $\Sha^{r_0}(K,\tilde{A})$ et $\Sha^{d+2-r_0}(K,A)$ sont de torsion de type cofini.
\end{theorem}

\begin{proof}
Cela découle immédiatement de la proposition \ref{d1'} et du lemme \ref{Sha'} en passant à la limite sur $U$.
\end{proof}

\begin{remarque}\label{restr}
\begin{itemize}
\item[$\bullet$] Les hypothèses de (H \ref{H1}) concernent toutes les places de $X^{(1)}$. On ne peut pas restreindre ces hypothèses aux places de mauvaise réduction de $A$ puisqu'on ne sait pas si l'ouvert $V_0$ du lemme \ref{Sha'}  peut être choisi égal à $U$. Ce problème vient en particulier du fait que le corps $K$ est de dimension cohomologique 3 et que, même si $v \in X^{(1)}$ est une place de bonne réduction, le groupe $H^1(\mathcal{O}_v, \mathcal{A})$ peut être non nul!
\item[$\bullet$] Même si le théorème est une dualité modulo divisibles, dans la preuve, on a besoin d'une dualité locale qui n'est pas modulo divisibles. C'est pourquoi nous sommes amenés à faire les hypothèses (H \ref{H1}).
\end{itemize}
\end{remarque}

\begin{remarque}
Toute variété abélienne sur $K$ vérifie les hypothèses du théorème lorsque $r_0=0$. Dans ce cas, le théorème affirme que la partie divisible de  $\Sha^{d+2}(K,A)$ est $p$-primaire.
\end{remarque}

\begin{example} 
Dans le cas où $k$ est $p$-adique et $K=k(u)$, si on se donne $f(u) \in K^{\times}$, la courbe elliptique d'équation $y^2=x^3+f(u)$ vérifie les hypothèses du théorème pour $r_0 \in \{1,2\}$.
\end{example}

\subsection{Approche avec les groupes de cohomologie non ramifiée symétrisés}

Soient $\ell$ un nombre premier (éventuellement égal à $p$) et $U$ un ouvert non vide de $X$ sur lequel $A$ a bonne réduction. Faisons l'hypothèse suivante:

\begin{hypol}\label{HQp}
\begin{minipage}[t]{12.72cm}
\begin{itemize}
\item[$\bullet$] si $\ell \neq p$, pour chaque $v \in X \setminus U$, au moins l'une des deux affirmations suivantes est vérifiée:
\begin{itemize}
\item[$\circ$] $\ell$ ne divise pas $|F_{v,d}|$,
\item[$\circ$] les tores $T_{v,d-1},...,T_{v,0}$ sont anisotropes.
\end{itemize}
\item[$\bullet$] si $\ell = p$, pour chaque $v \in X \setminus U$, au moins l'une des deux affirmations suivantes est vérifiée:
\begin{itemize}
\item[$\circ$] $\ell$ ne divise pas $|F_{v,d}|$,
\item[$\circ$] les groupes algébriques $T_{v,d},...,T_{v,1}$ et $B_{v,1}$ sont triviaux.
\end{itemize}
\end{itemize}
\end{minipage}
\end{hypol}

\begin{remarque}
Cette hypothèse est nettement mois forte que l'hypothèse de la section précédente. Elle ne concerne que les places de mauvaise réduction et est vérifiée pour presque tout $\ell$.
\end{remarque}

On note $Z$ l'ensemble suivant:
\begin{itemize}
\item[$\bullet$] si $\ell \neq p$, alors $Z$ désigne l'ensemble des $v \in X^{(1)}$ tels que les tores $T_{v,d-1},...,T_{v,-1}$ sont anisotropes,
\item[$\bullet$] si $\ell = p$, alors $Z$ désigne l'ensemble des $v \in X^{(1)}$ tels que les groupes algébriques $T_{v,d},...,T_{v,1}$ et $B_{v,1}$ sont triviaux.
\end{itemize}
On introduit le groupe suivant:
\begin{multline*}
\Sha^{d+1}_{nrs}(\tilde{A}):= \text{Ker} \left(  H^{d+1}(K,\tilde{A}) \rightarrow \prod_{v \in Z} H^{d+1}(K_v,\tilde{A}) \times \prod_{v \in X^{(1)} \setminus Z} H^{d+1}(K_v,\tilde{A})/ H^{d+1}_{nrs}(K_v,\tilde{A}) \right).
\end{multline*}

En procédant exactement de la même manière que dans la section \ref{Cfonctions}, on peut établir le théorème et le corollaire suivants:

\begin{theorem}\label{cornrsQp}
On rappelle que $k=k'((t_2))...((t_d))$ avec $k'$ un corps $p$-adique et que $K=k(X)$ est le corps des fonctions de la courbe $X$. Soit $A$ une variété abélienne sur $K$. On suppose (H \ref{HQp})$_{\ell}$. Alors il existe un accouplement non dégénéré de groupes finis:
$$\overline{\Sha^{d+1}_{nrs}(\tilde{A})}\{\ell\} \times \overline{\Sha^1(A)}\{\ell\} \rightarrow \mathbb{Q}/\mathbb{Z}.$$
\end{theorem}

\begin{corollary}\label{corglobQp}
On rappelle que $k=k'((t_2))...((t_d))$ avec $k'$ un corps $p$-adique et que $K=k(X)$ est le corps des fonctions de la courbe $X$. Soit $A$ une variété abélienne sur $K$. On suppose (H \ref{HQp})$_{\ell}$ et on note $\tilde{i}: \Sha^{d+1}(\tilde{A}) \hookrightarrow \Sha^{d+1}_{nrs}(\tilde{A})$ l'injection canonique. Alors il existe un accouplement non dégénéré à gauche de groupes finis:
$$\Sha^{d+1}(\tilde{A})\{\ell\}/\tilde{i}^{-1}(\Sha^{d+1}_{nrs}(\tilde{A})\{\ell\}_{div}) \times \overline{\Sha^1(A)}\{\ell\} \rightarrow \mathbb{Q}/\mathbb{Z}.$$
\end{corollary}

\textbf{Question:} Quel est le noyau à droite dans l'accouplement précédent?

\begin{remarque}
Dans le cas où $k$ est un corps $p$-adique, en utilisant le paragraphe \ref{precisions}, il est possible de remplacer l'hypothèse (H \ref{HQp})$_{p}$ par l'hypothèse légèrement plus faible suivante:

\centerline{\begin{minipage}[t]{0.9\textwidth}
« pour chaque $v \in X \setminus U$, au moins l'une des deux affirmations suivantes est vérifiée:
\begin{itemize}
\item[$\circ$] $\ell$ ne divise pas $|F_{v,d}|$,
\item[$\circ$] la variété abélienne $B_{v,1}$ est triviale. »
\end{itemize}
\end{minipage}}
\end{remarque}

\section{\textsc{Quelques remarques sur la finitude des groupes de Tate-Shafarevich}}

Le but de cette section est de donner, pour $k = \mathbb{C}((t))$ ou $k=\mathbb{Q}_p$, des exemples de variétés abéliennes sur $K$ pour lesquelles on peut déterminer si le premier groupe de Tate-Shafarevich est fini ou pas. Pour ce faire, nous allons utiliser le théorème 3.1 de \cite{Tat66}, dont nous rappelons l'énoncé (adapté à notre situation):

\begin{theorem}\label{Tate} (théorème 3.1 de \cite{Tat66})\\
Soit $Y$ une surface régulière sur $k$ munie d'un morphisme propre $f:Y \rightarrow X$ à fibres de dimension 1. On suppose que les fibres géométriques de $f$ sont connexes, et que la fibre générique est lisse. Si $f$ admet une section, $\text{Br} \; X$ est un sous-groupe de $\text{Br} \; Y$ et on a un isomorphisme  $\Sha^1(K, J)\cong \text{Br} \; Y / \text{Br} \; X$ où $J$ désigne la jacobienne de la fibre générique de $f$.
\end{theorem}

\subsection{Cas où $k=\mathbb{C}((t))$}

On se place dans le cas où $k = \mathbb{C}((t))$. 

\subsubsection{Cas où $Y$ est un produit}\label{produit}

Soient $C$ une courbe projective lisse sur $k$ telle que $C(k) \neq \emptyset$ et $Y = C \times_k X$. On note $J_C$ (resp. $J_X$) la jacobienne de $C$ (resp. $X$) sur $k$. D'après le théorème \ref{Tate}, le groupe $\Sha^1(K, J_C \times_k K)$ est égal à $\text{Br} \; Y / \text{Br} \; X$. Par ailleurs, nous savons que $\text{Br}_1 \; Y = H^1(k,\text{Pic}\; Y_{\overline{k}})$. D'après la proposition 1.7 de \cite{SZ14}, le morphisme naturel $H^1(k,\text{Pic}\; C_{\overline{k}})\times H^1(k,\text{Pic}\; X_{\overline{k}}) \rightarrow H^1(k,\text{Pic}\; Y_{\overline{k}})$ a un noyau et un conoyau finis. Écrivons la suite exacte de modules galoisiens:
$$0 \rightarrow J_C(\overline{k}) \rightarrow \text{Pic} \; C_{\overline{k}} \rightarrow \mathbb{Z} \rightarrow 0.$$
$$0 \rightarrow J_X(\overline{k}) \rightarrow \text{Pic} \; X_{\overline{k}} \rightarrow \mathbb{Z} \rightarrow 0.$$
On en déduit une suite exacte de cohomologie:
$$\mathbb{Z} \rightarrow H^1(k, J_C) \rightarrow H^1(k, \text{Pic} \; C_{\overline{k}}) \rightarrow 0.$$
$$\mathbb{Z} \rightarrow H^1(k, J_X) \rightarrow H^1(k, \text{Pic} \; X_{\overline{k}}) \rightarrow 0.$$
Les noyaux des morphismes surjectifs $H^1(k, J_C) \rightarrow H^1(k, \text{Pic} \; C_{\overline{k}})$ et $H^1(k, J_X) \rightarrow H^1(k, \text{Pic} \; X_{\overline{k}})$ sont donc finis. On en déduit que les parties divisibles de $\text{Br}_1 \; Y / \text{Br} \; X \cong H^1(k,\text{Pic}\; Y_{\overline{k}}) / H^1(k,\text{Pic}\; X_{\overline{k}})$ et $H^1(k,J_C)$ sont égales. Par conséquent, si $J_C$ n'a pas très mauvaise réduction, alors $\Sha^1(K, J_C \times_k K)_{div}\neq 0$. La réciproque est vraie par exemple si $X$ est de genre 0, puisque dans ce cas, la partie divisible de $\text{Br} \; Y_{\overline{k}}$ est nulle d'après la section 2.9 de \cite{SZ14}.

\subsubsection{Cas où $Y$ est de dimension de Kodaira $-\infty$}

Soit $Y$ une surface projective lisse sur $k$ de dimension de Kodaira $-\infty$. Le cas où $Y$ est une fibration en coniques sur une courbe est inintéressant, puisque la jacobienne de la fibre générique est triviale.\\ Supposons donc que $Y$ est une surface de del Pezzo vérifiant les hypothèses de \ref{Tate}. On note $J$ la jacobienne de la fibre générique de $Y \rightarrow X$. Soit $L$ une extension finie de $k$ telle que $Y \times_k L$ est rationnelle. Alors $\text{Br}\; Y$ est fini, et donc $\Sha^1(L(X),J \times_k L(X))_{div}=0$. Par restriction-corestriction, on déduit que $\Sha^1(K,J)_{div}=0$.\\
Remarquons finalement qu'il existe bien des surfaces de del Pezzo $Y$ vérifiant les hypothèses de \ref{Tate}. Par exemple, il suffit de choisir $Y_0/\mathbb{C}$ l'éclatement de $\mathbb{P}^2_{\mathbb{C}}$ en les points base d'un pinceau de cubiques et $Y=Y_0 \times_{\mathbb{C}} k$ puisque, dans ce cas, $Y$ est une surface jacobienne sur $\mathbb{P}^1_k$.

\subsubsection{Cas où $Y$ est de dimension de Kodaira 0}

Dans ce paragraphe, nous allons étudier le cas où $Y$ est une surface projective lisse minimale sur $k$ de dimension de Kodaira 0. La classification de telles surfaces montre que $Y$ est un twist d'une surface abélienne, une surface bielliptique, une surface K3 ou une surface d'Enriques. \\

\textbf{Surfaces abéliennes.} Supposons que $X$ soit une courbe elliptique et soit $Y$ le produit de $X$ par une courbe elliptique $E$. Dans ce cas, $Y$ est une surface abélienne fibrée au-dessus de $X$. Si $E$ n'a pas très mauvaise réduction, alors, d'après l'étude menée en \ref{produit}, $\Sha^1(K,E\times_k K)_{div} \neq 0$. Le cas où $E$ a très mauvaise réduction est plus difficile, puisque $\text{Br}_1 \; Y/\text{Br}\; X$ est fini et il faut donc s'intéresser au groupe de Brauer transcendant $\text{Im}(\text{Br} \; Y \rightarrow \text{Br}(Y \times_k \overline{k}))$.\\

\textbf{Surfaces bielliptiques.} Soit $Y=(E_1 \times_k E_2)/G$ une surface bielliptique, avec $E_1$ et $E_2$ deux courbes elliptiques et $G$ un sous-groupe fini de $E_1$ agissant sur $E_2$ de sorte que $E_2/G \cong \mathbb{P}^1_k$. Le morphisme $\pi : Y \rightarrow E_1/G$ est alors une fibration elliptique isotriviale, de fibre $E_2$. On prend $X=E_1/G$ et on suppose que la fibration a une section. Comme le genre géométrique de $Y$ est nul et $\text{Alb} \; Y = E_1/G$, on déduit que $\text{Br} (Y \times_k \overline{k})$ et $\text{Br}_1 \; Y/\text{Br}\; X$ sont finis. Cela montre que $\text{Br}\; Y/\text{Br}\; X$ est fini, et il en est donc de même de $\Sha^1(K,J)$ où $J$ désigne la jacobienne de la fibre générique de $\pi$. \\

\textbf{Surfaces K3.} Supposons que $X=\mathbb{P}^1_k$. Soient $Y_0$ une surface K3 sur $\mathbb{C}$ telle que $Y=Y_0\times_{\mathbb{C}} k$ est une surface elliptique sur $X$ avec une section. On note $\overline{Y}=Y \times_k \overline{k}$ et $\rho =\text{rg}(NS(Y_0))= \text{rg}(NS(\overline{Y}))$. Comme la variété d'Albanese de $Y$ est triviale, le groupe $\text{Br}_1 \; Y$ est fini. Concernant le groupe de Brauer transcendant de $Y$, comme $\text{Br} \; Y_0 \cong \text{Br} \; \overline{Y}$, on a $\text{Br} \; \overline{Y} = \text{Im}(\text{Br} \; Y \rightarrow \text{Br} \; \overline{Y})$. Or $(\text{Br} \; \overline{Y})_{div}\cong (\mathbb{Q}/\mathbb{Z})^{22-\rho}$. Étant donné que $\rho \leq 20$, on a $\text{Im}(\text{Br} \; Y \rightarrow \text{Br} \; \overline{Y})_{div}\neq 0$. Le théorème \ref{Tate} permet alors de conclure que $\Sha^1(K,J)_{div} \neq 0$ où $J$ est la jacobienne de la fibre générique de $Y \rightarrow X$. On remarquera que dans cette situation, la non nullité de $\Sha^1(K,J)_{div}$ n'est pas expliquée par le groupe de Brauer algébrique de $Y$.\\
Reste à rappeler qu'il existe bel et bien des surfaces K3 sur $\mathbb{C}$ qui sont des surfaces jacobiennes: 
\begin{itemize}
\item[$\bullet$] toutes les surfaces K3 avec $\rho \geq 13$ sont jacobiennes (lemme 12.22 de \cite{SS10});
\item[$\bullet$] pour $\rho<13$, dans la section 3.2 de \cite{HS11}, Hulek et Schütt construisent une famille de surfaces K3 qui sont des surfaces jacobiennes et qui vérifient $\rho \geq 10$; 
\item[$\bullet$] d'après \cite{CD89}, la surface jacobienne d'une surface K3 elliptique est une surface K3 de même rang de Picard, et toute surface K3 avec $\rho \geq 5$ est elliptique.\end{itemize}
~\\
\textbf{Surfaces d'Enriques.} Si $Y$ est une surface d'Enriques, c'est toujours une surface elliptique, mais elle ne possède jamais de section, ce qui ne permet donc pas d'appliquer \ref{Tate}.

\begin{remarque}
Si $Y$ est une surface projective lisse sur $k$ de dimension de Kodaira 1, alors $Y$ est automotiquement une surface elliptique, mais pas forcément jacobienne. Si $Y$ est de type général, alors $Y$ n'est pas une surface elliptique.
\end{remarque}

\subsection{Cas où $k=\mathbb{Q}_p$}

Soient $p$ et $\ell$ des nombres premiers distincts. On se place dans le cas où $k$ un corps $p$-adique. Soient $C$ une courbe projective lisse sur $k$ telle que $C(k) \neq \emptyset$ et $Y = C \times_k X$. On note $J_C$ (resp. $J_X$) la jacobienne de $C$ (resp. $X$) sur $k$. Comme dans le paragraphe précédent, on montre que la partie divisible de $(\text{Br}_1 \; Y / \text{Br} \; X)\{\ell\}$ est toujours triviale. En particulier, si $X$ est de genre 0, alors $\Sha^1(K, J_C \times_k K)\{\ell\}_{div}= 0$ (et ce résultat reste vrai si $C(k) = \emptyset$ par un argument de restriction-corestriction). Par contre, si $J_C \neq 0$, alors $\Sha^1(K,J_C \times_k K)\{p\}_{div} \neq 0$.\\

\textbf{Question:} Si $X$ est de genre 0, est-ce que toute jacobienne $J$ sur $K$ vérifie $\Sha^1(K,J)\{\ell\}_{div}=0$?

\section*{\textsc{Annexe: Le faisceau $\tilde{A}$}}\label{annexeva}

En tenant compte de la formule de Barsotti-Weil, il serait naturel de considérer $\text{\underline{Ext}}^1(A,\mathbb{Q}/\mathbb{Z}(d))$ au lieu de $\tilde{A}$ dans les sections \ref{Qp} et \ref{globQp}. Il se trouve en fait que ces deux faisceaux coïncident. Pour le voir, il faut utiliser certains travaux de Breen:

\begin{proposition}\label{breen}
Soient $l$ un corps de caractéristique nulle et $X$ un $l$-schéma séparé. Soit $\mathcal{A}$ un schéma abélien sur $X$. Soient $r$ un entier différent de 1, $i$ un entier quelconque et $n$ un entier naturel non nul. 
\begin{itemize}
\item[(i)] On a: $$\underline{\text{Ext}}^r_X(\mathcal{A},\mathbb{Q}/\mathbb{Z}(i))=\underline{\text{Ext}}^r_X(\mathcal{A},\mathbb{Z}/n\mathbb{Z}(i))=0.$$
\item[(ii)] Le faisceau $\underline{\text{Ext}}^1_X(\mathcal{A},\mathbb{Q}/\mathbb{Z}(i))$ est de torsion.
\item[(iii)] On a: $$\underline{\text{Ext}}^1_X(\mathcal{A},\mathbb{Q}/\mathbb{Z}(i)) = \varinjlim_n \underline{\text{Ext}}^1_X(\mathcal{A},\mathbb{Z}/n\mathbb{Z}(i)).$$
\end{itemize}
\end{proposition}

\begin{proof}
\begin{itemize}
\item[(i)]
\begin{itemize}
\item[$\bullet$] Montrons d'abord que le faisceau $\underline{\text{Ext}}^r_X(\mathcal{A},\mathbb{Z}/n\mathbb{Z}(i))$ est nul.
\begin{itemize}
\item[$\bullet$] Pour $r=0$, la multiplication par $n$ est injective sur $\underline{\text{Hom}}_X(\mathcal{A},\mathbb{Z}/n\mathbb{Z}(i))$. Ce faisceau étant de $n$-torsion, il est nul.
\item[$\bullet$] Supposons $r \geq 2$. On a une suite exacte:
$$\underline{\text{Ext}}^r_X(\mathcal{A},\mathbb{Z}/n\mathbb{Z}(i)) \rightarrow \underline{\text{Ext}}^r_X(\mathcal{A},\mathbb{Z}/n\mathbb{Z}(i)) \rightarrow \underline{\text{Ext}}^r_X({_n}\mathcal{A},\mathbb{Z}/n\mathbb{Z}(i)).$$
Or $\underline{\text{Ext}}^r_X({_n}\mathcal{A},\mathbb{Z}/n\mathbb{Z}(i)) = 0$. Donc le faisceau $\underline{\text{Ext}}^r_X({_n}\mathcal{A},\mathbb{Z}/n\mathbb{Z}(i))$ est $n$-divisible. Étant de $n$-torsion, il est nul.
 \end{itemize}
 On remarque alors que la nullité de $\underline{\text{Ext}}^r_X(\mathcal{A},\mathbb{Z}/n\mathbb{Z}(i))$ implique que le faisceau $\underline{\text{Ext}}^r_X(\mathcal{A},\mathbb{Q}/\mathbb{Z}(i))$ est sans torsion pour $r\neq 1$. Or, comme $\mathbb{Q}/\mathbb{Z}(i)$ est de torsion, en utilisant les résultats de \cite{Bre} (en particulier la méthode décrite dans le paragraphe 6 et le complexe 5.13), on déduit que le faisceau $\underline{\text{Ext}}^r_X(\mathcal{A},\mathbb{Q}/\mathbb{Z}(i))$ est de torsion. Cela achève la preuve.
 \end{itemize}
\item[(ii)] Cela découle de \cite{Bre} (en particulier la méthode décrite dans le paragraphe 6 et le complexe 5.13). 
\item[(iii)] D'après (i), on a une suite exacte:
$$0 \rightarrow \text{\underline{Ext}}^1(A,\mathbb{Z}/n\mathbb{Z}(d)) \rightarrow  \text{\underline{Ext}}^1(A,\mathbb{Q}/\mathbb{Z}(d)) \rightarrow  \text{\underline{Ext}}^1(A,\mathbb{Q}/\mathbb{Z}(d)) \rightarrow 0.$$
L'assertion (ii) permet alors de conclure.
\end{itemize}
\end{proof}

\nocite*


\begin{thebibliography}{CTGP04}

\footnotesize

\bibitem[Bre69]{Bre}
Lawrence Breen.
\newblock Extensions of abelian sheaves and {E}ilenberg-{M}ac{L}ane algebras.
\newblock {\em Invent. Math.}, 9:15--44, 1969.

\bibitem[CTH14]{CTH}
Jean-Louis Colliot-Th\'el\`ene et David Harari.
\newblock Dualité et principe local-global pour les tores sur une courbe
  au-dessus de $\mathbb{C}((t))$.
\newblock 2014.
\newblock Prépublication à paraître dans Proc. London Math. Soc.

\bibitem[CTS87]{CTS87}
Jean-Louis Colliot-Th\'el\`ene et Jean-Jacques Sansuc.
\newblock La descente sur les variétés rationnelles II.
\newblock {\em Duke Mathematical Journal}, 54:375--492, 1987.

\bibitem[CD89]{CD89}
François R. Cossec and Igor V. Dolgachev.
\newblock {\em Enriques surfaces I}.
\newblock Progress in Mathematics 76, Birkh\"auser, Boston, 1989.

\bibitem[FK88]{FK}
Eberhard Freitag and Reinhardt Kiehl.
\newblock {\em \'{E}tale cohomology and the {W}eil conjecture}.
\newblock Ergebnisse der Mathematik und ihrer Grenzgebiete (3) [Results in Mathematics and Related Areas (3)], Volume 13, Springer-Verlag, Berlin, 1988.

\bibitem[HSz13]{HS1}
{D}avid {H}arari and {T}am{\'a}s {S}zamuely.
\newblock {L}ocal-global principles for tori over {$p$}-adic function fields.
\newblock 2013.
\newblock A para\^itre dans \textit{Journal of Algebraic Geometry}.

\bibitem[HS11]{HS11}
Klaus Hulek and Matthias Sch\"utt.
\newblock Enriques surfaces and {J}acobian elliptic {K}3 surfaces.
\newblock {\em Mathematische Zeitschrift}, 268(3-4):1025--1056, 2011.

\bibitem[Izq14]{Izq1}
Diego Izquierdo.
\newblock Théorèmes de dualité pour les corps de fonctions sur des corps locaux supérieurs.
\newblock 2014.
\newblock Prépublication sur \texttt{http://www.eleves.ens.fr/home/izquierd/}.

\bibitem[Koy00]{Koy}
Yoshihiro Koya.
\newblock On a duality theorem of abelian varieties over higher dimensional local fields.
\newblock {\em Kodai Math. J.}, 2:297--308, 2000.

\bibitem[Mat55]{Mat}
Arthur Mattuck.
\newblock Abelian varieties over {$p$}-adic ground fields.
\newblock {\em Ann. of Math. (2)}, 62:92--119, 1955.

\bibitem[Mil86]{MilAV}
James~S. Milne.
\newblock {\em Abelian varieties}. 
\newblock In {\em Arithmetic geometry} ({S}torrs, {C}onn., 1984), p.103--150.
\newblock Springer, New York, 1986.

\bibitem[Mil06]{MilADT}
James~S. Milne.
\newblock {\em Arithmetic duality theorems}.
\newblock BookSurge, LLC, Charleston, SC, second edition, 2006.

\bibitem[NSW08]{CNF}
J{\"u}rgen Neukirch, Alexander Schmidt and Kay Wingberg.
\newblock {\em Cohomology of number fields}.
\newblock Second edition. Grundlehren der Mathematischen Wissenschaften [Fundamental Principles of Mathematical Sciences] vol. 323, Springer-Verlag, Berlin, 2008.

\bibitem[Ogg62]{Ogg}
Andrew P. Ogg.
\newblock Cohomology of abelian varieties over function fields.
\newblock {\em Ann. of Math. (2)}, 76:185--212, 1962.

\bibitem[Ono61]{Ono}
Takashi Ono.
\newblock Arithmetic of algebraic tori.
\newblock {\em Ann. of Math. (2)}, 74:101--139, 1961.

\bibitem[Oor66]{Oort}
Frans Oort.
\newblock {\em Commutative group schemes}.
\newblock Lecture Notes in Mathematics, Springer-Verlag, Berlin-New York, 1966.

\bibitem[SvH03]{SVH}
Claus Scheiderer and Joost van Hamel.
\newblock Cohomology of tori over {$p$}-adic curves.
\newblock {\em Math. Ann.}, 326(1):155--183, 2003.

\bibitem[SS10]{SS10}
Matthias Sch\"utt and Tetsuji Shioda.
\newblock Elliptic surfaces.
\newblock In {\em Algebraic geometry in East Asia—Seoul 2008}, Adv. Stud. Pure Math., 60:51–-160, Math. Soc. Japan, Tokyo, 2010.

\bibitem[Ser94]{Ser}
Jean-Pierre Serre.
\newblock {\em Cohomologie galoisienne}, volume~5 of {\em Lecture Notes in Mathematics}, fifth edition.
\newblock Springer-Verlag, Berlin, 1994.

\bibitem[SGA7]{SGA7}
{\em Groupes de monodromie en g\'eom\'etrie alg\'ebrique. {I}}, S{\'e}minaire de G{\'e}om{\'e}trie Alg{\'e}brique du Bois-Marie 1967--1969 (SGA 7 I), Dirig{\'e} par A. Grothendieck. Avec la collaboration de M. Raynaud et D. S. Rim.
\newblock Lecture Notes in Mathematics, Vol. 288, Springer-Verlag, 1972.

\bibitem[SZ14]{SZ14}
Alexei N. Skorobogatov and Yuri G. Zarhin.
\newblock The Brauer group and the Brauer–Manin set of products of varieties.
\newblock {\em J. Eur. Math. Soc.}, 16:749--768, 2014.

\bibitem[Tat58]{Tat58}
John Tate.
\newblock W{C}-groups over {$p$}-adic fields.
\newblock {\em S\'eminaire {B}ourbaki}, année 1957/1958, exposé 156.

\bibitem[Tat63]{Tat63}
John Tate.
\newblock Duality theorems in Galois cohomology over number fields.
\newblock In {\em Proc. {I}nternat. {C}ongr. {M}athematicians ({S}tockholm,
1962)}, pages 288--295. Inst. Mittag-Leffler, Djursholm, 1963.

\bibitem[Tat66]{Tat66}
John Tate.
\newblock On the conjectures of Birch and Swinnerton-Dyer and a geometric analog.
\newblock {\em S\'eminaire {B}ourbaki}, année 1964/1966, exposé 306.

\end{thebibliography}
\end{document}